\documentclass[11pt]{article}

\usepackage[margin=1in]{geometry}
\usepackage{setspace}
\usepackage[T1]{fontenc}
\usepackage[utf8]{inputenc}
\usepackage{lmodern}

\usepackage{amsmath, amssymb, amsthm}
\usepackage{hyperref}
\usepackage{graphicx}
\usepackage{enumitem}
\usepackage{authblk} 
\usepackage{xcolor}
\usepackage{enumitem}
\usepackage{dsfont}
\numberwithin{equation}{section} 

\newtheorem{theorem}{Theorem}
\newtheorem{lemma}[theorem]{Lemma}

\newtheorem{fact}[theorem]{Fact}
\newtheorem{definition}[theorem]{Definition}
\newtheorem{remark}[theorem]{Remark}

\title{Degrees of Freedom for Critical Random 2-SAT}
\author[, 1]{Andreas Basse-O'Connor, \thanks{ basse@math.au.dk}}
\author[, 1, 2]{Mette Skjøtt, \thanks{ metteskjott@math.au.dk}}
\affil[1]{Department of Mathematics, Aarhus University, Denmark}
\affil[2]{Kvantify ApS, DK-2300 Copenhagen S, Denmark}
\date{}

\newcommand{\inSAT}{\in\text{\texttt{SAT}}}

\newcommand{\rr}[1]{^{(#1)}}
\newcommand\eqdist{\mathrel{\stackrel{\makebox[0pt]{\mbox{\normalfont\tiny $\mathcal{D}$}}}{=}}}
\newcommand{\sign}{\text{sgn}}
\newcommand{\ind}{\perp\!\,\,\!\!\!\!\perp}
\begin{document}

\maketitle

\begin{abstract}
The random $k$-SAT problem serves as a model that represents the 'typical' $k$-SAT instances. This model is thought to undergo a phase transition as the clause density changes, and it is believed that the random $k$-SAT problem is primarily difficult to solve near this critical phase. In this paper, we introduce a weak formulation of degrees of freedom for random $k$-SAT problems and demonstrate that the critical random $2$-SAT problem has $\sqrt[3]{n}$ degrees of freedom. This quantity represents the maximum number of variables that can be assigned truth values without affecting the formula’s satisfiability. Notably, the value of $\sqrt[3]{n}$ differs significantly from the degrees of freedom in random $2$-SAT problems sampled below the satisfiability threshold, where the corresponding value equals $\sqrt{n}$. Thus, our result underscores the significant shift in structural properties and variable dependency as satisfiability problems approach criticality.
\end{abstract}

\newpage

\section{Introduction}

\subsection{Background and motivation}
The Boolean satisfiability problem (SAT) is a highly studied topic in computer science, notable for being the first problem proven to be NP-complete, see \cite{Cook71}. Its versatility extends beyond theoretical interest, with practical applications in areas like artificial intelligence, software verification, and optimization (see ~\cite{Marques_2008, Marco_2006, Knuth_2015}, and references therein). In recent years, SAT has also attracted significant attention in the fields of discrete probability and statistical physics. This interdisciplinary interest arises because SAT exhibits behaviors, such as phase transitions, making it a compelling subject for studying threshold behavior in combinatorial structures.

A SAT instance is a Boolean function that evaluates multiple Boolean variables and returns a single Boolean value. The function is typically expressed in conjunctive normal form (CNF), meaning it is a conjunction (\texttt{and}) of disjunctions (\texttt{or}) of literals. Each literal represents a variable or its negation. A formula in which every clause contains exactly $k$ literals is called a $k$-CNF formula. The following is an example of a $2$-CNF formula with four variables and five clauses:
\[\varphi(x)=
(x_1 \lor  x_2) \land (\neg x_2 \lor x_3) \land (\neg x_3\lor x_4) \land (\neg x_1 \lor \neg x_4)\wedge (x_2\lor  x_4),\quad (x\in  \{\texttt{true},\texttt{false}\}^4).
\]
The only assignment that makes the above formula evaluate to \texttt{true} is $(\texttt{false}, \texttt{true}, \texttt{true}, \texttt{true})$. The objective of the satisfiability problem is to determine whether such an assignment exists; if so, we write $\varphi \in \text{SAT}$. In the context of computational complexity theory, the $2$-SAT problem is NL-complete, meaning it can be solved non-deterministically with logarithmic storage space and is one of the most difficult problems within this class (see Thm.~16.3 in \cite{papadimitriou94}). Consequently, a deterministic algorithm that solves $2$-SAT using only logarithmic space would imply $L = NL$, which is a standing conjecture. For $k \geq 3$, the $k$-SAT problem is NP-complete, situating it at the core of the famous $P$ vs. $NP$ conjecture.

In practical applications, SAT instances are, in most cases, easily solvable, which appears to contradict the problem's computational hardness. This observation inspired the development of the random $k$-SAT model, designed to generate typical SAT instances, see \cite{Goldberg79, Cheeseman91, Kirkpatrick94, Gent94}. In this model, the number of input variables $n$, clauses $m$, and the clause size $k$ are fixed. Clauses are then sampled independently and uniformly from the $2^k\binom{n}{k}$ clauses with non-overlapping variables. This model is called the random $k$-SAT model, and the distribution is denoted $F_k(n,m)$. This model becomes particularly interesting when $n$ and $m$ grow large simultaneously. Specifically, by setting $m = \lfloor \alpha n \rfloor$, where $\alpha > 0$ represents the clause density, the random $k$-SAT problem is believed to undergo a phase transition: the asymptotic probability of satisfiability shifts from one to zero as $\alpha$ surpasses a critical threshold $\alpha_k$, that is for $k\geq 2$,
\begin{align} \label{eq satisfiability conjecture}
\lim_{n\rightarrow\infty}\mathbb{P}\big(F_k(n,\lfloor \alpha n\rfloor)\inSAT\big)=\begin{cases}
        1,\quad&\alpha<\alpha_k,\\
        0,\quad &\alpha>\alpha_k.
    \end{cases}
\end{align}
A random $k$-SAT problem that is satisfiable w.h.p.\ is referred to as under-constrained, while it is called over-constrained when it is unsatisfiable w.h.p. Furthermore, when a phase transition exists, problems at this critical value are referred to as being critical. 

As previously discussed, SAT problems are computationally challenging. Notably, it is near the expected phase transition of the random $k$-SAT model that the hardest instances are thought to arise, see \cite{Selmanetal1996}. Figure \ref{fig} displays how a spike in computational hardness appears when the clause density approaches the expected phase transition. This highlights why understanding the behavior of random $k$-SAT in this critical region is of substantial theoretical and practical importance. More broadly, the study of random structures near critical transitions is a significant and complex area of research. The prominence of this field is underscored by the fact that three Fields Medals have been awarded since 2006 for groundbreaking work on critical phenomena, with recipients including H. Duminil-Copin, S. Smirnov, and W. Werner.
\begin{figure}
\centering
\includegraphics[scale=0.1]{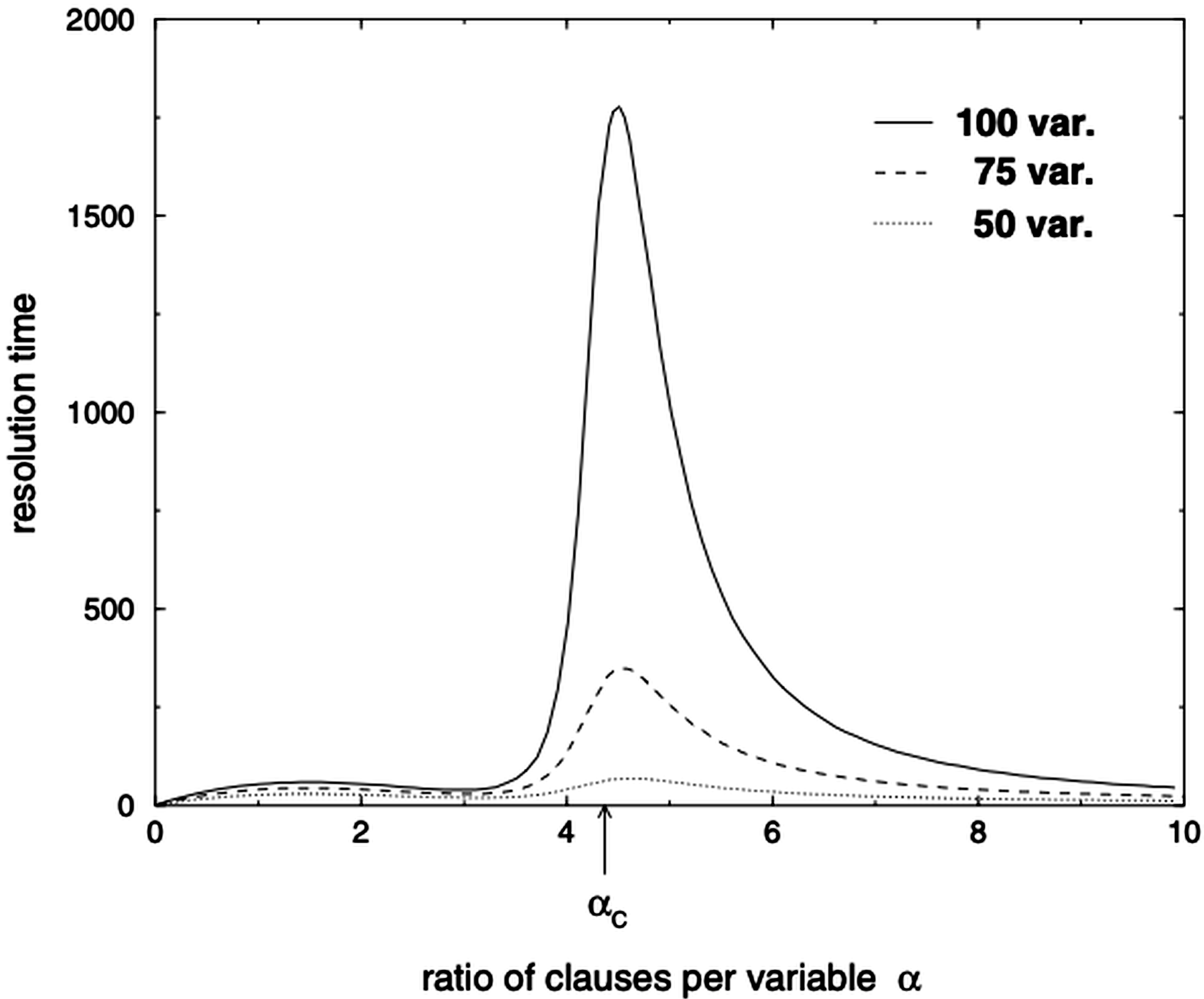}
\caption{Computational hardness of random $3$-SAT as a function of clause density. The y-axis displays the median resolution time of 10,000 instances solved using the DPLL algorithm. Credit: Fig. 8.2 in \cite{Biroli02}}
\label{fig}
\end{figure}

The phase transition phenomenon was in 1992 established for $k=2$ in the articles \cite{Goerdt_1996, Chvatal_1992, Fernandez_2001}, where the authors independently established that $\alpha_2 = 1$.  Recently, the sharp satisfiability conjecture (\ref{eq satisfiability conjecture}) has been affirmatively verified for all $k \geq k_0$, with $k_0$ being a large and unknown constant, see \cite{SSZ21}. The remaining cases of $k$ constitute an open problem.  In 1999, the result on random $2$-SAT was further refined in \cite{Bollobas_2001} as the rate of convergence was determined. Additionally, it was shown that the asymptotic probability of satisfiability of a random critical $2$-SAT problem is bounded away from both zero and one, though whether this probability converges remains an open question. Recent contributions to the random $2$-SAT model have focused on the under-constrained regime, where both the expected number of solutions and a central limit theorem for this quantity (see \cite{Noela2021, Doja-Ochlan2024}) has been established. Thus, while the phase transition of random $2$-SAT was proven many years ago, ongoing research continues to uncover new insights into the model, and several open questions remain unresolved. 

A recent study \cite{Basse_2024} examined variable interactions by analyzing the \emph{degrees of freedom} in under-constrained random $k$-SAT problems. This concept refers to the number of variables that can be fixed without impacting the formula's satisfiability. For under-constrained random $2$-SAT problems, where $\alpha < 1$, the degrees of freedom equal $n^{1/2}$, while in random $3$-SAT problems well below the phase transition ($\alpha<3.145$), the degrees of freedom equal $n^{2/3}$.

In this paper, we compute the degrees of freedom in \emph{critical} $2$-SAT problems. Our result shows that in this critical setting, the degrees of freedom decrease with a polynomial factor, scaling only as $n^{1/3}$.  This finding underscores the emergence of complex structures near the phase transition, where variable interdependencies become significantly more pronounced. Thus, our results highlight this marked shift in variable correlation as random SAT problems approach criticality.

\subsection{Main result} \label{section main result}

Consider a random $2$-CNF formula $\Phi$ sampled at the phase-transition point of the random $2$-SAT problem, where the asymptotic probability of satisfiability shifts from one to zero. We aim to determine how many input variables are free—that is, they can be assigned any value without effecting the asymptotic probability that the formula is satisfiable.

Let $\mathcal{L} \subseteq \pm [n] := \{-n, \dotsc, -1, 1, \dotsc, n\}$ be a set with $|\mathcal{L}| = f(n)$ elements, chosen such that if $\ell \in \mathcal{L}$, then $-\ell \notin \mathcal{L}$ (we say that $\mathcal{L}$ is \emph{consistent}). This set $\mathcal{L}$ dictates the variables being fixed, having $x_v = \texttt{true}$ when $v \in \mathcal{L}$ and $x_v = \texttt{false}$ when $-v \in \mathcal{L}$. Formally, let $\mathbb{B} = \{\texttt{true}, \texttt{false}\}$. For $x \in \mathbb{B}^n$, we define $x_\mathcal{L}\in\mathbb{B}^n$ as the vector with $(x_\mathcal{L})_v = \texttt{true}$ when $v \in \mathcal{L}$, $(x_\mathcal{L})_v = \texttt{false}$ when $-v \in \mathcal{L}$, and $(x_\mathcal{L})_v = x_v$ for all other entries. We then consider
\begin{equation} \label{eq Phi_L}
    \Phi_\mathcal{L}(x) = \Phi(x_\mathcal{L}).
\end{equation}
Note that $\Phi_\mathcal{L}$ denotes the mapping $\Phi$ with $f$ variables fixed to values specified by $\mathcal{L}$. Our goal is to identify the threshold value of $f$ that separates instances where $\Phi_\mathcal{L}$ remains solvable with positive probability from those where $\Phi_\mathcal{L}$ becomes unsatisfiable. To formalize this notion, we introduce the following definition, where we recall that $F_k(n, m)$ denotes a random $k$-CNF formula with $n$ variables and $m$ clauses.
\begin{definition} \label{def degrees of freedom}
    The random $k$-SAT problem with clause density $\alpha > 0$ is said to have $f_\star(n)$ degrees of freedom weakly if, for $\Phi \sim F_k(n, \lfloor \alpha n \rfloor)$, every consistent subset $\mathcal{L} \subseteq \pm [n]$ with $|\mathcal{L}| = f(n)$, and for all $\varepsilon > 0$, the following holds:
    \begin{enumerate}[label=(\arabic*)]
        \item Whenever $f= O(f_\star n^{-\varepsilon})$, then $$\liminf_{n\rightarrow\infty}\mathbb{P}(\Phi_\mathcal{L}\inSAT)=\liminf_{n\rightarrow\infty}\mathbb{P}(\Phi\inSAT)>0.$$
        \item Whenever $f =  \Omega(f_\star n^\varepsilon)$, then $$\lim_{n\rightarrow\infty}\mathbb{P}(\Phi_\mathcal{L}\inSAT)=0.$$
    \end{enumerate}
\end{definition}
Condition (1) states that fixing strictly fewer than $f_\star$ variables does not decrease the lower bound on the probability of satisfiability. On the other hand, condition (2) implies that when fixing strictly more than $f_\star$ variables, the problem becomes unsatisfied. This concept is a weaker form of the degrees of freedom notion introduced in \cite{Basse_2024}; specifically, having $f_\star$ degrees of freedom implies having $f_\star$ degrees of freedom weakly. Note that $f_\star$ is unique up to sub-polynomial factors, meaning that if both $f_\star$ and $g_\star$ are weak degrees of freedom, then for any $\varepsilon > 0$, we have $f_\star n^{-\varepsilon} \leq g_\star \leq f_\star n^{\varepsilon}$ for sufficiently large $n$. Our main result is the following:
\begin{theorem} \label{main theorem}
    The random critical $2$-SAT problem has $n^{1/3}$ degrees of freedom weakly. 
\end{theorem}
We recall that \emph{critical} refers to the situation with $\alpha=1$. Figure \ref{fig2} shows simulations indicating that, as $n$ increases, the curve representing the satisfiability of the random critical $2$-SAT problem as a function of the number of fixed variables becomes increasingly steep. Moreover, this steepening behavior points to a cutoff occurring at $n^{1/3}$. 
\begin{figure}[ht]
\centering
\includegraphics[scale=0.7]{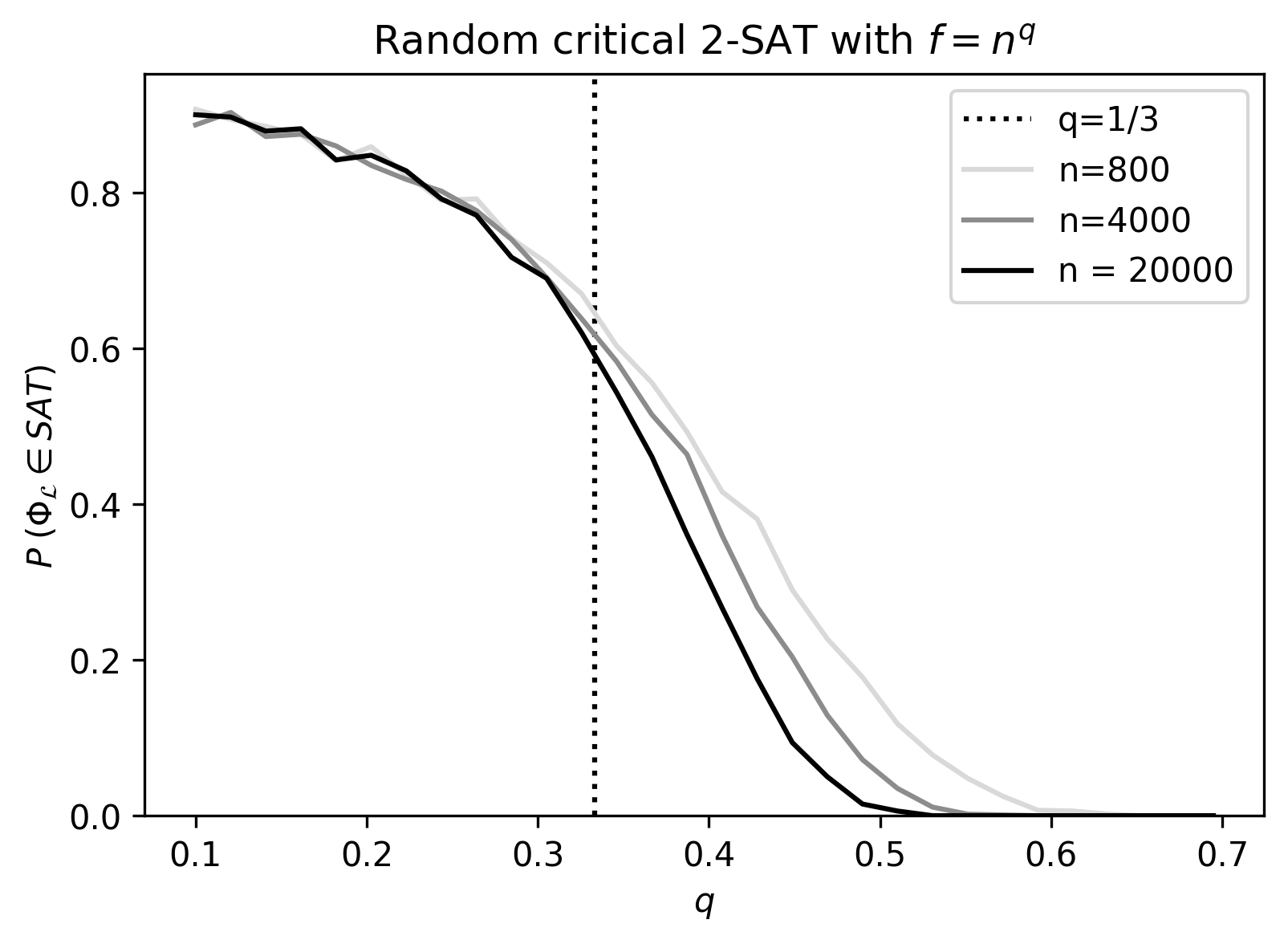}
\caption{Satisfiability of random critical $2$-SAT as a function of the number of fixed variables. The different curves represent a varying number of input variables. Each data point is comprised of $2,000$ simulations. The vertical dotted line indicates $q=1/3$.}
\label{fig2}
\end{figure}

\subsection{Related work}
In this section, we compare our results to related work, providing new insights and situating our findings within a broader context.
\begin{remark}
Theorem \ref{main theorem} allows us to compare the critical random $2$-SAT problem with general random $k$-SAT problems:
 \begin{itemize}
     \item The paper \cite{Basse_2024} established that under-constrained random $2$-SAT problems have $n^{1/2}$ degrees of freedom, which reveals a pronounced difference with the behavior observed at the critical phase-transition point. At this threshold, a dramatic reduction in degrees of freedom occurs, reflecting a fundamental shift in the underlying structure of the formula.  This is not surprising, 
as at this critical ratio, the system is on the "knife edge" between 
being satisfiable versus unsatisfiable, and therefore, long-range correlations between variables are expected to appear. To our knowledge, our result is one of the first to indicate this drastic change in variable dependence.
\item The paper \cite{Basse_2024} also examines random $3$-SAT problems, establishing that when $\alpha$ is significantly below the expected phase-transition threshold ($\alpha < 3.145$), the degrees of freedom are $n^{2/3}$. In comparison, our main theorem shows that the degrees of freedom in critical random $2$-SAT equal the square root of this amount, indicating a notable contrast in variable flexibility between the two cases.
 \end{itemize}
\end{remark}

The computational hardness of the satisfiability problem implies that finding solutions to challenging SAT formulas often requires traversing a substantial portion of the search tree, that is, assigning truth values to variables sequentially and backtracking when encountering contradictions. This approach forms the core of the DPLL algorithm, introduced in 1962 as one of the first SAT-solving algorithms, \cite{Davis62}. Decades later, in the 1990s, CDCL (Conflict-Driven Clause Learning) solvers transformed SAT solving, enabling the solution of instances with thousands or even millions of variables. Despite their modern enhancements, these solvers still rely on the simple procedure of assigning truth values and backtracking (see p. 62 in \cite{Knuth_2015}). The concept of degrees of freedom quantifies how deep one can navigate in the search tree before a contradiction arises when solving a random SAT problem. Moreover, the drastic change in degrees of freedom when comparing under-constrained problems with critical problems highlights why computational complexity intensifies near the satisfiability threshold. This also aligns with the observations in Figure \ref{fig}, which displayed the computation time of the DPLL algorithm when approaching criticality. 

Let again $\Phi\sim F_2(n,n)$, $\mathcal{L}\subseteq\pm[n]$ be consistent with $|\mathcal{L}|=f(n)$, and remember that fixing variables corresponds to shrinking the input space. Thus, it is clear
that $\{\Phi_{\mathcal{L}} \inSAT \} \subseteq \{\Phi \inSAT \}$. This along with our main theorem implies that whenever $f=O(n^{1/3-\varepsilon})$ for an $\varepsilon>0$ we have 
\begin{equation} \label{liminf<limsup<liminf<limsup}
    \liminf_{n \rightarrow \infty} \mathbb{P} \big(\Phi \inSAT \big) 
\leq \liminf_{n \rightarrow \infty} \mathbb{P} \big(\Phi_{\mathcal{L}} 
\inSAT \big)\leq \limsup_{n \rightarrow \infty} \mathbb{P} 
\big(\Phi_{\mathcal{L}} \inSAT \big) \leq \limsup_{n \rightarrow \infty} 
\mathbb{P} \big(\Phi \inSAT \big).
\end{equation}
Thus, if $\mathbb{P}(\Phi \inSAT)$ has a limit as $n\rightarrow\infty$, 
then $\mathbb{P}(\Phi_{\mathcal{L}} \inSAT)$ also has a limit, and these two limits coincide. In \cite{Bollobas_2001} it is shown that for all $\delta > 0$ sufficiently 
small, there exists a $c_\delta > 0$ such that if $\Phi_\alpha \sim F_2(n, \lfloor \alpha n\rfloor)$ with $\alpha \in [1 - c_\delta 
n^{-1/3}, 1 + c_\delta n^{-1/3}]$, then
\begin{equation} \label{eq interval of P(PhiinSAT)}
    \delta \leq \mathbb{P}(\Phi_\alpha \inSAT) \leq 1 - \delta.
\end{equation}
Moreover, this interval is the best possible 
in the sense that if a sufficiently large constant replaces $c_\delta$, 
the statement becomes false. Combining (\ref{liminf<limsup<liminf<limsup}) and (\ref{eq interval of P(PhiinSAT)}) we get that for $\delta > 0$ small 
enough
$$\delta \leq \liminf_{n \rightarrow \infty} \mathbb{P} 
\big(\Phi_{\mathcal{L}} \inSAT \big) \leq \limsup_{n \rightarrow \infty} 
\mathbb{P} \big(\Phi_{\mathcal{L}} \inSAT \big) \leq 1 - \delta,$$
so the limiting probability is bounded away from zero and one, and this interval is not larger than the corresponding interval for satisfiability when no variables are fixed. We observe that the length of the scaling window in \cite{Bollobas_2001} is on the order of $n^{-1/3}$, which is the reciprocal of the degrees of freedom for the critical 2-SAT problem. However, the proof presented in \cite{Bollobas_2001} differs from that of the current paper, and there is no direct coupling between the two results.

The main idea of the proof in \cite{Bollobas_2001} is to consider an order parameter for the phase transition of random $2$-SAT. This is a concept often used in statistical physics and it refers to a function that vanishes on one side of a transition and becomes non-zero on the other side. The order parameter that they consider is the average size of the spine, where the spine of a CNF-formula $\varphi$ is defined to be the set of literals $\ell$ for which there is a satisfiable sub-formula $\psi$ of $\varphi$ with $\psi\wedge\ell$ not satisfiable. By carefully controlling this quantity in a random CNF-formula as clauses are added one by one their result follows. Note that the size of the spine equals the number of variables that are free to be given any truth value without making a satisfiable SAT problem unsatisfiable. The spine only describes how each variable on its own affects the satisfiability of a CNF-formula. In contrast, we need to understand how all the fixed variables simultaneously impact the satisfiability of the formula. Multiple other papers, e.g. \cite{Chao86, Achlioptas01, Achlioptas01_1} also consider the procedure of fixing one single variable at a time, and in \cite{Achlioptas00} they consider fixing two variables at a time. This is different from the approach in the present paper where many variables are fixed simultaneously and hereby long implication chains emerge that intervene with each other and affect satisfiability. 

As previously mentioned, the paper \cite{Basse_2024} was the first to introduce and compute degrees of freedom in certain random under-constrained $k$-SAT problems. Their proof is based on the idea that fixing variables in a CNF-formula creates clauses of size one, also called unit-clauses. The presence of these unit-clauses, in turn, corresponds to further variable fixing. Thus, variables are fixed repeatedly in rounds, and the probability of encountering a contradiction in each round is calculated. The sequence describing the number of fixed variables throughout the rounds is then studied. This procedure is closely related to the unit-propagation algorithm, a well-studied technique used as a subroutine in most modern SAT solvers. We also base our proof on an appropriate adaptation of the unit propagation algorithm. In the under-constrained regime of random 2-SAT, it is possible to control the number of unit-clauses produced in each round $r$, and this number decreases exponentially at a rate of $\alpha$, i.e.\ as $\alpha^r$. However, at the phase transition, we have $\alpha = 1$, and thus the expected number of unit-clauses produced in each round remains approximately constant ($\alpha^r=1$). As a result, controlling unit propagation becomes more challenging because the entire process must be analyzed as a whole, unlike in the under-constrained regime, where the rounds could be considered independently. This again suggests the presence of long-range correlations between variables when $\alpha = 1$. When $f=\Omega(n^{1/3+\varepsilon})$ for some $\varepsilon>0$ the key idea is to show that the number of unit-clauses produced in each round remains high for a certain number of rounds w.h.p. This implies that a contradiction is likely to occur before the process terminates. On the other hand, when $f=O(n^{1/3-\varepsilon})$ for some $\varepsilon>0$ we show that the sequence dies out w.h.p.\ before encountering a contradiction.

The results of \cite{Basse_2024} extend further as they also determine the limiting satisfiability of the random SAT problem when $\Theta(f_\star)$ variables are fixed, where $f_\star$ represents the degrees of freedom of the random formula. In this setting, they show that the limiting probability remains bounded away from zero and one, and they provide the exact limiting value. By adjusting a parameter, this limiting value smoothly interpolates between the two edge cases. An open question is whether a similar result holds for the random critical $2$-SAT problem. Specifically, it remains unknown what happens when $\Theta(n^{1/3})$ variables are fixed in such formulas, and whether the limiting probability will also interpolate between the edge cases.

\section{Preliminaries}

\subsection{Notation and conventions}
For any set $A\subseteq\mathbb{Z}$ we define $-A=\{-a:a\in A\}$, $\pm A=A\cup(-A)$ and we denote by $|A|$ the number of elements in $A$. For elements $x_i$, $i\in A$ belonging to some space we let $(x_a)_{a\in A}$ denote the vector $(x_{a_1},\dotsc,x_{a_{|A|}})$, where $\{a_1,\dotsc,a_{|A|}\}= A$ and $a_1<a_2<\dotsb< a_{|A|}$. Furthermore, for any $n,m\in\mathbb{N}$ with $m<n$ we let $[n]=\{1,\dotsc,n\}$, $[m,n]=\{m,m+1,\dotsc,n\}$, and $[0]=\emptyset$. The two sets $\mathbb{B}=\{\texttt{true},\texttt{false}\}$ and $K=\{0,1,2,\star\}$ are also considered repeatedly. For an $x\in\mathbb{R}$ we let $x^+=\max\{0,x\}$. 

When considering random elements a probability space $(\Omega,\mathcal{F},\mathbb{P})$ will always be given.  Whenever new random elements are introduced, unless specified otherwise, they are independent of all previously existing randomness. We define $\frac{0}{0} = 0$. As we will ultimately let $n$ approach infinity, certain inequalities will hold only for sufficiently large $n$. In such cases, the required size of $n$ for the inequality to hold may depend on $q$, but it will always be independent of the round $r$. As has been the case thus far, $n$ is often omitted from the notation, even though most elements depend on this parameter.

\subsection{The random SAT-problem}
Let $n,m\in\mathbb{N}_0$ and $k\in\mathbb{N}$, where $n\geq k$ when $m>0$. The random $k$-SAT distribution was defined in section \ref{section main result}, but we will infer some additional notation needed for our proof. Firstly, we will specify the non-random case. When $m>0$ we let a $k$-clause over $n$ variables be a vector from the set
$$\mathcal{D}=\big\{(\ell_1,\dotsc,\ell_k)\in(\pm[n])^k:|\ell_1|<\dotsb<|\ell_k|\big\}.$$
The entries of such a vector are called the literals of the clause. Consider $m$ such clauses $(\ell_{j,i})_{i\in [k]}$, $j\in[m]$. From these clauses we define a $k$-SAT formula $\varphi$ with $n$ variables and $m$ clauses by letting
$$\varphi=\bigwedge_{j=1}^m\big(\ell_{j,1}\vee\dotsb\vee \ell_{j,k}\big).$$
We let the order of the clauses matter such that two formulas $\varphi$ and $\varphi'$ with literals $((\ell_{j,i})_{i\in [k]})_{j\in[m]}$ and $((\ell_{j,i}')_{i\in [k]})_{j\in[m]}$, respectively, are equal if and only if $\ell_{j,i}=\ell'_{j,i}$ for all $j\in[m]$ and $i\in[k]$. This implies a one-to-one correspondence between a formula and its (ordered set of) literals. Now, we define a mapping related to a SAT-formula. For $\ell\in\pm[n]$ we associate a mapping by letting
\begin{equation} \label{eq ell mapping}
\ell:\mathbb{B}^n\rightarrow\mathbb{B},\quad\text{where}\quad \ell:x=(x_1,\dotsc,x_n)\mapsto \begin{cases}
    \;\;\,x_{|\ell|},\quad &\text{if }\sign(\ell)=1,\\
    \neg\, x_{|\ell|},\quad&\text{if }\sign(\ell)=-1.
\end{cases}
\end{equation}
Letting $\wedge$ denote the logical \texttt{and} and $\vee$ denote the logical \texttt{or}, we associate $\varphi$ with the function mapping $\mathbb{B}^n$ to $\mathbb{B}$ that is given by
$$\varphi(x)=\bigg(\bigwedge_{j=1}^m\big(\ell_{j,1}\vee\dotsb\vee \ell_{j,k}\big)\bigg)(x)=\bigwedge_{j=1}^m\big(\ell_{j,1}(x)\vee\dotsb\vee \ell_{j,k}(x)\big),\quad x\in\mathbb{B}^n.$$
 We now define a distribution over the set of $k$-SAT formulas with $n$ variables and $m$ clauses and we denote this distribution by $F_k(n,m)$. Consider random vectors $(L_{j,i})_{i\in[k]}$, $j\in[m]$, that are uniformly distributed on $\mathcal{D}$. We say that these are random clauses. Furthermore, let
$$\Phi=\bigwedge_{j=1}^m\big(L_{j,1}\vee\dotsb\vee L_{j,k}\big),$$
then $\Phi$ has distribution $F_k(n,m)$ and we say that $\Phi$ is a random $k$-SAT formula with $n$ variables and $m$ clauses. For $x\in\mathbb{B}^n$ we let $\Phi(x)$ denote the point-wise evaluation of $\Phi$ in $x$.

\subsection{Fixing variables and the unit-propagation algorithm}
Let $n,m\in\mathbb{N}_0$ and $k\in\mathbb{N}$ with $n\geq k$ when $m>0$. Let $\mathcal{L}\subseteq\pm[n]$ be consistent. For an $x\in\mathbb{B}^n$ we let the vector $x_\mathcal{L}$ be as defined in subsection \ref{section main result} and for functions $g:\mathbb{B}^n\rightarrow\mathbb{B}$ we define $g_\mathcal{L}(x)=g(x_\mathcal{L})$. Consider a $2$-SAT formula $\varphi$ with $n$ variables and $m$ clauses where its literals are denoted $((\ell_{j,i})_{i\in[2]})_{j\in[m]}$. Consider the formula with fixed variables
$$\varphi_\mathcal{L}=\bigwedge_{j=1}^m\big(\ell_{j,1}\vee\ell_{j,2}\big)_\mathcal{L}=\bigwedge_{j=1}^m\big((\ell_{j,1})_\mathcal{L}\vee(\ell_{j,2})_\mathcal{L}\big).$$
The set $[m]$ is now split into four non-overlapping subsets:
\begin{equation} \label{eq sets A first def}
\begin{aligned}
    \mathcal{C}_0&=\big\{j\in[m]:\ell_{j,1}\in-\mathcal{L}\text{ and }\ell_{j,2}\in-\mathcal{L}\big\},\\
    \mathcal{C}_1&=\big\{j\in[m]:\ell_{j,i_{j_1}}\notin\pm\mathcal{L}\text{ and }\ell_{j,i_{j_2}}\in -\mathcal{L},\;\{i_{j_1},i_{j_2}\}=\{1,2\}\big\},\\
    \mathcal{C}_2&=\big\{j\in[m]:\ell_{j,1}\notin\pm\mathcal{L}\text{ and }\ell_{j,2}\notin\pm\mathcal{L}\big\},\\
    \mathcal{C}_\star&=\big\{j\in[m]:\ell_{j,1}\in\mathcal{L}\text{ or }\ell_{j,2}\in \mathcal{L}\big\}.
    \end{aligned}
\end{equation}
Using the definition of $i_{j_1}$ from above we ease notation and let $\ell_{j,i_{j_1}}=\ell_j$ for $j\in\mathcal{C}_1$. Note that
\begin{itemize}
    \item When $j\in\mathcal{C}_0$ then $(\ell_{j,1}\vee \ell_{j,1})_\mathcal{L} (x)=\texttt{false}$ for all $x\in\mathbb{B}^n$.
    \item When $j\in\mathcal{C}_1$ then $(\ell_{j,1}\vee\ell_{j,2})_\mathcal{L}(x)=\ell_{j}(x)$ for all $x\in\mathbb{B}^n$.
    \item When $j\in\mathcal{C}_2$ then $(\ell_{j,1}\vee\ell_{j,2})_{\mathcal{L}}(x)=(\ell_{j,1}\vee\ell_{j,2})(x)$ for all $x\in\mathbb{B}^n$.
    \item When $j\in\mathcal{C}_\star$ then $(\ell_{j,1}\vee\ell_{j,2})(x)=\texttt{true}$ for all $x\in\mathbb{B}^n$.
\end{itemize}
Define
\begin{equation} \label{eq phi_1 and phi_2 definition}
    \varphi_1=\bigwedge_{j\in\mathcal{C}_1}(\ell_{j}),\qquad \text{and}\qquad \varphi_2=\bigwedge_{j\in\mathcal{C}_2}(\ell_{j,1}\vee\ell_{j,2}).
\end{equation}
Note that the above literals will belong to the set $(\pm[n]\backslash\pm\mathcal{L})$. The above implies that $\varphi\inSAT$ if and only if $\mathcal{C}_0=\emptyset$ and $(\varphi_1\wedge\varphi_2)\inSAT$. We will now further determine when $(\varphi_1\wedge\varphi_2)\inSAT$.  Define
\begin{equation} \label{eq L(phi_1) definition}
    \mathcal{L}(\varphi_1)=\{\ell_j\in\mathcal{C}_1:-\ell_j\notin \mathcal{C}_1\}
\end{equation}
We let this be the set associated with the $1$-SAT formula $\varphi_1$, and we note that it is a consistent set. Moreover, for $x\in\mathbb{B}^n$
\begin{equation} \label{eq varphi_1 SAT}
\varphi_1(x)=\texttt{true}\quad \Longleftrightarrow \quad x_{|\ell_{j}|}=\begin{cases}
    \texttt{true},\quad&\text{when }\sign(\ell_{j})=+1,\\
    \texttt{false},\quad&\text{when }\sign(\ell_{j})=-1.
\end{cases}\quad \forall j\in\mathcal{C}_1.
\end{equation}
This along with the definition of $x_{\mathcal{L}(\varphi_1)}$ implies that when $\varphi_1\inSAT$ then for all $x\in\mathbb{B}^n$ we have that $\varphi_1(x)=\texttt{true}$ if and only if $x=x_{\mathcal{L}(\varphi_1)}$. Therefore
$$(\varphi_1\wedge\varphi_2)\inSAT\;\;\Longleftrightarrow\;\; \varphi_1\inSAT\;\text{ and }\;(\varphi_1\wedge\varphi_2)_{\mathcal{L}(\varphi_1)}\inSAT\;\;\Longleftrightarrow\;\; \varphi_1\inSAT\;\text{ and }\;(\varphi_2)_{\mathcal{L}(\varphi_1)}\inSAT.$$
Thus
\begin{equation} \label{eq varphi inSAT when}
\varphi\inSAT\quad\Longleftrightarrow\quad \mathcal{C}_0=\emptyset,\quad \varphi_1\inSAT,\quad\text{and}\quad (\varphi_2)_{\mathcal{L}(\varphi_1)}\inSAT.
\end{equation}
This decomposition of the event $\{\varphi\inSAT\}$ becomes a key tool in the proof. Moreover, note that the same procedure, as just described, can now be applied to the formula $(\varphi_2)_{\mathcal{L}(\varphi_1)}$. Hence, the procedure of fixing variables continues recursively in rounds and this is the idea behind the unit-propagation algorithm. One of the main ingredients in the proof of our main theorem concerns controlling this process. 

\subsection{Sketch of proof} \label{section proof sketch}
Consider a random $2$-CNF formula $\Phi \sim F_2(n, n)$ with literals $(L_{j,i})_{i \in [2], j \in [n]}$, and a consistent set $\mathcal{L} \subseteq \pm[n]$ with $|\mathcal{L}| = f$. We now apply the unit-propagation procedure to $\Phi$, thereby decomposing the probability of interest into a collection of simpler terms.

\emph{Initial round:} Let $\mathcal{C}_k^{(1)}$, for $k \in K$, be the random sets defined from $\Phi$ and $\mathcal{L}$ as described in (\ref{eq sets A first def}), and define $M_k\rr{1} = |\mathcal{C}_k\rr{1}|$ for $k \in K$. Additionally, let $\Phi_1^{(1)}$ and $\Phi_2^{(1)}$ be the random formulas constructed from $\Phi$ and $\mathcal{L}$, corresponding to the definitions in (\ref{eq phi_1 and phi_2 definition}). Finally, let $\mathcal{L}^{(1)}$ denote the set associated with $\Phi_1^{(1)}$, as defined in \eqref{eq L(phi_1) definition}. From the decomposition in (\ref{eq varphi inSAT when}), we get that
\begin{equation} \label{eq 1 asymp. prob cal}
\begin{aligned}
    \mathbb{P}(\Phi_\mathcal{L}\inSAT)&=\mathbb{P}(M_0\rr{1}=0,\;\Phi_1\rr{1}\inSAT,\;(\Phi_2\rr{1})_{\mathcal{L}\rr{1}}\inSAT).\\
    \end{aligned}
\end{equation}
The independence of the clauses of $\Phi$ implies that the three events in (\ref{eq 1 asymp. prob cal}) only are dependent through the random vector $(M_k\rr{1})_{k\in[K]}$. Moreover, the i.i.d. structure of the clauses in $\Phi$ implies that this vector has a multinomial distribution, where the entries concentrate around their mean and hence become asymptotically independent. This implies that also the events in (\ref{eq 1 asymp. prob cal}) are asymptotically independent, allowing for the desired decomposition:
\begin{equation} \label{eq 1 asymp. prob cal 2}
\begin{aligned}
    \mathbb{P}(\Phi_\mathcal{L}\inSAT)\sim \mathbb{P}(M_0\rr{1}=0\big)\mathbb{P}\big(\Phi_1\rr{1}\inSAT\big)\mathbb{P}\big((\Phi_2\rr{1})_{\mathcal{L}\rr{1}}\inSAT\big).
    \end{aligned}
\end{equation}

\emph{Subsequent rounds:} The procedure from the initial round is now repeated recursively, replacing $\Phi$ and $\mathcal{L}$ with $\Phi^{(1)}_2$ and $\mathcal{L}^{(1)}$, respectively. Hereby, new random elements $(M_k^{(2)})_{k\in K}$, $\Phi_1^{(2)}$, $\Phi_2^{(2)}$, and $\mathcal{L}^{(2)}$ are constructed. The procedure is then repeated iteratively on $\Phi_2^{(2)}$ and $\mathcal{L}^{(2)}$, and so on. Continuing a total of $R$ times ($R$ being some suitable integer), we construct the random elements $(M_k\rr{r})_{k\in K}$, $\Phi_1\rr{r}$, $\Phi_2\rr{r}$, and $\mathcal{L}\rr{r}$ for each $r \in [R]$. Using these constructed elements, the probability calculation in (\ref{eq 1 asymp. prob cal 2}) can be extended iteratively, leading to
\begin{equation} \label{eq decomp of prob}
    \mathbb{P}(\Phi_{\mathcal{L}}\inSAT)\sim \mathbb{P}\big((\Phi_2\rr{R})_{\mathcal{L}\rr{R}}\inSAT\big)\prod_{r=1}^R\mathbb{P}\big(M_0\rr{r}=0\big)\mathbb{P}\big(\Phi_1\rr{r}\inSAT\big).
\end{equation}
The probabilistic decomposition in (\ref{eq decomp of prob}) plays a central role in the overall proof. 
To evaluate the terms of (\ref{eq decomp of prob}), we need to know the distributions of the defined elements. As the elements are defined recursively, the distributions can be found as conditional distributions, and when conditioning on the past, we get that
\begin{align*}
    &M_0^{(r)}|M_1^{(r-1)}\approx  \text{Binomial}\bigg(\bigg(\frac{M_1^{(r-1)}}{2n}\bigg)^2,n\bigg),\quad \text{and}\quad\Phi_1^{(r)}|M_1^{(r)}\approx F_1(n,M_1^{(r)}),\quad (r\in[R]).
\end{align*}
 Thus, it becomes crucial to control the size of the sequence $(M_1^{(r)})_{r\in[R]}$, and the remaining part of the proof concerns this.

Firstly, we establish that $\lim_{n\rightarrow\infty}\mathbb{P}(\Phi_{\mathcal{L}}\inSAT)= 0$ when $f=n^q$ with $q=1/3+\varepsilon$. Here we will prove the existence of constants $c,C>0$, such that
\begin{align} \label{eq M_1^r in [cn^q,Cn^q]}
    \lim_{n\rightarrow\infty}\mathbb{P}\big(M_1^{(r)}\in[cn^q,Cn^q]\;\forall r\in[R]\big)=1,
\end{align}
which will imply that the product in (\ref{eq decomp of prob}) approaches zero and thus, this implies our main result. When proving \eqref{eq M_1^r in [cn^q,Cn^q]} a simple union bound will not do, and thus we will need to exploit the Markov structure of the sequence $(M_1\rr{r})_{r\in[R]}$. 

Next, we will establish that $\liminf_{n\rightarrow \infty}\mathbb{P}(\Phi_\mathcal{L}\inSAT)\geq \liminf_{n\rightarrow\infty}\mathbb{P}(\Phi\inSAT)$,  when $f=n^q$ with $q=1/3-\varepsilon$. In this setup, the sequence $(M_1\rr{r})_{r\in[R]}$ is a super-martingale, and thus optional sampling gives that $M_1\rr{r}\leq\log{n}\cdot n^q$ for all $r\in[R]$. This further implies that the product in \eqref{eq decomp of prob} approaches one as $n\rightarrow\infty$. Next, we will establish that the sequence $(M_1\rr{r})_{r\in[R]}$ is close in distribution to a critical Galton-Watson tree, and from this we can establish that $M_1^{(R)}=0$ w.h.p., which implies that $\mathcal{L}^{(R)}=\emptyset$ w.h.p. This further gives that $(\Phi_2^{(R)})_{\mathcal{L}^{(R)}}$ is close in distribution to $\Phi$, and thus the first term of \eqref{eq decomp of prob} is asymptotically equivalent to $\mathbb{P}(\Phi\inSAT)$. Thus, this finally proves our main theorem.  

\section{Main decomposition of probability}
In this section, we present a mathematically rigorous version of the decomposition in \ref{section proof sketch}. This decomposition will break the proof of our main result into smaller lemmas, which will be proven later.  In subsection \ref{lemma technical}, we introduce the technical lemmas that primarily provide distributional results for the sequences of elements that will be defined in sections \ref{Decomposition in over-constrained regime} and \ref{Decomposition in under-constrained regime}. The two sequences defined in these sections both serve as approximations to the unit propagation procedure. Section \ref{Decomposition in over-constrained regime} addresses the case $q > 1/3$, where the corresponding sequence is used to establish an upper bound on the probability, which approaches zero. In Section \ref{Decomposition in under-constrained regime}, the other sequence provides a lower bound that is used for the proof in the case $q < 1/3$.

\subsection{Technical lemmas} \label{lemma technical}
The first lemma of this section states that for $\Phi\sim F_2(n,m)$ and a consistent set of literals $\mathcal{L}\subseteq\pm[n]$, with $|\mathcal{L}|=f$ we can construct a coupled SAT-formula $\Phi'$ which has the same distribution as $\Phi$ but where fixing the literals of $\mathcal{L}$ in $\Phi$ corresponds to fixing the literals of the set $[n]\backslash[n-f]$. When considering the different rounds of the unit-propagation algorithm later on, the repeated use of this lemma will allow us to control which variables are fixed.
\begin{lemma} \label{lemma technical G}
    There exists a function $G$ such that if $\Phi\sim F_2(n,m)$ and $\mathcal{L}\subseteq \pm[n]$ is a consistent set of literals with $|\mathcal{L}|=f$, then $\Phi':=G(\Phi,\,\mathcal{L})\eqdist \Phi$ and
    $$\big\{\Phi_{\mathcal{L}}\inSAT\big\}=\big\{\Phi_{[n]\backslash [n-l]}'\inSAT\big\}.$$
\end{lemma}
The below is an easy consequence of the above lemma.
\begin{fact} \label{fact technical G with random L}
    Let $\Phi\sim F_2(n,m)$ and let $\mathcal{L}\subseteq \pm[n]$ be a consistent random set of literals independent of $\Phi$. Then $G(\Phi,\,\mathcal{L})\eqdist\Phi$ and $G(\Phi,\,\mathcal{L})$ is independent of $\mathcal{L}$.
\end{fact}
The proof of Lemma \ref{lemma technical G} relies on the uniformity of the clauses that imply that literals can be swapped without changing the distribution of the formula.

Next, we want to decompose a $2$-CNF formula with fixed variables into its $1$- and $2$-CNF sub-formulas. Let $\varphi$ be a (non-random) $2$-CNF formula with $n$ variables and $m$ clauses and let $\mathcal{L}=[n]\backslash[n-f]$ for some $f\in\mathbb{N}$. Define the sets
\begin{equation*} \label{eq technical A sets}
    \begin{aligned}
        \mathcal{A}_0=\mathcal{A}_0(n,f):=-\mathcal{L}\times -\mathcal{L},\quad& \mathcal{A}_1=\mathcal{A}_1(n,f):=\pm[n-l]\times -\mathcal{L},\\
    \mathcal{A}_2=\mathcal{A}_2(n,f):=\pm[n-f]\times \pm[n-f],\quad& \mathcal{A}_\star=\mathcal{A}_\star(n,f):=\pm[n]\times \mathcal{L}.
    \end{aligned}
\end{equation*}
Let $(\ell_{j,1},\ell_{j,2})$, $j\in[m]$, be the literals of $\varphi$ and define $\mathcal{C}_k=\{j\in[m]:(\ell_{j,1},\ell_{j,2})\in\mathcal{A}_k\}$, $k\in K$. Note that this definition corresponds to the definition in (\ref{eq sets A first def}). A clause that belongs to $\mathcal{A}_0$ is said to be an unsatisfied clause, and a clause in $\mathcal{A}_\star$ is said to be satisfied. Define
$$G_1(\varphi,f):=\bigwedge_{j\in \mathcal{C}_1}\ell_{j,1},\quad G_2(\varphi,f)=\bigwedge_{j\in \mathcal{C}_2}(\ell_{j,1}\vee \ell_{j,2}).$$
In (\ref{eq varphi inSAT when}) we saw that
$$\varphi_\mathcal{L}\inSAT\quad\Longleftrightarrow \quad \mathcal{C}_0=\emptyset,\quad G_1(\varphi,f)\inSAT,\quad G_2(\varphi,f)_{\mathcal{L}(G_1(\varphi,f))}\inSAT,$$
where $\mathcal{L}(G_1(\varphi,f))$ is defined in \eqref{eq L(phi_1) definition}. In the setup with $\mathcal{L}=[n]\backslash[n-f]$ we further note that when $(\ell_{j,1},\ell_{j,2})\in\mathcal{A}_1$, then $\ell_{j,1}\in\pm[n-f]$ and when $(\ell_{j,1},\ell_{j,2})\in\mathcal{A}_2$ then $(\ell_{j,1},\ell_{j,2})\in(\pm[n-f])^2$. Hence both $G_1(\varphi,f)$ and $G_2(\varphi,f)$ can be viewed as boolean functions that map $\{\pm 1\}^{n-f}$ into $\{\pm 1\}$.
The above setup will now be applied to a random $2$-CNF formula. The next lemma describes the simultaneous distribution of the elements defined in this setup.
\begin{lemma} \label{lemma technical dist of M}
   Let $\Phi\sim F_2(n,m)$ and $\mathcal{L}=[n]\backslash[n-f]$. If $M_k$ is the random variable denoting the number of clauses in $\Phi_k:=G_k(\Phi,f)$ for $k\in\{1,2\}$, and $M_0$ and $M_\star$ are the number of unsatisfied- and satisfied clauses, respectively, then
    $$(M_k)_{k\in K}=(M_0,\,M_1,\,M_2,\,M_\star)\sim\text{Multinomial}\big(m,\;{p}(n,f)\big),$$
    where ${p}=(p_k)_{k\in K}$ and 
    \begin{equation*}
        \begin{aligned}
            p_0(n,f)= \frac{f(f-1)}{4n(n-1)},&\quad p_1(n,f)=\frac{(n-f)f}{n(n-1)},\\
            p_2(n,f)=\frac{(n-f)(n-f-1)}{n(n-1)},&\quad p_\star(n,f)=\frac{(n-\frac{f}{4}-\frac{3}{4})f}{n(n-1)}.
        \end{aligned}
    \end{equation*}
    Furthermore
    $$\Phi_k|(M_k)_{k\in K}\sim F_k(n-f,\;M_k),\quad (k\in\{1,2\}),$$
    and $\Phi_1$ and $\Phi_2$ are conditionally independent given $(M_k)_{k\in K}$. 
\end{lemma}
This lemma is again a direct consequence of the uniformity and the independence of the clauses of a random $2$-CNF formula.

The last lemma of this section gives a lower bound on the probability that a $1$-CNF formula is satisfiable.
\begin{lemma} \label{lemma 1-SAT}
    Let $n,m\in\mathbb{N}$ with $n\geq m$ and let $\Phi\sim F_1(n,m)$. Then
    $$\mathbb{P}(\Phi\inSAT)\geq \big(1-\tfrac{m}{n}\big)^m.$$
\end{lemma}
This lemma can be proven in the same way as they prove Lemma 8 in \cite{Basse_2024}. Thus, we will not repeat the argument here. 

\subsection{Decomposition of probability when many variables are fixed}
\label{Decomposition in over-constrained regime}
Let $\Phi\sim F_2(n,n)$, and $\mathcal{L}$ be a consistent set of literals with $|\mathcal{L}|=f=f(n)$, where $f(n)=\Omega( n^{1/3+\varepsilon
})$ for a small $\varepsilon>0$. We will prove that $\lim_{n\rightarrow\infty}\mathbb{P}(\Phi_{\mathcal{L}}\inSAT)=0$. In this section, the aim is to closely regulate the unit-propagation procedure and hereby establish an upper bound on the probability of interest. Later, it is established that this upper bound approaches zero as $n\rightarrow\infty$. 

\subsubsection*{Controlling the unit-propagation procedure}

The assumption on $f$ implies that $f(n)\geq n^{q}$, where $q=1/3+\varepsilon$ for some small $\varepsilon>0$. Let $\mathcal{L}'\subseteq \mathcal{L}$ with $|\mathcal{L}'|=\lfloor n^q\rfloor$. As $\{\Phi_{\mathcal{L}}\inSAT\}\subseteq \{\Phi_{\mathcal{L}'}\inSAT\}$ it is sufficient to establish that $\lim_{n\rightarrow\infty}\mathbb{P}(\Phi_{\mathcal{L}'}\inSAT)=0$. Thus, we will WLOG assume that $f(n)=\lfloor n^q\rfloor$ for some $1/3<q<1/2$. 

Next, we define a sequence of random elements that resembles a controlled version of the unit-propagation procedure. First, we define the initial elements of the procedure. Let $G$ be the function defined in Lemma \ref{lemma technical G}. Then define
\begin{align*}
    \Psi_2\rr{0}:=G(\Phi,\mathcal{L}),\quad S\rr{-1}=0,\quad\bar{S}\rr{-1}=0,\quad S\rr{0}:=f,\\ M_1\rr{0}=f+1,\quad
    \mathcal{L}\rr{0}:=[n]\backslash[n-f],\quad\quad\mathcal{M}\rr{0}:=\{\emptyset,\Omega\}.
\end{align*}
Note that $S\rr{0}=S\rr{-1}+(M_1\rr{0}-1)^+$,  $\mathcal{L}\rr{0}=[n-S\rr{-1}]\backslash[n-S\rr{0}]$, and Lemma \ref{lemma technical G} states that $\Psi_2\rr{0}$ is constructed such that
\begin{equation} \label{eq over-constrained initial initial}
    \Psi_2\rr{0}\sim F_2(n,n)=F_2\big(n-S\rr{-1},\;(n-\bar{S}\rr{-1})^+\big)\quad\text{and}\quad \big\{\Phi_\mathcal{L}\inSAT\big\}=\big\{(\Psi_2\rr{0})_{\mathcal{L}\rr{0}}\inSAT\big\}.
\end{equation}
Furthermore, $\mathcal{M}\rr{0}$ is the trivial $\sigma$-algebra and thus it provides no information. Now, additional elements are constructed recursively. Let $R:=\lfloor n^{1-2q}\log{n}\rfloor$ denote the number of rounds. Then for each $r\in[R]$ we define the following recursively. 

Let $G_1$ and $G_2$ be the functions from Lemma \ref{lemma technical dist of M} and define $\Phi_k\rr{r}:=G_k(\Psi_2\rr{r-1},|\mathcal{L}\rr{r-1}|)$ for $k\in\{1,2\}$. Also, let $M_k\rr{r}$ denote the number of clauses in $\Phi_k\rr{r}$ for $k\in\{1,2\}$ and let $M_0\rr{r}$ and $M_\star\rr{r}$ denote the number of unsatisfied- and satisfied clauses of $(\Psi_2\rr{r-1})_{\mathcal{L}\rr{r-1}}$, respectively. Define the $\sigma$-algebra $\mathcal{M}\rr{r}:=\sigma\big(\mathcal{M}\rr{r-1}\cup\sigma(M_k\rr{r},\;k\in K)\big)$.
The elements are constructed such that
\begin{equation} \label{eq over-constrained initial 1}
    \begin{aligned}
        \big\{(\Psi_2\rr{r-1})_{\mathcal{L}\rr{r-1}}\inSAT\big\}=\big\{(\Phi_2\rr{r})_{\mathcal{L}(\Phi_1\rr{r})}\inSAT,\;\Phi_1\rr{r}\inSAT,\;\;M_0\rr{r}=0\big\},
    \end{aligned}
\end{equation}
see (\ref{eq varphi inSAT when}), and Lemma \ref{lemma technical dist of M} states that
\begin{align}
        (M_k\rr{r})_{k\in K}|\mathcal{M}\rr{r-1}&\sim\text{Binomial}\Big((n-\bar{S}\rr{r-2})^+,\;{p}\big(n-S\rr{r-2},\;(M_1\rr{r-1}-1)^+\big)\Big), \label{eq over M_k^r dist}\\
        \Phi_k\rr{r}|\mathcal{M}\rr{r}&\sim F_2(n-S\rr{r-1},M_k\rr{r}),\quad k\in\{1,2\}, \label{eq over Phi_1^r dist}
\end{align}
and $\Phi_2\rr{r}$ and $\Phi_1\rr{r}$ are independent when conditioning on $\mathcal{M}\rr{r}$. Now, define $\bar{\Psi}_2\rr{r}:=G(\Phi_2\rr{r},\mathcal{L}(\Phi_1\rr{r}))$, where $G$ is the function from Lemma \ref{lemma technical G} and the set corresponding to $\Phi_1\rr{r}$ is defined in (\ref{eq L(phi_1) definition}). As $\Phi_2\rr{r}$ and $\Phi_1\rr{r}$ are independent given $\mathcal{M}\rr{r}$, Fact \ref{fact technical G with random L} states that
$$\bar{\Psi}_2\rr{r}|\mathcal{M}\rr{r}\sim F_2(n-S\rr{r-1},M_k\rr{r}),\quad\text{and}\quad \bar{\Psi}_2\rr{r}\ind\Phi_1\rr{r}\;|\mathcal{M}\rr{r}.$$
Moreover, if $\bar{M}_1\rr{r}=|\mathcal{L}(\Phi_1\rr{r})|$, then
\begin{equation} \label{eq over-constrained Psi}
    \big\{(\Phi_2\rr{r})_{\mathcal{L}(\Phi_1\rr{r})}\inSAT\big\}=\big\{(\bar{\Psi}_2\rr{r})_{\bar{\mathcal{L}}\rr{r}}\inSAT\big\},\quad\text{where}\quad \bar{\mathcal{L}}\rr{r}=[n-S\rr{r-1}]\backslash[n-S\rr{r-1}-\bar{M}_1\rr{r}].
\end{equation}
Now, we either add clauses to $\bar{\Psi}_2\rr{r}$ or remove clauses. Define $\bar{S}\rr{r-1}=\lfloor\log{n}\rfloor \cdot S\rr{r-1}$ and let $(L_{j,1}\rr{r},L_{j,2}\rr{r})$ for $j\in[M_2\rr{r}]$ be the random literals of $\bar{\Psi}_2\rr{r}$. If $M_2\rr{r}<(n-\bar{S}\rr{r-1})^+$ define additional random literals $(L_{j,1}\rr{r},L_{j,2}\rr{r})$ for $j\in\{M_2\rr{r},\dotsc,(n-\bar{S}\rr{r-1})^+\}$ where conditional on $\mathcal{M}\rr{r}$ they are i.i.d. and uniformly distributed on $\mathcal{D}\rr{r}:=\{(\ell_1,\ell_2)\in(\pm[n-S\rr{r-1}])^2:|\ell_1|<|\ell_2|\}$. Define
$${\Psi}_2\rr{r}=\bigwedge_{j\in[(n-\bar{S}\rr{r-1})^+]}\big(L_{j,1}\rr{r}\vee L_{j,2}\rr{r}\big),\quad\text{then}\quad {\Psi}_2\rr{r}|\mathcal{M}\rr{r}\sim F_2\big(n-S\rr{r-1},(n-\bar{S}\rr{r-1})^+\big).$$
Lastly, let 
$$S\rr{r}=S\rr{r-1}+(M_1\rr{r}-1)^+,\quad\text{and}\quad \mathcal{L}\rr{r}=[n-S\rr{r-1}]\backslash[n-S\rr{r}].$$
Then we are in the same setting again and we can repeat the procedure on $\Psi_2\rr{r}$ and $\mathcal{L}\rr{r}$ under the conditional distribution given $\mathcal{M}\rr{r}$, where we note that $\mathcal{L}\rr{r}$ is deterministic given $\mathcal{M}\rr{r}$.

Note that it is mainly the sequence $\{S\rr{r}\}_{r\in[R]}$ that controls the size of the different elements constructed above and this sequence is defined from the sequence $\{M_1\rr{r}\}_{r\in[R]}$. Thus, a big part of the proof in the over-constrained setting is controlling the size of this sequence, which describes the number of unit-clauses constructed in each round. We show that that this number remains on the order of $n^q$ (remember that $f(n)=\lfloor n^q\rfloor$) throughout the $R$ rounds as the below lemma states.
\begin{lemma} \label{lemma over-constrained control size of M_1^r}
    There exist constants $c_0>0$ and $C_0>0$ such that the two events
    $$B_l=\big\{M_1^{(r)}\geq c_0 n^q,\; r\in[R]\big\}\quad \text{and}\quad {B}_u=\big\{M_1^{(r)}\leq C_0 n^q,\; r\in[R]\big\}$$
    satisfy
    $$\lim_{n\rightarrow\infty}\mathbb{P}\big({B}_l\big)=\lim_{n\rightarrow\infty}\mathbb{P}\big({B}_u\big)=1.$$
\end{lemma}
As $S\rr{r}\leq\lfloor n^q\rfloor +\sum_{r=1}^RM_1\rr{r} $ the above Lemma also implies that:
\begin{fact} \label{fact bound on S^r}
    There exists a constant $C_1>0$ such that for $r\in[R]$ (and $n$ large enough) we have 
    $$\big\{S\rr{r}\leq C_1n^{1-q}\log{n}\big\}\subseteq B_u.$$
\end{fact}
Lemma \ref{lemma over-constrained control size of M_1^r} is technical to prove. It is easy to find constants $c_0>0$ and $C_0>0$ such that for each $r\in[R]$ we have that $c_0n^q<M_1\rr{r}<C_0n^q$ w.h.p. This does however not imply that the entire sequence $\{M_1\rr{r}\}_{r\in[R]}$ is uniformly bounded w.h.p. A union bound is not tight enough to establish the uniform boundedness so the dependence structure of the sequence needs to be exploited. We establish that when $M_1\rr{R}$ is bounded w.h.p. the previous elements will be bounded w.h.p. as well.

\subsubsection*{Decomposing the probability}
The random elements defined in the controlled unit-propagation procedure above will now be related to the probability that $\Phi_\mathcal{L}$ is satisfiable. Our aim is to show that the probability tends to zero and thus we want to construct an upper bound on the probability. Let $B_u$ and $B_l$ be the events from Lemma \ref{lemma over-constrained control size of M_1^r}.
Using equation (\ref{eq over-constrained initial initial}) we first note that
\begin{equation*}
    \begin{aligned}
        \mathbb{P}\big(\Phi_\mathcal{L}\inSAT,\;B_u,\;B_l\big)=\mathbb{P}\big((\Psi_2\rr{0})_{\mathcal{L}\rr{0}}\inSAT,\;B_u,\;B_l\Big).
    \end{aligned}
\end{equation*}
Next, using (\ref{eq over-constrained initial 1}) and (\ref{eq over-constrained Psi}) on the term at the right gives that
\begin{equation} \label{eq over-constrained first}
    \begin{aligned}
        &\mathbb{P}\Big((\Psi_2\rr{0})_{\mathcal{L}\rr{0}}\inSAT,\;B_u,\;B_l\Big)\\
        = & \mathbb{P}\Big((\Phi_2\rr{1})_{\mathcal{L}(\Phi_1\rr{1})}\inSAT,\;\Phi_1\rr{1}\inSAT,\;M_0\rr{1}=0,\;B_u,\;B_l\Big)\\
        =&\mathbb{P}\Big((\bar{\Psi}_2\rr{1})_{\bar{\mathcal{L}}\rr{1}}\inSAT,\;\Phi_1\rr{1}\inSAT,\;M_0\rr{1}=0,\;B_u,\;B_l\Big)\\
        \leq &\mathbb{P}\big(\big(\bar{\Psi}_2\rr{1}\big)_{\bar{\mathcal{L}}\rr{1}}\inSAT,\;M_0\rr{1}=0,\;B_u,\;B_l,\;M_1\rr{1}\leq\bar{M}_1\rr{1}+1,\;M_2\rr{1}\geq (n-\bar{S}\rr{-1})^+\big)\\
    &+\mathbb{P}\big(M_1\rr{1}>\bar{M}_1\rr{1}+1,\;\Phi_1\rr{1}\inSAT,\;B_u\big)+\mathbb{P}\big(M_2\rr{1}<(n-\bar{S}\rr{-1})^+,\;B_u,\;B_l\big).
    \end{aligned}
\end{equation}
The first term in the last expression above will now be further decomposed. Note that when $M_1\rr{1}\leq \bar{M}_1\rr{1}+1$ then ${\mathcal{L}}\rr{1}\subseteq \bar{\mathcal{L}}\rr{1}$ and when $M_2\rr{1}\geq (n-\bar{S}\rr{-1})^+$ then ${\Psi}_2\rr{1}$ is a sub-formula of $\bar{\Psi}_2\rr{1}$. Thus
\begin{equation} \label{eq over-constrained second}
    \begin{aligned}
&\mathbb{P}\big(\big(\bar{\Psi}_2\rr{1}\big)_{\bar{\mathcal{L}}\rr{1}}\inSAT,\;M_0\rr{1}=0,\;B_u,\;B_l,\;M_1\rr{1}\leq\bar{M}_1\rr{1}+1,\;M_2\rr{1}\geq (n-\bar{S}\rr{-1})^+\big)\\
\leq &\mathbb{P}\big(({\Psi}_2\rr{1})_{\mathcal{L}\rr{1}}\inSAT,M_0\rr{1}=0,\;B_u,\;B_l\big).
    \end{aligned}
\end{equation}
Now, recursively repeating (\ref{eq over-constrained first}) and (\ref{eq over-constrained second}) $R$ times in total we eventually arrive at the decomposition
\begin{equation} \label{eq upper final decomposition}
    \begin{aligned}
    \mathbb{P}\big((\Phi_\mathcal{L}\inSAT,\;B_u,\;B_l\big)\leq &\mathbb{P}\big(M_0^{(r)}=0,\; r\in[R],\;B_u,\;B_l\big)\\
    +&\sum_{r=1}^R\mathbb{P}\big(M_1^{(r)}>\bar{M}_1^{(r)}+1,\;\Phi_1\rr{1}\inSAT,\;B_u\big)\\
    +&\sum_{r=1}^R\mathbb{P}\big(M_2^{(r)}<(n-\bar{S}\rr{r-1})^+,\;B_u,\;B_l\big).
    \end{aligned}
\end{equation}
The below lemma establishes the limits of the above upper bound.
\begin{lemma} \label{lemma three limits}
    It holds that
    \begin{enumerate}[label=(\arabic*)]
        \item $\lim_{n\rightarrow\infty}\mathbb{P}\big(M_0^{(r)}=0,\; r\in[R],\;B_u,\;B_l\big)=0$,
        \item $\lim_{n\rightarrow\infty}\sum_{r=1}^R\mathbb{P}\big(M_1\rr{r}\geq  \bar{M}_1\rr{r}+2,\;\Phi_1\rr{1}\inSAT,\;B_u\big)= 0$,
        \item $\lim_{n\rightarrow\infty}\sum_{r=1}^R\mathbb{P}\big(M_2\rr{r}<(n-\bar{S}\rr{r-1})^+,\;B_u,\;B_l\big)= 0.$
    \end{enumerate}
\end{lemma}
When proving the above Lemma the events $B_u$ and $B_l$ make it possible to control the sizes of the different random elements. Further, Lemma \ref{lemma over-constrained control size of M_1^r} makes it possible to evaluate one event at a time by conditioning on previous information and hereby knowing exact distributions. 

The above decomposition and lemmas make it straight forward to prove that $\Phi_\mathcal{L}$ is indeed asymptotically unsatisfiable. 

\begin{proof}[Proof of Definition \ref{def degrees of freedom} (2)]
    Lemma \ref{lemma over-constrained control size of M_1^r} implies that it is sufficient to establish that the right-hand side of (\ref{eq upper final decomposition}) tends to zero as $n\rightarrow\infty$. But this is a direct consequence of Lemma \ref{lemma three limits}. 
\end{proof}

\subsection{Decomposition of probability when few variables are fixed}
\label{Decomposition in under-constrained regime}

We will now also control the unit-propagation procedure when the problem is asymptotically satisfiable, where we instead need a lower bound. Let $\Phi\sim F_2(n,n)$ and $\mathcal{L}\subseteq\pm [n]$ be consistent with $|\mathcal{L}|=f(n)$, where $f(n)\leq n^{1/3-\varepsilon}$ for a small $\varepsilon>0$. In section \ref{Decomposition in over-constrained regime}, we saw that the number of unit-clauses remained of order $n^q$ throughout the $R$ rounds. In this section, we instead want to show that the unit-propagation procedure terminates and thus that the number of unit-clauses reaches zero within the number of rounds we consider. It turns out that the number of unit-clauses generated by this algorithm will be a super-martingale (on a set of probability one) when considering the sequence from round $r=2$ and onward. This is helpful as we will make use of optional sampling. Therefore, we will start by stating another lemma for which the entire sequence of one-clauses is a super-martingale and then we will connect this lemma to our main theorem.
\begin{lemma} \label{lemma under-constrained alternative result}
    Let $0<q<1/3$ and let $M_1\rr{-1}$ and $M_1\rr{0}$ be random variables taking values in $[n]$ satisfying that $\mathbb{E}[M_1\rr{0}]\leq C_0n^q$ for some $C_0>0$ and also that
    $$\lim_{n\rightarrow\infty}\mathbb{P}\big(M_1\rr{-1}\leq n^q\log{n}\big)=\lim_{n\rightarrow\infty}\mathbb{P}\big(M_1\rr{0}\leq n^q\log{n}\big)=1.$$
    Define $\mathcal{L}'=[n-M_1\rr{-1}]\backslash[n-M_1\rr{-1}-M_1\rr{0}]$ and $\mathcal{M}\rr{0}=\sigma(M_1\rr{-1},\,M_1\rr{0})$ and let $\Phi'$ be a random function with 
    $$\Phi'|\mathcal{M}\rr{0}\sim F_2\big(n-M_1\rr{-1},\;n-M_1\rr{-1}-M_1\rr{0}\big).$$
    If $\Phi\sim F_2(n,n)$ then
   $$\liminf_{n\rightarrow\infty}\mathbb{P}\big(\Phi'_{{\mathcal{L}}'}\inSAT\big)\geq \liminf_{n\rightarrow\infty}\mathbb{P}\big(\Phi\inSAT\big).$$
\end{lemma}
\subsubsection*{Controlling the unit-propagation procedure}
The notation used when naming the elements of the unit-propagation procedure in subsection \ref{Decomposition in over-constrained regime} is now reused in this section. As there are small differences in the definitions in the two cases it is important to pay attention to which definitions apply to which lemmas.

Let $\Phi'$, $\mathcal{L}'$, $M_1\rr{-2}$ and $M_1\rr{-1}$ be the elements of Lemma \ref{lemma under-constrained alternative result}. Again we start by defining some initial elements of our unit-propagation procedure:
\begin{equation*}
    \begin{aligned}
        \Psi_2\rr{0}:=\Phi',\quad &S\rr{-1}=M_1\rr{-1},\quad S\rr{0}=S\rr{-1}+M_1\rr{0},\quad
        \mathcal{L}\rr{0}:={\mathcal{L}}',\quad \mathcal{M}\rr{0}=\sigma\big(M_1\rr{-2},M_1\rr{-1}\big).
    \end{aligned}
\end{equation*}
Note that $\mathcal{L}\rr{0}=[n-S\rr{-1}]\backslash[n-S\rr{0}]$ and $\Psi_2\rr{0}|\mathcal{M}\rr{0}\sim F_2(n-S\rr{-1},\;n-S\rr{0})$. Now the rest of the elements are generated recursively. Let $R=\lfloor n^{1-2q}\log^{-3}{n}\rfloor$ denote the number of rounds. Then for each $r\in[R]$ we define the following recursively. Let $G_1$ and $G_2$ be the functions defined in Lemma \ref{lemma technical G} and let $\Phi_k\rr{r}=G_k(\Psi_2\rr{r-1},\mathcal{L}\rr{r-1})$ for $k\in\{1,2\}$. Also, let $M_k\rr{r}$ be the number of clauses in $\Phi_k\rr{r}$ for $k\in\{1,2\}$ and let $M_0\rr{r}$ and $M_\star\rr{r}$ denote the number of unsatisfied- and satisfied clauses of $(\Psi_k\rr{r-1})_{\mathcal{L}\rr{r-1}}$, respectively. We further define $\mathcal{M}\rr{r}=\sigma\big(\mathcal{M}\rr{r-1}\cup\sigma(M_k\rr{r},\,k\in K)\big).$ We have that
\begin{equation} \label{eq under-constrained initial 1}
    \begin{aligned}
        \big\{(\Psi_2\rr{r-1})_{\mathcal{L}\rr{r-1}}\inSAT\big\}=\big\{(\Phi_2\rr{r})_{\Phi_1\rr{r}}\inSAT,\;\Phi_1\rr{r}\inSAT,\;M_0\rr{r}=0\big\},
    \end{aligned}
\end{equation}
see (\ref{eq varphi inSAT when}), and Lemma \ref{lemma technical dist of M} states that
\begin{align}
        (M_k\rr{r})_{k\in K}|\mathcal{M}\rr{r-1}&\sim\text{Binomial}\big(n-S\rr{r-1},\;{p}(n-S\rr{r-2},\;M_1\rr{r-1})\big), \label{eq under M_k^r dist}\\
        \Phi_k\rr{r}|\mathcal{M}\rr{r}&\sim F_2(n-S\rr{r-1},M_k\rr{r}),\;k\in\{1,2\},\quad \Phi_1\rr{r}\ind\Phi_2\rr{r}|\mathcal{M}\rr{r}.\label{eq under Phi dist}
\end{align}
Now, define $\bar{\Psi}_2\rr{r}:=G(\Phi_2\rr{r},\Phi_1\rr{r})$, where $G$ is the function from Lemma \ref{lemma technical G} and $\Phi_1\rr{r}$ is seen as set, see \eqref{eq L(phi_1) definition}. As $\Phi_2\rr{r}$ and $\Phi_1\rr{r}$ are independent given $\mathcal{M}\rr{r}$, Fact \ref{fact technical G with random L} states that
\begin{equation} \label{eq under ind of Phi and Psi}
    \bar{\Psi}_2\rr{r}|\mathcal{M}\rr{r}\sim F_2(n-S\rr{r-1},M_k\rr{r}),\quad\text{and}\quad \Psi_2\rr{r}\ind\Phi_1\rr{r}\;|\mathcal{M}\rr{r}. 
\end{equation}
Moreover, if $\bar{M}_1\rr{r}$ denotes the number of distinct variables appearing in $\Phi_1\rr{r}$, then
\begin{equation} \label{eq under-constrained Psi}
    \big\{(\Phi_2\rr{r})_{\Phi_1\rr{r}}\inSAT\big\}=\big\{(\bar{\Psi}_2\rr{r})_{\bar{\mathcal{L}}\rr{r}}\inSAT\big\},\quad\text{where}\quad \bar{\mathcal{L}}\rr{r}=[n-S\rr{r-1}]\backslash[n-S\rr{r-1}-\bar{M}_1\rr{r}].
\end{equation}
Now, we add additional clauses to $\bar{\Psi}_2\rr{r}$ or remove clauses. Let $S\rr{r}=S\rr{r-1}+M_1\rr{r}$. Recall that
$$M_2\rr{r}=M_2\rr{r-1}-\big(M_0\rr{r}+M_1\rr{r}+M_\star\rr{r}\big)=\big(n-S\rr{r-1}\big)-\big(M_0\rr{r}+M_1\rr{r}+M_\star\rr{r}\big)\leq n-S\rr{r}.$$
Let $(L_{j,1}\rr{r},L_{j,2}\rr{r})$, $j\in [M_2\rr{r}]$ be the random random literals of $\bar{\Psi}_2\rr{r}$ and define additional random literals $(L_{j,1}\rr{r},L_{j,2}\rr{r})$ for $j\in\{M_2\rr{r}+1,\dotsc,n-S\rr{r}\}$ that when conditioning on $\mathcal{M}\rr{r}$ are i.i.d. and uniformly distributed on $\mathcal{D}\rr{r}:=\big\{(\ell_1,\ell_2)\in(\pm [n-S\rr{r-1}])^2\,:\,|l_1|<|l_2|\big\}$. Define
$${\Psi}_2\rr{r}=\bigwedge_{j\in[n-S\rr{r}]}\big(L_{j,1}\rr{r}\vee L_{j,2}\rr{r}\big),\quad\text{then}\quad {\Psi}_2\rr{r}|\mathcal{M}\rr{r}\sim F_2\big(n-S\rr{r-1},n-S\rr{r}\big).$$
Lastly, define $\mathcal{L}\rr{r}=[n-S\rr{r-1}]\backslash[n-S\rr{r}]$. Now, the same procedure can be applied to $\Psi_2\rr{r}$ and $\mathcal{L}\rr{r}$ in the conditional distribution given $\mathcal{M}\rr{r}$, where we note that $\mathcal{L}\rr{r}$ is deterministic given $\mathcal{M}\rr{r}$.
\subsubsection*{Decomposing the probability}
We will use the elements defined previously to create a lower bound on the probability of $\Phi'_{\mathcal{L}'}$ being satisfiable. The definitions in the initial round imply that
$$\big\{\Phi'_{\mathcal{L}'}\inSAT\big\}=\big\{(\Psi_2\rr{0})_{\mathcal{L}\rr{0}}\inSAT\big\}.$$
Now, using (\ref{eq under-constrained initial 1}) we get
\begin{equation} \label{eq under first}
    \big\{(\Psi_2\rr{0})_{\mathcal{L}\rr{0}}\inSAT\big\}=\big\{(\Phi_2\rr{1})_{\Phi_1\rr{1}}\inSAT,\;\Phi_1\rr{1}\inSAT,\;M_0\rr{1}=0\big\},
\end{equation}
and equation (\ref{eq under-constrained Psi}) further implies
\begin{equation} \label{eq under second}
    \big\{(\Phi_2\rr{1})_{\mathcal{L}\rr{1}}\inSAT \big\}=\big\{(\bar{\Psi}_2\rr{1})_{\bar{\mathcal{L}}\rr{1}}\inSAT\big\}.
\end{equation}
As $\bar{M}_1\rr{1}\leq M_1\rr{1}$ we have that $\mathcal{L}\rr{1}\subseteq\bar{\mathcal{L}}\rr{1}$ and also $\Psi_2\rr{1}$ is constructed such that $\bar{\Psi}_2\rr{1}$ is its sub-formula. Thus, we get the inclusions
\begin{equation} \label{eq under third}
    \big\{(\bar{\Psi}_2\rr{1})_{\bar{\mathcal{L}}\rr{1}}\inSAT\big\}\supseteq\big\{(\Psi_2\rr{1})_{\mathcal{L}\rr{1}}\inSAT\big\}.
\end{equation}
Combining all of the above set inclusions imply that
\begin{equation*}
    \mathbb{P}\big((\Phi')_{\mathcal{L}'}\inSAT\big)\geq\mathbb{P}\big((\Psi_2\rr{1})_{\mathcal{L}\rr{1}}\inSAT,\;\Phi_1\rr{1}\inSAT,\;M_0\rr{1}=0\big).
\end{equation*}
Now, we are back at considering the event $\{(\Psi_2\rr{1})_{\mathcal{L}\rr{1}}\inSAT\}$ and thus (\ref{eq under first}), (\ref{eq under second}) and (\ref{eq under third}) can be repeated for $r=2,\dotsc,R$. Hereby, we eventually get the lower bound
\begin{align} \label{eq lower decomposition}
\mathbb{P}\big(\Phi'_{\mathcal{L}'}\inSAT\big)\geq\mathbb{P}\big((\Psi_2\rr{R})_{\mathcal{L}\rr{R}}\inSAT,\;\Phi_1\rr{r}\inSAT,\;M_0\rr{1}\inSAT,\;r\in[R]\big).
\end{align}
Our next lemma gives that the above lower-bound tends to one as $n\rightarrow\infty$.
\begin{lemma} \label{lemma five limits}
We have that
\begin{enumerate}[label=(\arabic*)]
    \item $\lim_{n\rightarrow\infty}\mathbb{P}\big(M_1\rr{r}\leq n^q\log{n},\;r\in[R]\big)=1$,
    \item $\lim_{n\rightarrow\infty}\mathbb{P}\big(\Phi_1\rr{r}\inSAT,\;r\in[R]\,\big|\,M_1\rr{r}\leq n^q\log{n},\;r\in\{-1,\dotsc,R\}\big)= 1$,
    \item $\lim_{n\rightarrow\infty}\mathbb{P}\big(M_0\rr{r}=0,\;r\in[R]\,\big|\,M_1\rr{r}\leq n^q\log{n},\;r\in\{-1,\dotsc,R\}\big)=1$,
    \item $\lim_{n\rightarrow\infty}\mathbb{P}\big(M_1\rr{R}=0\big)=1.$
\end{enumerate}
\end{lemma}
That the sequence $(M_1\rr{r})_{r\in[R]}$ is bounded from above follows using optional sampling where we exploit that the sequence turns out to be a super-martingale. Lemma \ref{lemma five limits} (2) and (3) are then consequences of Lemma \ref{lemma technical dist of M}. Lastly (4) is proven by a Poisson approximation and also using theory of Galton-Watson trees. 
Lemma \ref{lemma under-constrained alternative result} is now an easy consequence of Lemma \ref{lemma five limits}.

\begin{proof}[Proof of Lemma \ref{lemma under-constrained alternative result}]
The definitions of $M_1\rr{-1}$ and $M_1\rr{0}$ along with Lemma \ref{lemma five limits} (1) imply that the event that we condition on in (2) and (3) of Lemma \ref{lemma five limits} happens w.h.p. Therefore, Lemma \ref{lemma five limits} (2) implies that $$\lim_{n\rightarrow\infty}\mathbb{P}(\Phi_1\rr{r}\inSAT,\;r\in[R])=1,$$ and Lemma \ref{lemma five limits} (3) implies that $$\lim_{n\rightarrow\infty}\mathbb{P}(M_0\rr{r}=0,\;r\in[R])=1.$$ 
Also, Lemma \ref{lemma five limits} (4) implies that $\mathcal{L}\rr{R}=\emptyset$ w.h.p. and when this is the case also $\Psi_2\rr{R}|\mathcal{M}\rr{R}\sim F_2(n-S\rr{R-1},n-S\rr{R-1})$. Moreover, that $M_1\rr{r}\leq n^q\log{n}$ for all $r\in[R]$ w.h.p. implies that $S\rr{R-1}\leq n^{1-q}$ w.h.p. These observations along with Fatou's Lemma give
\begin{equation*}
    \begin{aligned}
    &\liminf_{n\rightarrow\infty}\mathbb{P}\big((\Psi_2\rr{R})_{\mathcal{L}\rr{R}}\inSAT\big)\\
    = &\liminf_{n\rightarrow\infty}\mathbb{E}\big[\mathbb{P}\big(\Psi_2\rr{R}\inSAT\big|\mathcal{M}\rr{R}\big)\big|S\rr{R-1}\leq n^{1-q},\;\mathcal{L}\rr{R}=\emptyset\big]\\
    \geq &\mathbb{E}\Big[\liminf_{n\rightarrow\infty}\mathbb{P}\big(\Psi_2\rr{R}\inSAT\big|\mathcal{M}\rr{R}\big)\Big|S\rr{R-1}\leq n^{1-q},\;\mathcal{L}\rr{R}=\emptyset\Big]\\
    =&\liminf_{n\rightarrow\infty}\mathbb{P}\big(\Phi\inSAT\big).
    \end{aligned}
\end{equation*}
Combining these limits with the decomposition in (\ref{eq lower decomposition}) gives the result.
\end{proof}

\section{Proofs}
In this section we provide proofs of the lemmas stated previously. Sections \ref{section proof 1} and \ref{section proof 2} are devoted to the case $q>1/3$ and sections \ref{section proof 3} and \ref{section proof 4} concern the case $q<1/3$. Lastly, in section \ref{section proof 5} the technical lemmas of section \ref{lemma technical} are proven. 

\subsection{Proof of Lemma \ref{lemma over-constrained control size of M_1^r}} \label{section proof 1}

In this section, we again consider the elements defined in the unit-propagation procedure of section \ref{Decomposition in over-constrained regime}. We establish that the two events $B_u=\{M_1\rr{r}\leq C_0n^q\}$ and $B_l=\{M_1\rr{r}\geq c_0n^q\}$ happen w.h.p. A problem we encounter is that we cannot control the size of the sequence $\{S\rr{r}\}_{r\in[R]}$. Thus, we will need to define a new sequence of random elements that approximates our previously defined elements but for which we do not have this problem. Let $T\rr{-1}=0$, $N_1\rr{0}=\lfloor n^q\rfloor$ and define recursively for each $r\in[0,R]$
\begin{align*}
    T^{(r)}=\min\big\{T^{(r-1)}+(N_1\rr{r}-1)^+,\;\lceil n^{1-q}\log^{2}{n}\rceil \big\},\qquad \bar{T}\rr{r}=\lfloor\log{n}\rfloor\cdot T\rr{r},\\
    N_1^{(r+1)}|N_1^{(1)},\dotsc,N_1^{(r)}\sim \text{Binomial}\left(\big(n-\bar{T}\rr{r-2}\big)^+,\;p_1\big(n-S\rr{r-1},\big(M_1\rr{r-1}-1\big)^+\big)\right).
\end{align*}
Now, the sequence $\{T\rr{r}\}_{r\in[R]}$ is upper-bounded by $\lceil n^{1-q}\log^2{n}\rceil$ but at the same time it turns out that it has the same distribution as $\{S\rr{r}\}_{r\in[R]}$ w.h.p. Let $c_0>0$ and $C_0>0$ be two constants (which will be further specified later) and define the events
\begin{align*}
    D_l=\big\{N_1\rr{r}\geq c_0 n^q\;\forall r\in[R]\big\},\qquad &D_u=\big\{N_1\rr{r}\leq C_0 n^q\;\forall r\in[R]\big\}\\
    A_S=\big\{S\rr{R}< n^{1-q}\log^{2}{n}\big\}\qquad &A_T=\big\{T\rr{R}< n^{1-q}\log^{2}{n}\big\}. 
\end{align*}
Equation (\ref{eq over M_k^r dist}) implies that 
$$\big(\big\{M_1\rr{r}\mathds{1}_{A_S}\big\}_{r\in[R]},\;\mathds{1}_{A_S}\big)\eqdist \big(\big\{N_1\rr{r}\mathds{1}_{A_T}\big\}_{r\in[R]},\;\mathds{1}_{A_T}\big).$$
Moreover, on $A_S$ it holds that $\{M_1\rr{r}\mathds{1}_{A_S}\}_{r\in[R]}=\{M_1\rr{r}\}_{r\in[R]}$ and on $A_T$ we have
$\{N_1\rr{r}\mathds{1}_{A_T}\}_{r\in[R]}=\{N_1\rr{r}\}_{r\in[R]}$. Thus, if $B_l$ and $B_u$ are the events of Lemma \ref{lemma over-constrained control size of M_1^r}, then
\begin{align*}
    \mathbb{P}(B_l^c)\leq \mathbb{P}(B_l^c\cap A_S)+\mathbb{P}(A_S^c)= \mathbb{P}(D_l^c\cap A_T)+\mathbb{P}(A_S^c)\leq \mathbb{P}(D_l^c)+\mathbb{P}(A_S^c),
\end{align*}
and similarly $\mathbb{P}(B_u^c)\leq \mathbb{P}(D_u^c)+\mathbb{P}(A_S^c)$. Thus, in order to establish Lemma \ref{lemma over-constrained control size of M_1^r} it is sufficient to establish that $\lim_{n\rightarrow\infty}\mathbb{P}(A_S)=1$ and that $\lim_{n\rightarrow\infty}\mathbb{P}(D_l)=\lim_{n\rightarrow\infty}\mathbb{P}(D_u)=1$. Thus, proving Lemma \ref{lemma over-constrained control size of M_1^r} reduces to proving the below two lemmas
\begin{lemma} \label{lemma P(A_S)->1}
    We have $\lim_{n\rightarrow\infty}\mathbb{P}(A_S)=1$.
\end{lemma}
\begin{lemma} \label{lemma P(D_l)=P(D_u)->1}
    We have $\lim_{n\rightarrow\infty}\mathbb{P}(D_l)=\lim_{n\rightarrow\infty}\mathbb{P}(D_u)=1$.
\end{lemma}
We start by proving the first of the above two lemmas.
\begin{proof}[Proof of Lemma \ref{lemma P(A_S)->1}.]
    We will establish this using Markov's inequality. Note that
    $$\mathbb{E}[S\rr{R}]=\lfloor n^q\rfloor +\sum_{r=1}^R\mathbb{E}[(M_1\rr{r}-1)^+]\leq \sum_{r=0}^R\mathbb{E}[M_1\rr{r}], 
    $$
    and equation (\ref{eq over M_k^r dist}) and the definition of $p_1$ in Lemma \ref{lemma technical dist of M} implies
    \begin{align*}
        \mathbb{E}[M_1\rr{r}]&=\big(n-\bar{S}\rr{r-2}\big)^+\frac{\big(n-S\rr{r-2}-(M_1\rr{r-1}-1)^+\big)(M_1\rr{r-1}-1)^+}{\big(n-S\rr{r-2}\big)\big(n-S\rr{r-2}-1\big)}\\
        &\leq \mathbb{E}\bigg[M_1\rr{r-1}\cdot \frac{(n-\lfloor \log{n}\rfloor S\rr{r-2})^+}{n-S\rr{r-2}-1}\cdot \frac{n-S\rr{r-2}-(M_1\rr{r-1}-1)^+}{n-S\rr{r-2}}\bigg]\leq\mathbb{E}\big[M_1\rr{r-1}\big].
    \end{align*}
    In the above, we used that $S\rr{r-2}\geq n^q-1$ so $n-g(n)S\rr{r-2}\leq n-S\rr{r-2}-1$. Repeating the above argument we eventually get that 
    $$\mathbb{E}[M_1\rr{r}]\leq \mathbb{E}[M_1\rr{0}]\leq n^q.$$
Thus, Markov's inequality implies that
\begin{align*}
    \mathbb{P}(A_S^c)=\mathbb{P}\big(S\rr{R}\geq n^{1-q}\log^{2}{n}\big)\leq \frac{\mathbb{E}\big[S\rr{R}\big]}{n^{1-q}\log^{2}{n}}\leq \frac{(R+1)n^q}{n^{1-q}\log^{2}{n}}\leq \frac{(n^{1-2q}\log{n}+1)n^q}{n^{1-q}\log^2{n}}\rightarrow 0\;\text{ as }n\rightarrow\infty,
\end{align*}
where we use that for $q<1/2$ we have $1-q>q$. 
\end{proof}
To prove the next Lemma we need the below technical lemma.
\begin{lemma} \label{lemma mean and var bounds}
    Let $r,s\in[0,R]$ with $s<r$. Then
    \begin{enumerate}[label=(\arabic*)]
        \item $\mathbb{E}[N_1\rr{r}|N_1\rr{1},\dotsc,N_1\rr{s}]\leq N_1\rr{s},$
        \item $\mathbb{E}[(N_1\rr{r})^2|N_1\rr{1},\dotsc,N_1\rr{s}]\leq RN_1\rr{s}+(N_1\rr{s})^2$,
        \item Assume $N_1\rr{s}\leq C_0 n^q$ for some $C_0>0$. Then there exists $C_1 >0$ (dependent on $C_0$ but independent of $r$ and $s$) such that  
        $\mathbb{E}[N_1\rr{r}|N_1\rr{1},\dotsc,N_1\rr{s}]\geq N_1\rr{s}-C_1 n^{1-2q}\log^4{n}$
        \item Assume $c_0 n^q\leq N_1\rr{s}\leq C_0 n^q$ for some $c_0,C_0>0$. Then there exists $C_2 >0$ (dependent on $c_0$ and $C_0$ but independent of $r$ and $s$) such that $\mathbb{V}\big[N_1\rr{r}\big|N_1\rr{1},\dotsc,N_1\rr{s}\big]\leq C_2 n^{1-q}\log^4{n}.$
    \end{enumerate}
\end{lemma}
\begin{proof}
    The inequalities will be established one at a time.
    \vspace{0.2 cm}
\newline
    (1) Direct calculations give that
        \begin{align*}
        &\mathbb{E}\big[N_1\rr{r}\big|N_1\rr{1},\dotsc,N_1\rr{s}\big]\\
        =&\mathbb{E}\big[\mathbb{E}\big[N_1\rr{r}\big|N_1\rr{1},\dotsc,N_1\rr{r-1}\big]\big|N_1\rr{1},\dotsc,N_1\rr{s}\big]\\
        =&\mathbb{E}\bigg[(n-\bar{T}\rr{r-2})\frac{\big(n-T\rr{r-2}-(N_1\rr{r-1}-1)^+\big)(N_1\rr{r-1}-1)^+}{\big(n-T\rr{r-2}\big)\big(n-T\rr{r-2}-1\big)}\bigg|N_1\rr{1},\dotsc,N_1\rr{s}\bigg]\\
        \leq &\mathbb{E}\bigg[N_1\rr{r-1}\frac{(n-\lfloor\log{n}\rfloor T\rr{r-2})^+}{n-T\rr{r-2}-1}\cdot \frac{\big(n-T\rr{r-2}-(N_1\rr{r-1}-1)^+\big)}{n-T\rr{r-2}}\bigg|N_1\rr{1},\dotsc,N_1\rr{s}\bigg]\\
        \leq&\mathbb{E}\big[N_1\rr{r-1}\big|N_1\rr{1},\dotsc,N_1\rr{s}\big]\leq\dotsb\leq N_1\rr{s},
    \end{align*}
    where we in the first inequality use that $T\rr{r-2}\geq n^q-1$. 
    \vspace{0.2 cm}
\newline
    (2) For the second moment, we use that when $X\sim \text{Binomial}(n,p)$, then
\begin{equation} \label{eq second moment of binomial}
    \mathbb{E}[X^2]=np(1-p)+n^2p^2\leq \mathbb{E}[X]+\big(\mathbb{E}[X]\big)^2.
\end{equation}
This along with the calculations and result of (1) imply
\begin{align*}
    \mathbb{E}\big[(N_1\rr{r})^2\big|N_1\rr{1},\dotsc,N_1\rr{s}\big]=&\mathbb{E}\big[\mathbb{E}\big[(N_1\rr{r})^2\big|N_1\rr{1},\dotsc,N_1\rr{r-1}\big]\big|N_1\rr{1},\dotsc,N_1\rr{s}\big]\\
    \leq &\mathbb{E}\big[N_1\rr{r-1}\big|N_1\rr{1},\dotsc,N_1\rr{s}\big]+\mathbb{E}\big[(N_1\rr{r-1})^2\big|N_1\rr{1},\dotsc,N_1\rr{s}\big]\\
    \leq &\dotsb\leq (r-s)N_1\rr{s}+(N_1\rr{s})^2\\
    \leq & RN_1\rr{s}+(N_1\rr{s})^2.
\end{align*}
(3) Next, we want to find a lower bound on the mean. Here we use that $T\rr{r-2}\leq n^{1-q}\log^2{n}+1$ and we also make use of (1) and (2).
\begin{equation}
    \begin{aligned}
        \label{eq lower bound mean 0}
    &\mathbb{E}\big[N_1\rr{r}\big|N_1\rr{1},\dotsc,N_1\rr{s}\big]\\
    =&\mathbb{E}\bigg[(n-\bar{T}\rr{r-2})\frac{\big(n-T\rr{r-2}-(M_1\rr{r-1}-1)^+\big)(N_1\rr{r-1}-1)^+}{(n-T\rr{r-2})(n-T\rr{r-2}-1)}\bigg|N_1\rr{1},\dotsc,N_1\rr{s}\bigg]\\
    \geq & \mathbb{E}\bigg[\big(n-n^{1-q}\log^3{n}-\log{n}\big) \frac{n-\big(T\rr{r-2}+N_1\rr{r-1}\big)}{n-T\rr{r-2}}\cdot \frac{N_1\rr{r-1}-1}{n}\bigg|N_1\rr{1},\dotsc,N_1\rr{s}\bigg]\\
    \geq &\frac{n-n^{1-q}\log^3{n}-\log{n}}{n}\cdot \mathbb{E}\bigg[(N_1\rr{r-1}-1)\bigg(1-\frac{N_1\rr{r-1}}{n-n^{1-q}\log^2{n}-1}\bigg)\bigg|N_1\rr{1},\dotsc,N_1\rr{s}\bigg]\\
    \geq &\frac{n-n^{1-q}\log^3{n}-\log{n}}{n}\bigg(\mathbb{E}\big[N_1\rr{r-1}\big|N_1\rr{1},\dotsc,N_1\rr{s}\big]-1-\frac{\mathbb{E}\big[(N_1\rr{r-1})^2\big|N_1\rr{1},\dotsc,N_1\rr{s}\big]}{n-n^{1-q}\log^2{n}-1}\bigg)\\
    \geq &\frac{n-n^{1-q}\log^3{n}-\log{n}}{n}\cdot\mathbb{E}\big[N_1\rr{r-1}\big|N_1\rr{1},\dotsc,N_1\rr{s}\big]-1-\frac{RN_1\rr{s}+(N_1\rr{s})^2}{n-n^{1-q}\log^2{n}-1}\\
    \geq &\dotsb\geq \left(\frac{n-n^{1-q}\log^3{n}-\log{n}}{n}\right)^{r-s}N_1\rr{s}-(r-s)\bigg(1+\frac{RN_1\rr{s}+(N_1\rr{s})^2}{n-n^{1-q}\log^2{n}-1}\bigg).
    \end{aligned}
\end{equation}
We will now bound the above two terms one at a time. For the first term, we will need the below inequality which is true for $x\geq 2$ and $y>0$:
\begin{align*}
    \Big[\big(1-\tfrac{2}{x}\big)^x\Big]^y&=\Big[\exp\big(x\log\big(1-\tfrac{2}{x}\big)\big)\Big]^y\geq \Big[\exp\big(x\big(-\tfrac{4}{x}\big)\big)\Big]^y=\exp\big(-4y\big)\geq 1-4y.
\end{align*}
This and that $r-s\leq R\leq n^{1-2q}\log{n}$ now implies
\begin{equation} \label{eq lower bound mean 1}
    \left(\frac{n-n^{1-q}\log^3{n}-\log{n}}{n}\right)^{r-s}\geq \left(\left(1-\frac{2}{n^{q}\log^{-3}{n}}\right)^{n^{q}\log^{-3}{n}}\right)^{n^{1-3q}\log^{4}{n}}\geq 1-4n^{1-3q}\log^4{n}.
\end{equation}
For the other term, we use the assumption that $N_1\rr{s}\leq C_0n^q$. Then as $2q<1$ we get 
$$\frac{RN_1\rr{s}+(N_1\rr{s})^2}{n-n^{1-q}\log^2{n}-1}\leq\frac{C_0 n^{1-q}\log{n}+C_0^2n^{2q}}{n-n^{1-q}\log^2{n}-1} \rightarrow 0\quad\text{as }n\rightarrow\infty.$$
Thus, for a $C>1$ we have that
\begin{equation} \label{eq lower bound mean 2}
    (r-s)\left(1+\frac{RN_1\rr{s}+(N_1\rr{s})^2}{n-n^{1-q}\log^3{n}-1}\right)\leq (r-s)C \leq  C n^{1-2q}\log{n}.
\end{equation}
Now, combining (\ref{eq lower bound mean 0}), (\ref{eq lower bound mean 1}) and (\ref{eq lower bound mean 2}) along with the assumption that $N_1\rr{s}\leq C_0 n^q$ we get that 
\begin{align*}
    \mathbb{E}\big[N_1\rr{r}\big|N_1\rr{1},\dotsc,N_1\rr{s}\big]\geq& \big(1-4n^{1-3q}\log^4{n}\big)N_1\rr{s}-C n^{1-2q}\log{n}\\
    \geq & N_1\rr{s}-2n^{1-3q}\log^4{n}C_0 n^q-C n^{1-2q}\log{n}\geq N_1\rr{s}-C_1 n^{1-2q}\log^4{n}.
\end{align*}
(4)  Lastly, we combine (2) and (3) along with the extra assumption that $N_1\rr{s}\geq c_0 n^q$ (which implies that $N_1\rr{s}- C_1 n^{1-2q}\log^4{n}\geq 0$) to conclude that
\begin{align*}
    \mathbb{V}(N_1\rr{r}|N_1\rr{1},\dotsc,N_1\rr{s})&=\mathbb{E}\big[(N_1\rr{r})^2\big|N_1\rr{1},\dotsc,N_1\rr{s}\big]-\big(\mathbb{E}\big[N_1\rr{r}\big|N_1\rr{1},\dotsc,N_1\rr{s}\big]\big)^2\\
    &\leq RN_1\rr{s}+(N_1\rr{s})^2-\big(N_1\rr{s}-C_1 n^{1-2q}\log^4{n}\big)^2\\
    &\leq RN_1\rr{s}+(N_1\rr{s})^2-(N_1\rr{s})^2+2C_0n^qC_1 n^{1-2q}\log^4{n}\\
    &\leq C_2 n^{1-q}\log^{4}{n}.
    \end{align*}
\end{proof}
\begin{fact} \label{fact bound on mean and variance}
    As $N_1^{(0)}=\lfloor n^q\rfloor$ the above Lemma implies the existence of constants $c_1>0$ and $C_1 >0$ such that
    \begin{equation*}
        \begin{aligned}
            \mathbb{E}[N_1\rr{R}]\leq n^q+1,\qquad
        \mathbb{E}[N_1\rr{R}]\geq c_1 n^q,\qquad
        \mathbb{V}[N_1\rr{R}]\leq C_1 n^{1-q}\log^4{n}.
        \end{aligned}
    \end{equation*}
\end{fact}
We are now ready to prove the last lemma of this section which will imply Lemma \ref{lemma over-constrained control size of M_1^r}.
\begin{proof}[Proof of Lemma \ref{lemma P(D_l)=P(D_u)->1}.] 
Let $c_1>0$ and $C_1>0$ be the constants of Fact \ref{fact bound on mean and variance}. We let $c_0=c_1/2$ and $C_0=2C_1$ be the constants of our lemma. Note that when $q>1/3$ then $(1-q)/2<q$ why we can choose a $q_1\in\big(\tfrac{1-q}{2},\,q\big)$. Then using Chebyshev's inequality and Fact \ref{fact bound on mean and variance} we get
\begin{align*}
    \mathbb{P}\big(\big|N_1\rr{R}-\mathbb{E}[N_1\rr{R}]\big|\geq n^{q_1}\big)\leq \frac{\mathbb{V}\big[N_1\rr{R}\big]}{n^{2{q_1}}}\leq \frac{C_1 n^{1-q}\log^4{n}}{n^{2q_1}}\rightarrow 0\quad\text{as }n\rightarrow\infty.
\end{align*}
Fact \ref{fact bound on mean and variance} then implies
\begin{align}
    \mathbb{P}\big(N_1\rr{R}\leq c_1 n^q-n^{q_1}\big)\rightarrow 0\quad\text{as }n\rightarrow\infty, \label{eq prob lower bound}\\
    \mathbb{P}\big(N_1\rr{R}\geq C_1 n^q+n^{q_1}\big)\rightarrow 0\quad\text{as }n\rightarrow\infty \label{eq prob upper bound}.
\end{align}
The above implies that the sequence is still of order $n^q$ at step $R$. We will use this to show that the sequence cannot have been too small or too large in previous steps. Remember that we want to establish that $\lim_{n\rightarrow\infty}\mathbb{P}(D_l)=1$, where the complimentary event is given by
$$D_l^c=\big\{\exists r\in[R]\;\text{ s.t. }N_1\rr{r}< c_0n^q\big\}=\big\{\exists r\in[R]\;\text{ s.t. }N_1\rr{r}< \tfrac{1}{2}c_1 n^q\big\}.$$
Using (\ref{eq prob lower bound}) we see that the above is implied if we show that 
\begin{equation} \label{EQUATION}
    \liminf_{n\rightarrow\infty}\mathbb{P}\big(N_1\rr{R}\leq c_1 n^q-n^{q_1}\big|D_l^c\big)>0.
\end{equation}
Define
$$D_l\rr{r}=\big\{N_1\rr{1}>\tfrac{1}{2}c_1 n^q,\;N_1\rr{2}>\tfrac{1}{2}c_1 n^q,\dotsc,N_1\rr{r-1}>\tfrac{1}{2}c_1 n^q,\;N_1\rr{r}\leq \tfrac{1}{2}c_1 n^q\big\},\quad (r\in[R]).$$
Then the above events are disjoint and $D_l^c=\cup_{r\in[R]}D_l\rr{r}$. Using Markov's inequality we get
\begin{equation} \label{eq Markov cond on A_l^r}
    \mathbb{P}(N_1\rr{R}\leq c_1 n^q-n^{q_1}|D_l\rr{r})=1-\mathbb{P}(N_1\rr{R}>c_1 n^q-n^{q_1}|D_l\rr{r})\geq 1-\frac{\mathbb{E}[N_1\rr{R}|D_l\rr{r}]}{c_1 n^q-n^{q_1}},\quad (r\in[R]).
\end{equation}
Next, using Lemma \ref{lemma mean and var bounds} (1) we see
\begin{align*}
    \mathbb{E}[N_1\rr{R}|D_l\rr{r}]=\mathbb{E}\big[\mathbb{E}[N_1\rr{R}|N_1\rr{1},\dotsc,N_1\rr{r}]\big|D_l\rr{r}\big]\leq \mathbb{E}[N_1\rr{r}|D_l\rr{r}]\leq \tfrac{1}{2}c_1 n^q,\quad (r\in[R]).
\end{align*}
This upper bound is then inserted in (\ref{eq Markov cond on A_l^r}):
\begin{align*}
    \mathbb{P}(N_1\rr{R}\leq c_1 n^q-n^{q_1}|D_l\rr{r})\geq 1-\frac{\tfrac{1}{2}c_1 n^q}{c_1 n^q-n^{q_1}}\geq \frac{1}{4},\quad (r\in[R]).
\end{align*}
This finally implies that
\begin{align*}
    \mathbb{P}\big(N_1\rr{R}\leq c_1 n^q-n^{q_1}\big|D_l^c\big)&=\frac{\mathbb{P}\big(\big\{N_1\rr{R}\leq c_1 n^q-n^{q_1}\big\}\cap D_l^c\big)}{\mathbb{P}(D_l^c)}\\
    &=\frac{\sum_{r=1}^R\mathbb{P}\big(N_1\rr{R}\leq c_1 n^q-n^{q_1}\big|D_l\rr{r}\big)\mathbb{P}(D_l\rr{r})}{\mathbb{P}(D_l^c)}\\
    &\geq \frac{1}{4}\frac{\sum_{r=1}^R\mathbb{P}(D_l\rr{r})}{\mathbb{P}(D_l^c)}=\frac{1}{4},
\end{align*}
which is (\ref{EQUATION}). 

Next, we will establish that $\lim_{n\rightarrow\infty}\mathbb{P}(D_u)=1$, where the complimentary event is given by
$$D_u^c=\big\{\exists r\in[R]\;\text{ s.t. }N_1\rr{r}>C_0n^q\big\}=\big\{\exists r\in[R]\;\text{ s.t. }N_1\rr{r}>2C_1n^q\big\}.$$
Using (\ref{eq prob upper bound}) we get that this is implied if we can show that
\begin{align} \label{eq liminf A_u}
    \liminf_{n\rightarrow\infty}\mathbb{P}\big(N_1\rr{R}\geq C_1 n^q+n^{q_1}\big|D_u^c\big)>0.
\end{align}
Define
\begin{align*}
    D_u\rr{r}=\big\{N_1\rr{1}< \lfloor 2C_1n^q\rfloor,\,N_1\rr{2}< \lfloor 2C_1n^q\rfloor,\dotsc,N_1\rr{r-1}< \lfloor 2C_1n^q\rfloor,\;N_1\rr{r}\geq \lfloor 2C_1n^q\rfloor \big\},\quad(r\in[R]).
\end{align*}
Note that the above events are disjoint and $D_u^c\subseteq \cup_{r\in[R]}D_u\rr{r}$. Let $r\in[R]$ be fixed. The event $D_u\rr{r}$ does not give us an upper bound on $N_1\rr{r}$ which implies that we do not have good bounds on $\mathbb{E}[N_1\rr{R}|N_1\rr{1},\dotsc, N_1\rr{r}]$. Therefore, we split the below probability into two terms. Write
\begin{equation} \label{eq A_u^r}
    \begin{aligned}
        &\mathbb{P}\big(N_1\rr{R}<  C_1 n^q+n^{q_1}\big|D_u\rr{r}\big)\\
    = & \mathbb{P}\big(N_1\rr{R}< C_1 n^q+n^{q_1}\big|D_u\rr{r}\cap\{N_1\rr{r}< 2\lfloor 2C_1 n^q\rfloor\}\big)\mathbb{P}\big(N_1\rr{r}< 2\lfloor 2C_1 n^q\rfloor \big|D_u\rr{r}\big)\\
    +&\mathbb{P}\big(N_1\rr{R} <C_1 n^q+n^{q_1}\big|D_u\rr{r}\cap\{N_1\rr{r}\geq  2\lfloor 2C_1 n^q\rfloor \}\big)\mathbb{P}\big(N_1\rr{r}\geq 2\lfloor 2C_1 n^q \rfloor \big|D_u\rr{r}\big).
    \end{aligned}
\end{equation}
We will now consider the above two terms separately. For the first term, we now have a bound on $N_1\rr{r}$, as we condition on the event $\lfloor 2C_1 n^q\rfloor \leq  N_1\rr{r}<   2\lfloor 2C_1 n^q\rfloor$. However, we can not use Markov's inequality as before as our inequality points in the wrong direction. Thus, we will instead use Chebyshev's inequality. In Lemma \ref{lemma mean and var bounds} the bounds (3) and (4) imply that there exists a constant $C>0$ (which is independent of $r$) such that 
$$\mathbb{E}[N_1\rr{R}|N_1\rr{1},\dotsc,N_1\rr{r}]\geq  2C_1 n^q  -C n^{1-2q}\log^4{n}\quad\text{and}\quad \mathbb{V}(N_1\rr{R}|N_1\rr{1},\dotsc,N_1\rr{r})\leq C n^{1-q}\log^4{n}.$$
Then
$$\mathbb{E}[N_1\rr{R}|N_1\rr{1},\dotsc,N_1\rr{r}]-(C_1 n^q+n^{q_1})\geq C_1 n^q-C n^{1-2q}\log^4{n}-n^{q_1}>0,$$
and we can use Chebyshev's inequality to establish that
\begin{equation} \label{eq A_u^r number 1}
    \begin{aligned}
        &\mathbb{P}\big(N_1\rr{R}< C_1 n^q+n^{q_1}\big|N_1\rr{1},\dotsc,N_1\rr{r}\big)\\
    \leq & \mathbb{P}\big(\big|N_1\rr{R}-\mathbb{E}[N_1\rr{R}|N_1\rr{1},\dotsc,N_1\rr{r}]\big|> C_1 n^q-C n^{1-2q}\log^4{n}-n^{q_1}\big|N_1\rr{1},\dotsc,N_1\rr{r}\big)\\
    \leq &\frac{\mathbb{V}[N_1\rr{R}|N_1\rr{1},\dotsc,N_1\rr{r}]}{(C_1 n^q-C n^{1-2q}\log^4{n}-n^{q_1})^2}\\
    \leq& \frac{C n^{1-q}\log^4{n}}{(C_1 n^q-C n^{1-2q}\log^4{n}-n^{q_1})^2}\leq \frac{1}{4}.
    \end{aligned}
\end{equation}
Lastly, we used that the fraction is of order $n^{1-3q}\log^4{n}$ and thus it approaches zero as $n\rightarrow\infty$.  For the second term of (\ref{eq A_u^r}) we want to show that $\mathbb{P}(N_1\rr{r}\geq 2\lfloor 2C_1 n^q\rfloor |D_u\rr{r})$ is small. Note that $D_u\rr{r}$ contains the event $N_1\rr{r-1}< \lfloor 2C_1 n^q\rfloor $ which makes it unlikely that $N_1\rr{r}\geq 2\lfloor 2C_1 n^q\rfloor$. When $N_1\rr{r-1}< \lfloor 2C_1 n^q\rfloor$ we have:
\begin{equation} \label{eq A_u^r number 2}
    \begin{aligned}
        &\mathbb{P}\big(N_1\rr{r}\geq2\lfloor 2C_1 n^q\rfloor \big|N_1\rr{1},\dotsc,N_1\rr{r-1},\;N_1\rr{r}\geq \lfloor 2C_1 n^q\rfloor \big)\\
    = &\sum_{x=2\lfloor 2C_1n^q\rfloor}^n\mathbb{P}\big(N_1\rr{r}=x\big|N_1\rr{1},\dotsc,N_1\rr{r-1},\;N_1\rr{r}\geq  \lfloor 2C_1 n^q\rfloor \big)\\
    =&\sum_{x=2\lfloor 2C_1 n^q\rfloor}^n \frac{\mathbb{P}\big(N_1\rr{r}=x\big|N_1\rr{1},\dotsc,N_1\rr{r-1}\big)}{\mathbb{P}\big(N_1\rr{r}\geq  \lfloor 2C_1 n^q \rfloor\big|N_1\rr{1},\dotsc,N_1\rr{r-1}\big)}\\
    \leq &\sum_{x=2\lfloor 2C_1 n^q\rfloor}^n \frac{\mathbb{P}\big(N_1\rr{r}=2\lfloor 2C_1 n^q\rfloor \big|N_1\rr{1},\dotsc,N_1\rr{r-1}\big)}{\mathbb{P}\big(N_1\rr{r}= \lfloor 2C_1 n^q\rfloor \big|N_1\rr{1},\dotsc,N_1\rr{r-1}\big)}.
    \end{aligned}
\end{equation}
In the last inequality, we made the nominator larger and the denominator smaller. Here we used that $N_1\rr{r}|N_1\rr{1},\dotsc,N_1\rr{r-1}$ has a Binomial distribution and as $N_1\rr{r-1}<\lfloor 2C_1 n^q\rfloor$ the mode of its probability mass function is smaller than $2\lfloor 2C_1 n^q\rfloor$. Now, if $X\sim\text{Binomial}(n,p)$ then
\begin{align*}
    \frac{\mathbb{P}\big(X=2y\big)}{\mathbb{P}\big(X=y\big)}&=\frac{\binom{N}{2y}p^{2y}(1-p)^{n-2y}}{\binom{n}{k}p^{y}(1-p)^{n-y}}=\frac{y!}{(2y)!}\frac{(n-y)!}{(n-2y)!}(p)^{y}(1-p)^{-y}\\
    &\leq \frac{e\cdot y\cdot \big(\tfrac{y}{e}\big)^{y}}{e\cdot \big(\tfrac{2y}{e}\big)^{2y}}(np)^{y}\exp\big(-y\log(1-p)\big)\\
    &\leq \frac{y}{\big(\tfrac{4}{e}\big)^{y}y^{y}}(np)^{y}\exp\big(yp\big)=y\Big(\frac{4}{e}\Big)^{-y}\left(\frac{np}{y}\right)^y\exp(yp).
\end{align*}
In the above we have used that $\log(1-p)\geq -p$ and that $e\big(\tfrac{n}{e}\big)^{n}\leq n!\leq en\big(\tfrac{n}{e}\big)^n$. 
Furthermore, coupling this with (\ref{eq A_u^r number 2}) we note that in our setup we have that the number of trials equals $n-\lfloor \log{n}\rfloor T^{r-2}$ and the probability parameter equals
\begin{equation}
    p_1\big(n-T\rr{r-2},N_1\rr{r-1}\big)\leq \frac{N_1\rr{r-1}}{n-T\rr{r-2}}\leq Cn^{q-1},
\end{equation}
for a constant $C>0$ chosen large enough. This implies that the factor in the middle becomes less than one, so in (\ref{eq A_u^r number 2}) we get
\begin{align*}
    &\mathbb{P}\big(N_1\rr{r}>  2\lfloor 2C_1 n^q\rfloor \big|N_1\rr{1},\dotsc,N_1\rr{r-1},\;N_1\rr{r}\geq \lfloor 2C_1 n^q\rfloor \big)\\
    \leq &\sum_{x=2\lfloor 2C_1 n^q\rfloor}^n\frac{\mathbb{P}\big(N_1\rr{r}=2\lfloor 2C_1 n^{q}\rfloor \big|N_1\rr{1},\dotsc,N_1\rr{r-1}\big)}{\mathbb{P}\big(N_1\rr{r}= \lfloor 2C_1 n^q\rfloor \big|N_1\rr{1},\dotsc,N_1\rr{r-1}\big)}\\
    \leq & n2C_1 n^{q}\left(\frac{4}{e}\right)^{-2C_1 n^q}\exp\left(2C_1 Cn^{2q-1}\right)\leq \frac{1}{4}.
\end{align*}
Lastly, we used that $n^{1+q}(4/e)^{-2C_1 n^q}\rightarrow 0$ and $\exp(2C_1 Cn^{2q-1})\rightarrow 1$ when $n\rightarrow\infty$. Using this and (\ref{eq A_u^r number 1}) in (\ref{eq A_u^r}) we get
\begin{equation*}
    \begin{aligned}
        &\mathbb{P}\big(N_1\rr{R}\leq C_1 n^q+n^p\big|A_u\rr{r}\big)\\
    \leq & \mathbb{E}\big[\mathbb{P}\big(N_1\rr{R}\leq C_1 n^q+n^p\big|N_1\rr{1},\dotsc,\;N_1\rr{r}\big)\big|A_u\rr{r}\cap\{N_1\rr{r}\leq 2\lfloor 2C_1 n^q\rfloor\}\big)\\
    &+\mathbb{E}\big[\mathbb{P}\big(N_1\rr{r}> 2\lfloor 2C_1 n^q\rfloor \big|N_1\rr{1},\dotsc,N_1\rr{r-1},\;N_1\rr{r}\geq \lfloor 2C_1 n^q\rfloor\big)\big|A_u\rr{r}\big]\leq \frac{1}{4}+\frac{1}{4}=\frac{1}{2}. 
    \end{aligned}
\end{equation*}
This finally implies 
\begin{align*}
    \mathbb{P}\big(N_1\rr{R}>C_1 n^q+n^p\big|A_u\big)=&\frac{\mathbb{P}\big(\big\{N_1\rr{R}>C_1 n^q+n^p\big\}\cap A_u\big)}{\mathbb{P}(A_u)}\\
    =&\frac{\sum_{r=1}^R\mathbb{P}\big(N_1\rr{R}>C_1 n^q+n^p\big|A_u\rr{r}\big)\mathbb{P}(A_u\rr{r})}{\mathbb{P}(A_u\rr{r})}\\
    \geq& \frac{1}{2}\frac{\sum_{r=1}^R\mathbb{P}(A_u\rr{r})}{\mathbb{P}(A_u)}=\frac{1}{2},
\end{align*}
and this implies (\ref{eq liminf A_u}) and finishes the proof. 
\end{proof}

\subsection{Proof of Lemma \ref{lemma three limits}} \label{section proof 2}

The three limits of this lemma are established one at a time. Remember that we use the defined elements of Section \ref{Decomposition in over-constrained regime}, i.e. the unit-propagation procedure elements constructed for the regime with $q>1/3$. 

We will first establish that 
\begin{equation*}
    \lim_{n\rightarrow\infty}\mathbb{P}\big(M_0^{(r)}=0,\; r\in[R],\;B_u,\;B_l\big)=0.
\end{equation*}
\begin{proof}[Proof of Lemma \ref{lemma three limits} (1)]
    Let $c_0>0$ be the constant of Lemma \ref{lemma over-constrained control size of M_1^r} and $C_1>0$ be the constant of Fact \ref{fact bound on S^r}. Remember that $\bar{S}\rr{r}=\lfloor \log{n}\rfloor S\rr{r}$ for $r\in[R]$. As $n(4\lfloor \log{n}\rfloor )^{-1}> C_1n^{1-q}\log{n},$ we note that it is sufficient to establish that
    \begin{equation*}
\lim_{n\rightarrow\infty}\mathbb{P}\big(M_0^{(r)}=0,\;M_1^{(r)}\geq c_0 n^q,\; r\in[R],\;\bar{S}^{(R-2)}<\tfrac{n}{4} \big)= 0.
\end{equation*}
Recall that for $r\in[R]$ equation (\ref{eq over M_k^r dist}) gives that
$$M_0\rr{r}|\mathcal{M}\rr{r-1}\sim\text{Binomial}\bigg((n-\bar{S}\rr{r-2})^+,\;p_0\big(n-S\rr{r-2},\;(M_1\rr{r-1}-1)^+\big)\bigg),$$
and the function $p_0$, defined in Lemma \ref{lemma technical dist of M}, is given by
$$ p_0\big(n-S\rr{r-2},\;M_1\rr{r-1}\big)=\frac{(M_1\rr{r-1}-1)^+\big((M_1\rr{r-1}-1)^+-1\big)}{4(n-S\rr{r-2})(n-S\rr{r-2}-1)}.$$
Define i.i.d. random variables
$$X\rr{1},\dotsc,X\rr{R}\sim\text{Binomial}\bigg(\Big\lfloor \frac{n}{2}\Big\rfloor ,\;\Big(\frac{c_0n^q-2}{2n}\Big)^2\bigg).$$
The above considerations imply that
\begin{align*}
    &\mathbb{P}\big(M_0^{(r)}=0,\;M_1^{(r)}\geq c_0 n^q,\; r\in[R],\;\bar{S}^{(R-2)}<\tfrac{n}{4} \big)\\
    \leq &\mathbb{P}\big(M_0^{(r)}=0,\;r\in[R],\;M_1^{(r)}\geq c_0 n^q,\; r\in[R-1],\;\bar{S}^{(R-2)}<\tfrac{n}{4} \big)\\
    = & \mathbb{E}\Big[\mathbb{P}\big(M_0\rr{R}=0\big|\mathcal{M}\rr{r-1}\big)\mathds{1}_{\{M_0\rr{r}=0,\;M_1\rr{r}\geq c_0n^q,\;r\in[R-1],\;\bar{S}\rr{R-2}<\frac{n}{4}\}}\Big]\\
    \leq &\mathbb{E}\Big[\mathbb{P}\big(X\rr{R}=0\big)\mathds{1}_{\{M_0\rr{r}=0,\;M_1\rr{r}\geq c_0n^q,\;r\in[R-1],\;\bar{S}\rr{R-2}<\frac{n}{4}\}}\Big]\\
    \leq &\mathbb{P}\big(X\rr{R}=0\big)\mathbb{P}\big(M_0\rr{r}=0,\;M_1\rr{r}\geq c_0n^q,\;r\in[R-1],\;\bar{S}\rr{R-3}\leq \tfrac{n}{4}\big),
\end{align*}
where we lastly use that $\bar{S}\rr{r}$ increases in $r$. Now, the argument can be repeated on the last factor of the above upper bound. Eventually, we then derive that
\begin{align*}
    \mathbb{P}\big(M_0^{(r)}=0,\;M_1^{(r)}\geq c_0n^q\;\forall r\in[R],\;\bar{S}^{(R-2)}<\tfrac{n}{4}\big)\leq \prod_{r=1}^R\mathbb{P}(X\rr{r}=0).
\end{align*}
Let $c>0$ denote a constant satisfying 
$$\lfloor n/2\rfloor \geq cn,\quad R\geq cn^{1-2q}\log{n},\quad \left(\frac{c_0 n^q-2}{2}\right)^2\geq cn^{2q}.$$
Then
\begin{align*}
    \prod_{r=1}^R\mathbb{P}(X\rr{r}=0)&=\left(\left(1-\Big(\frac{c_0 n^q-2}{2n}\Big)^2\right)^{\lfloor n/2\rfloor }\right)^R\\
    &\leq \left(\left(1-\frac{c}{n^{2(1-q)}}\right)^{cn}\right)^{cn^{1-2q}\log{n}}\\
    &=\left(\left(1-\frac{c}{n^{2(1-q)}}\right)^{n^{2(1-q)}}\right)^{c^2\log{n}}\rightarrow 0\quad\text{as }n\rightarrow\infty,
\end{align*}
and this finishes the proof.
\end{proof}
The next limit that will be established is the following
$$\lim_{n\rightarrow\infty}\sum_{r=1}^R\mathbb{P}\big(M_1\rr{r}\geq  \bar{M}_1\rr{r}+2,\;B_u\big)= 0.$$
\begin{proof}[Proof of Lemma \ref{lemma three limits} (2)]
Let $C_0$ be the constant of Lemma \ref{lemma over-constrained control size of M_1^r} and $C_1$ be the constant of Fact \ref{fact bound on S^r}. As $n/2>C_1n^{1-q}\log{n}$ we note that it is sufficient to establish that
    $$\lim_{n\rightarrow\infty}\sum_{r=1}^R\mathbb{P}\big(M_1\rr{r}\geq\bar{M}_1\rr{r}+2,\;M_1\rr{r}\leq C_0n^q,\;S\rr{r-1}<\tfrac{n}{2}\big)=0.$$
    Let $r\in[R]$ be fixed and consider the conditional probability
    $$\mathbb{P}\big(M_1\rr{r}\geq\bar{M}_1\rr{r}+2,\;M_1\rr{r}\leq C_0n^q,\;S\rr{r-1}<\tfrac{n}{2}\big|\mathcal{M}\rr{r}\big).$$
    Let $V_1,\dotsc,V_{M_1\rr{r}}$ be the variables of the $1$-SAT formula $\Phi_1\rr{r}$. From (\ref{eq over Phi_1^r dist}) we get that these are i.i.d. and uniformly distributed on $[n-S\rr{r-1}]$ when conditioning on $\mathcal{M}\rr{r}$. Thus
\begin{align*}
    &\mathbb{P}\big(M_1^{(r)}\geq \bar{M}_1^{(r)}+2,\;M_1^{(r)}\leq C_0 n^q,\;S^{(r-1)}< \tfrac{n}{2}\;\big|\;\mathcal{M}\rr{r}\big)\\
    =&\mathbb{P}\bigg(\bigcup_{v_1,v_2}\;\;\bigcup_{\substack{j_1,j_2,j_3,j_4\\\text{distinct}}}\big\{V_{j_1}=V_{j_2}=v_1,\;V_{j_3}=V_{j_4}=v_2\big\}\bigg|\;\mathcal{M}\rr{r}\bigg)\mathds{1}_{\{M_1^{(r)}\leq C_0 n^q,\;S^{(r-1)}< \frac{n}{2}\}}\\
    \leq & \sum_{v_1,v_2}\;\;\sum_{\substack{j_1,j_2,j_3,j_4\\ \text{distinct}}}\;\;\mathbb{P}\big(V_{j_1}=V_{j_2}=v_1,\;V_{j_3}=V_{j_4}=v_2\;\big|\;\mathcal{M}\rr{r}\big)\mathds{1}_{\{M_1^{(r)}\leq C_0 n^q,\;S^{(r-1)}< \frac{n}{2}\}}\\
    \leq & (n-S^{(r-1)})^2(M_1^{(r)})^4\left(\frac{1}{n-S^{(r-1)}}\right)^4\mathds{1}_{\{M_1^{(r)}\leq C_0 n^q,\;S^{(r-1)}< \frac{n}{2}\}}\\
    \leq &\frac{(C_0 n^q)^4}{(n/2)^2}=\frac{C}{n^{2(1-2q)}},
\end{align*}
where we sum over $v_1,v_2\in[n-S^{(r-1)}]$ and $j_1,j_2,j_3,j_4\in [M_1^{(r)}]$ and $C=4C_0 ^4$. Therefore, we get
\begin{align*}
    &\sum_{r=1}^R\mathbb{P}\big(M_1^{(r)}\geq \bar{M}_1^{(r)}+2,\;M_1^{(r)}\leq C_0 n^q,\;S^{(r-1)}< \tfrac{n}{2}\big)\\
    =&\sum_{r=1}^R\mathbb{E}\big[\mathbb{P}\big(M_1^{(r)}\geq \bar{M}_1^{(r)}+2,\;M_1^{(r)}\leq C_0 n^q,\;S^{(r-1)}< \tfrac{n}{2}\;\big|\;\mathcal{M}\rr{r}\big)\big]\\
    \leq &\sum_{r=1}^R\frac{C}{n^{2(1-2q)}}\leq \frac{C\log{n}}{n^{1-2q}}\rightarrow 0\quad\text{as }n\rightarrow\infty,
\end{align*}
as $1-2q>0$ for $q<1/2$ and this was the claim. 
\end{proof}
The last limit to be established in this section is the following
$$\lim_{n\rightarrow\infty}\sum_{r=1}^R\mathbb{P}\big(M_2\rr{r}<(n-\bar{S}\rr{r-1})^+,\;B_u,\;B_l\big)= 0.$$
\begin{proof}[Proof of Lemma \ref{lemma three limits} (3)]
    Let $C_0$, $c_0$, and $C_1$ be the constants of Lemma \ref{lemma over-constrained control size of M_1^r} and Fact \ref{fact bound on S^r}. As $n(4\lfloor \log{n}\rfloor)^{-1}>C_1n^{1-q}\log{n}$ we note that it is sufficient to establish that
    $$\lim_{n\rightarrow\infty}\sum_{r=1}^R\mathbb{P}\big(M_2\rr{r}<(n-\bar{S}\rr{r-1})^+,\;M_1\rr{r-1}\leq C_0n^q,\;M_1\rr{r-1}\geq c_0n^q,\;\bar{S}\rr{r-2}<\tfrac{n}{4}\big)=0.$$
    Let $r\in[R]$ be fixed. We now consider the conditional distribution given $\mathcal{M}\rr{r-1}$ and assume that $M_1\rr{r-1}\leq C_0n^q$, $M_1\rr{r-1}\geq c_0n^q$ and $\bar{S}\rr{r-2}<n/4$ which also implies that $S\rr{r-2}<n/4$. Using the definition of equation (\ref{eq over M_k^r dist}), and the definition of $p_2$ in Lemma \ref{lemma technical dist of M}, we get
    \begin{equation} \label{eq over cond exp of M_2^r}
        \begin{aligned}
            \mathbb{E}[M_2\rr{r}|\mathcal{M}\rr{r-1}]=&\big(n-\bar{S}\rr{r-2}\big)p_2\big(n-S\rr{r-2},\,(M_1\rr{r-1}-1)^+\big)\\
            =& \big(n-\bar{S}\rr{r-2}\big)\frac{\big(n-S\rr{r-2}-(M_1\rr{r-1}-1)\big)\big(n-S\rr{r-2}-(M_1\rr{r-1}-1)-1\big)}{\big(n-S\rr{r-2}\big)\big(n-S\rr{r-2}-1\big)}\\
            = &\big(n-\bar{S}\rr{r-2}\big)\Big(1-\frac{M_1\rr{r-1}-1}{n-S\rr{r-2}}\Big)\Big(1-\frac{M_1\rr{r-1}-1}{n-S\rr{r-2}-1}\Big)\\
            \geq & \big(n-\bar{S}\rr{r-2}\big)\Big(1-\frac{C_0n^q}{n/2}\Big)^2\\
            \geq & \big(n-\bar{S}\rr{r-2}\big)\big(1-\tfrac{1}{2}Cn^{q-1}\big)^2,
        \end{aligned}
    \end{equation}
    where $C=4C_0$. Using the above we get
    \begin{equation}\label{eq over M_2^r mean - value}
        \begin{aligned}
            \mathbb{E}[M_2\rr{r}|\mathcal{M}\rr{r-1}]-\big(n-\bar{S}\rr{r-1}\big)=&\mathbb{E}[M_2\rr{r}|\mathcal{M}\rr{r-1}]-\Big(n-\bar{S}\rr{r-2}-\lfloor \log{n}\rfloor \big(M_1\rr{r-1}-1\big)\Big)\\
            \geq &\big(n-\bar{S}\rr{r-2}\big)\Big((1-\tfrac{1}{2}Cn^{q-1})^2-1\Big)+\lfloor\log{n}\rfloor \big(M_1\rr{r-1}-1\big)\\
            \geq &\frac{n}{2}\big(-Cn^{q-1}\big)+\lfloor \log{n}\rfloor \big(c_0n^q-1\big)\geq Cn^q,
        \end{aligned}
    \end{equation}
    The conditional variance can also be bounded (again when $M_1\rr{r-1}\leq C_0n^q$ and $\bar{S}\rr{r-1}<n/4$). To do so we again make use of (\ref{eq over M_k^r dist}) and the calculations in (\ref{eq over cond exp of M_2^r})
    \begin{equation} \label{eq over var bound M_2^r}
        \begin{aligned}
            \mathbb{V}(M_2^{(r)}|\mathcal{M}\rr{r-1})=&  \big(n-\bar{S}\rr{r-2}\big)p_2\big(n-S\rr{r-2},\,(M_1\rr{r-1}-1)^+\big)\Big(1-p_2\big(n-S\rr{r-2},\,(M_1\rr{r-1}-1)^+\big)\Big)\\
        \leq &n\big(1-p_2\big(n-S\rr{r-2},\,(M_1\rr{r-1}-1)^+\big)\big)\\
        \leq& n\big(1-\big(1-\tfrac{1}{2}Cn^{q-1}\big)^2\big)\leq Cn^{q}.
        \end{aligned}
    \end{equation}
    Using (\ref{eq over M_2^r mean - value}) and (\ref{eq over var bound M_2^r}) along with Chebyshev's inequality we get
\begin{equation*}
    \begin{aligned}
        \mathbb{P}\big(M_2^{(r)}<n-\bar{S}\rr{r-1}\big|\mathcal{M}\rr{r-1}\big)\leq& \mathbb{P}\left(\big|M_2^{(r)}-\mathbb{E}\big[M_2^{(r)}\big|\mathcal{M}\rr{r-1}\big]\big|>Cn^q\;\Big|\;\mathcal{M}\rr{r-1}\right)\\
        \leq& \frac{\mathbb{V}(M_2^{(r)}|{M\rr{r-1}})}{(Cn^q)^2}\leq \frac{1}{Cn^q}.
    \end{aligned}
\end{equation*}
This finally implies that
    \begin{align*}
    &\mathbb{P}\big(M_2\rr{r}<(n-\bar{S}\rr{r-1})^+,\;M_1\rr{r-1}\leq C_0n^q,\;M_1\rr{r-1}\geq c_0n^q,\;\bar{S}\rr{r-2}<\tfrac{n}{4}\big)\\
    = &\mathbb{E}\Big[\mathbb{P}\big(M_2^{(r)}<n-\bar{S}^{(r-1)}\big|\mathcal{M}\rr{r-1}\big)\mathds{1}_{\{M_1^{(r-1)}\leq C_0 n^q,\;M_1^{(r-1)}\geq c_0 n^q,\;\bar{S}^{(r-1)}< \frac{n}{4}\}}\Big]\\
    \leq &\frac{1}{Cn^q}\mathbb{P}\big(M_1^{(r-1)}\leq C_0 n^q,\;M_1^{(r-1)}\geq \lambda n^q,\;S^{(r-1)}< n/4\big)\leq \frac{1}{Cn^q},
\end{align*}
and thus we get the limit
\begin{align*}
    \sum_{r=1}^R\mathbb{P}\big(M_2\rr{r}&<(n-\bar{S}\rr{r-1})^+,\;M_1\rr{r-1}\leq  C_0n^q,\;M_1\rr{r-1}\geq c_0n^q,\;\bar{S}\rr{r-2}<\tfrac{n}{4}\big)\\
    \leq& \sum_{r=1}^R\frac{1}{Cn^q}\leq \frac{1}{C}\cdot n^{1-3q}\log{n}\rightarrow 0\quad\text{as }n\rightarrow\infty,
\end{align*}
where we use that $1-3q<0$ when $q>1/3$.
\end{proof}

\subsection{Proof of Lemma \ref{lemma five limits}} \label{section proof 3}
The four limits of this lemma are established one at a time. Remember that we use the defined elements of Section \ref{Decomposition in under-constrained regime}, i.e. the unit-propagation procedure elements constructed for the case $q<1/3$. 
To begin with we want to show that
$$\lim_{n\rightarrow\infty}\mathbb{P}\big(M_1\rr{r}\leq n^q\log{n},\;r\in[R]\big)=1.$$

\begin{proof}[Proof of Lemma \ref{lemma five limits} (1)]
For each $r\in[R]$ we use (\ref{eq under M_k^r dist}) and the definition of $p_1$ in Lemma \ref{lemma technical dist of M} to get that
\begin{align*}
    \mathbb{E}[M_1\rr{r}|\mathcal{M}\rr{r-1}]&=\big(n-S\rr{r-1}\big)p_1\big(n-S\rr{r-2},\;M_1\rr{r-1}\big)\\
    &=\frac{\big(n-S\rr{r-2}-M_1\rr{r-1}\big)^2}{\big(n-S\rr{r-2}\big)\big(n-S\rr{r-2}-1\big)}\cdot M_1\rr{r-1}\leq M_1\rr{r-1}.
\end{align*}
The last inequality is obviously true when $M_1\rr{r-1}\geq 1$ and when $M_1\rr{r-1}=0$ both sides of the inequality equals zero. Now, by letting $M_1\rr{r}=M_1\rr{R}$ and $\mathcal{M}\rr{r}=\mathcal{M}\rr{R}$ for $r>R$ we can extend our sequence and consider $\{M_1\rr{r}\}_{r\in\mathbb{N}_0}$ which then becomes a super-martingale w.r.t. the filtration $\{\mathcal{M}\rr{r}\}_{r\in\mathbb{N}_0}$. Define the stopping time
$$\tau=\min\{r\in\mathbb{N}_0\,:\,M_1\rr{r}=0\text{ or }M_1\rr{r}>n^q\log{n}\}.$$
Let $C_0$ be the constant of Lemma \ref{lemma under-constrained alternative result}. As our sequence $\{M_1\rr{r}\}_{r\in\mathbb{N}_0}$ is a non-negative super-martingale we can make use of the optional sampling theorem (Thm. 28, Chapter V in \cite{Dellacherie11}). Hereby we get that
$$C_0n^q\geq \mathbb{E}[M_1\rr{0}]\geq \mathbb{E}[M_1\rr{\tau}]\geq n^q\log{n}\cdot \mathbb{P}\big(M_1\rr{\tau}>n^q\log{n}\big).$$
Rearranging the above terms implies that 
$$\mathbb{P}\big(M_1\rr{\tau}>n^q\log{n}\big)\leq C_0\log^{-1}{n}.$$ 
As the sequence terminates when hitting zero this establishes the result.
\end{proof}

The next limit to establish is the following
$$\lim_{n\rightarrow\infty}\mathbb{P}\big(\Phi_1\rr{r}\inSAT,\;r\in[R]\,\big|\,M_1\rr{r}\leq n^q\log{n},\;r\in[-1,R]\big)= 1.$$

\begin{proof}[Proof of Lemma \ref{lemma five limits} (2)]
We will use that when $M_1\rr{r}\leq n^q\log{n}$ for all $r\in[-1,R]$ then $S\rr{r}\leq\tfrac{n}{2}$ for all $r\in[0,R]$. We will further use that if $X,Y,Z$ are random variables then
\begin{equation} \label{L3 independence relations}
    X\ind (Y,Z)\Rightarrow (X\ind Y)|Z,\quad\text{and}\quad (X\ind Y)|Z\Rightarrow X|(Y,Z)\eqdist X|Z.
\end{equation}
From (\ref{eq under Phi dist}) we got that
\begin{align*}
    \Phi_1\rr{r}\ind \Phi_2\rr{r}\;|\;\mathcal{M}\rr{r},\quad (r\in[R]).
\end{align*}
The random function $\Psi_2\rr{r}$ is constructed from $\Phi_1\rr{r}$ and $\Phi_2\rr{r}$ , but we noticed in (\ref{eq under ind of Phi and Psi}) that $\Psi_2\rr{r}$ and $\Phi_1\rr{r}$ are independent given $\mathcal{M}\rr{r}$. From this point and on all remaining random objects are constructed from $\Psi_2\rr{r}$ and from $\mathcal{L}\rr{r}$, which is deterministic given $\mathcal{M}\rr{r}$, and then also from random objects that are defined independently of $\Phi_1\rr{r}|\mathcal{M}\rr{r}$. This implies that
\begin{align*}
    \Phi_1\rr{r}\ind \big[(M_k\rr{r+1})_{k\in K},\dotsc,(M_k\rr{R})_{k\in K},\,\Phi_1\rr{r+1},\dotsc,\Phi_1\rr{R}\big]\;\big|\mathcal{M}\rr{r},\quad (r\in[R]). 
\end{align*}
Thus, the first implication of (\ref{L3 independence relations}) implies that
\begin{align*}
    \Phi_1\rr{r}\ind \big(\Phi_1\rr{r+1},\dotsc,\Phi_1\rr{R}\big)\big|\mathcal{M}\rr{R},\quad(r\in[R]),
\end{align*}
and the second implication of (\ref{L3 independence relations}) gives that 
$$\Phi_1\rr{r}|\mathcal{M}\rr{r}\eqdist \Phi_1\rr{r}|\mathcal{M}\rr{R},\quad(r\in[R]).$$
From (\ref{eq under Phi dist}) we have that
$$\Phi_1\rr{r}|\mathcal{M}\rr{r}\sim F_1\big(n-S\rr{r-1},\,M_1\rr{r}\big),\quad(r\in[R]),$$
and Lemma \ref{lemma 1-SAT} states that if $\Phi_1\sim F_1(n,m)$ then $\mathbb{P}\big(\Phi_1\inSAT\big)\geq \big(1-\frac{m}{n}\big)^m$. Thus, when $M_1\rr{r}\leq n^q\log{n}$ it holds that
\begin{align*}
    \mathbb{P}\big(\Phi_1\rr{r}\inSAT\big|\mathcal{M}\rr{r}\big)\geq \left(1-\frac{n^q\log{n}}{n/2}\right)^{n^q\log{n}}.
\end{align*}
Combining the above we get that
\begin{equation*}
    \begin{aligned}
        &\mathbb{P}\big(\Phi_1\rr{r}\inSAT,\; r\in[R]\,\big|\,M_1\rr{r}\leq n^q\log{n},\;r\in[-1,R]\big)\\
        =&\mathbb{E}\Big[\mathbb{P}\big(\Phi_1\rr{r}\inSAT,\; r\in [R]\big|\mathcal{M}\rr{R}\big)\Big|M_1\rr{r}\leq n^q\log{n},\; r\in[-1,R]\Big]\\
        =&\mathbb{E}\Big[\mathbb{P}\big(\Phi_1\rr{1}\inSAT\big|\mathcal{M}\rr{1}\big)\mathbb{P}\big(\Phi_1\rr{r}\inSAT,\; r\in[2,R]\big|\mathcal{M}\rr{R}\big)\Big|M_1\rr{r}\leq n^q\log{n},\; r\in[-1,R]\Big]\\
        \geq &\left(1-\frac{n^q\log{n}}{n/2}\right)^{n^q\log{n}}\mathbb{P}\big(\Phi_1\rr{r}\inSAT,\; r\in[2,R]\big|M_1\rr{r}\leq n^q\log{n},\; r\in[-1,R]\big)\\
        \geq&\dotsb\geq \left(\left(1-\frac{n^q\log{n}}{n/2}\right)^{n^q\log{n}}\right)^{R}\\
        \geq& \left(\left(1-\frac{2}{\frac{n^{1-q}}{\log{n}}}\right)^{\frac{n^{1-q}}{\log{n}}}\right)^{\frac{1}{\log{n}}}\rightarrow 1\quad\text{as }n\rightarrow\infty
    \end{aligned}
\end{equation*}
which was the claim. 
\end{proof}

Next up, we will establish that 
$$\lim_{n\rightarrow\infty}\mathbb{P}\big(M_0\rr{r}=0,\;r\in[R]\,\big|\,M_1\rr{r}\leq n^q\log{n},\;r\in[-1,R]\big)=1,$$

\begin{proof}[Proof of Lemma \ref{lemma five limits} (3)]
In (\ref{eq under M_k^r dist}) it is stated that
\begin{equation} \label{eq M_k^r in L2)}
    (M_k\rr{r})_{k\in K}\big|\mathcal{M}\rr{r-1}\sim\text{Multinomial}\Big(n-S\rr{r-1},\,{p}\big(n-S\rr{r-2},\,M_1\rr{r-1}\big)\Big),\quad(r\in[R]).
\end{equation}
We will use the following fact:
$$\text{If }(X_1,\dotsc,X_n)\sim \text{Multinomial}\big(n,(p_1,\dotsc,p_n)\big) \text{ then } X_i|X_j\sim\text{Binomial}\big(n-X_j,\,\tfrac{p_i}{1-p_j}\big)\text{ for }i\ne j.$$ 
This implies that
\begin{equation} \label{eq L2) M_0^r|M_1^r}
    M_0\rr{r}|\mathcal{M}\rr{r-1},\,M_1\rr{r}\sim\text{Binomial}\left(n-S\rr{r-1}-M_1\rr{r},\,\frac{p_0\big(n-S\rr{r-2},\,M_1\rr{r-1}\big)}{1-p_1\big(n-S\rr{r-2},\,M_1\rr{r-1}\big)}\right),\quad (r\in[R]).
\end{equation}
Now, given $M_1\rr{r}\leq n^q\log{n}$ for all $r\in[-1,R]$ we get that $S\rr{r}\leq n/4$ for all $r\in[0,R]$ and using the definitions of $p_0$ and $p_1$ given in Lemma \ref{lemma technical dist of M} we get for each $r\in[R]$:
\begin{equation} \label{eq L2 bound on p_0 and p_1}
    \begin{aligned}
    p_0\big(n-S\rr{r-1},\,M_1\rr{r-1}\big)&\leq \frac{(n^q\log{n})^2}{(\tfrac{3}{4}n)^2},\qquad
    1-p_1\big(n-S\rr{R-1},\,M_1\rr{R-1}\big)&\geq 1-\frac{n^q\log{n}}{\tfrac{3}{4}n}\geq \frac{3}{4}.
    \end{aligned}
\end{equation}
Using (\ref{eq under M_k^r dist}) we also note that there exists functions $g\rr{r}$, $r\in[R]$ such that for $r\in[R-1]$ we have that
\begin{equation} \label{eq L2) independence}
    \begin{aligned}
        &\mathbb{P}\big(M_1\rr{s}=m\rr{s},\;s\in[r+1,R]\big|\mathcal{M}\rr{r-1},\,M_1\rr{r},\,M_0\rr{r}\big)\\
        =&\mathbb{E}\big[\mathbb{P}\big(M_1\rr{R}=m\rr{R}\big|\mathcal{M}\rr{R-1}\big)\mathds{1}_{\{M_1\rr{s}=m\rr{s},\;s\in[r+1,R-1]\}}\big|\mathcal{M}\rr{r-1},\,M_1\rr{r},\,M_0\rr{r}\big]\\
        =&g\rr{R}(M_1\rr{-1},\dotsc,M_1\rr{r},m\rr{r+1},\dotsc,m\rr{R})\mathbb{P}\big(M_1\rr{s}=m\rr{s},\;s\in[r+1,R-1]\big|\mathcal{M}\rr{r-1},\,M_1\rr{r},\,M_0\rr{r}\big)\\
        =&\dotsb=\prod_{s=r+1}^Rg\rr{s}\big(M_1\rr{-1},\dotsc,M_1\rr{r},m\rr{r+1},\dotsc,m\rr{s}\big).
    \end{aligned}
\end{equation}
This implies that $(M_1\rr{r+1},\dotsc,M_1\rr{R})$ is independent of $M_0\rr{r}$ when conditioning on $\mathcal{M}\rr{r-1}$ and $M_1\rr{r}$. Now using (\ref{eq L2) M_0^r|M_1^r}) and (\ref{eq L2 bound on p_0 and p_1}) we get that
\begin{equation*} 
    \begin{aligned}
        &\mathbb{P}\big(M_0\rr{r}=0,\; r\in[R]\,\big|\,M_1\rr{r}\leq n^q\log{n},\; r\in[-1,R]\big)\\
    =&\mathbb{E}\Big[\mathbb{P}\big(M_0\rr{R}=0\big|\mathcal{M}\rr{R-1},\,M_1\rr{R}\big)\mathds{1}_{\{M_0\rr{r}=0,\; r\in[R-1]\}}\,\Big|\,M_1\rr{r}\leq n^q\log{n},\; r\in[-1,R]\Big]\\
    \geq &\mathbb{E}\bigg[\Big(1-\frac{(n^q\log{n})^2}{(\tfrac{1}{2})^3n^2}\Big)^{n}\mathds{1}_{\{M_0\rr{r}=0,\; r\in[R-1]\}}\bigg|M_1\rr{r}\leq n^q\log{n},\; r\in[-1,R]\bigg]\\
    =&\left(1-\frac{(n^q\log{n})^2}{(\tfrac{1}{2})^3n^2}\right)^{n}\mathbb{P}\big(M_0\rr{r}=0,\; r\in[R-1]\big|M_1\rr{r}\leq n^q\log{n},\; r\in[-1,R]\big).
    \end{aligned}
\end{equation*}
Then using (\ref{eq L2) independence}) the above argument can be repeated on the last factor above
\begin{equation*}
    \begin{aligned}
        &\mathbb{P}\big(M_0\rr{r}=0,\; r\in[R-1]\big|M_1\rr{r}\leq n^q\log{n},\; r\in[-1,R]\big)\\
        =&\mathbb{E}\big[\mathbb{P}\big(M_0\rr{R-1}=0\big|\mathcal{M}\rr{R-2},\;M_1\rr{R-1},\;M_1\rr{R}\big)\mathds{1}_{\{M_0\rr{r}=0,\;r\in[R-2]\}}\big|M_1\rr{r}\leq n^q\log{n},\; r\in[-1,R]\big]\\
        =&\mathbb{E}\big[\mathbb{P}\big(M_0\rr{R-1}=0\big|\mathcal{M}\rr{R-2},\;M_1\rr{R-1}\big)\mathds{1}_{\{M_0\rr{r}=0,\;r\in[R-2]\}}\big|M_1\rr{r}\leq n^q\log{n},\; r\in[-1,R]\big]\\
        \geq &\left(1-\frac{(n^q\log{n})^2}{(\tfrac{1}{2})^3n^2}\right)^{n}\mathbb{P}\big(M_0\rr{r}=0,\; r\in[R-2]\big|M_1\rr{r}\leq n^q\log{n},\; r\in[-1,R]\big).
    \end{aligned}
\end{equation*}
Repeating the above $R$ times in total we eventually arrive at the expression
\begin{equation*}
\begin{aligned}
    &\mathbb{P}\big(M_0\rr{r}=0,\; r\in[R]\,\big|\,M_1\rr{r}\leq n^q\log{n},\; r\in[-1,R]\big)\\
    \geq&\left(1-\frac{(n^q\log{n})^2}{(\tfrac{1}{2})^3n^2}\right)^{Rn}\geq \left(\left(1- \frac{2^3}{\frac{n^{2(1-q)}}{\log^2{n}}}\right)^{\frac{n^{2(1-q)}}{\log^2{n}}}\right)^{\frac{1}{\log{n}}}\rightarrow 1,\quad\text{as }n\rightarrow\infty,
    \end{aligned}
\end{equation*}
which was the claim.
\end{proof}

Next, we will establish that
$$\mathbb{P}\big(M_1\rr{R}=0\big)\rightarrow 1\quad\text{as }n\rightarrow\infty,$$
i.e. we will now establish that our process of $1$-clauses terminates in less than $R$ rounds w.h.p. We will show this by proving that our recursive sequence of Binomial random variables can be approximated by a recursive sequence of Poisson random variables. Afterwards, it is proven that the recursive sequence of Poisson random variables terminates. 

\begin{lemma} \label{lemma Poisson approximation sequence}
    Let $(X\rr{r})_{r\in[-1,R]}$ be a sequence of random variables where $X\rr{-1}=M_1\rr{-1}$, $X\rr{0}=M_1\rr{0}$ and $X\rr{r}|X\rr{r-1}\sim\text{Poisson}\big(X\rr{r-1}\big)$ for $r\in[R].$ Then for $x\rr{-1},x\rr{0},\dotsc,x\rr{R}\in[0,\lfloor n^q\log{n}\rfloor]$ it holds that
    \begin{align*}
        &\mathbb{P}\big(M_1\rr{r}=x\rr{r},\; r\in[R]\big|M_1\rr{-1}=x\rr{-1},\,M_1\rr{r}=x\rr{0}\big)\\
    \geq& \mathbb{P}\big(X\rr{r}=x\rr{r},\; r\in[R]\big|X\rr{-1}=x\rr{-1},X\rr{0}=x\rr{0}\big)\cdot E(n),
    \end{align*}
    where $E$ is a function satisfying that $\lim_{n\rightarrow\infty}E(n)=1$. 
\end{lemma}
\begin{proof}
    We will make use of the below inequality which holds for $y>x>0$ and $z>0$:
    $$\left(1-\frac{x}{y}\right)^{y-z}\geq \left(1-\frac{x}{y}\right)^y=\exp\left(y\log\left(1-\frac{x}{y}\right)\right)\geq \exp\left(y\Big(-\frac{x}{y}-\frac{x^2}{y^2}\Big)\right)=\exp\left(-x\right)\exp\left(-\frac{x^2}{y}\right).$$
    Let $s\rr{-2}=0$ and $s\rr{r}=s\rr{r-1}+x\rr{r}$ for $r\in[-1,R]$. The elements $x\rr{-1},\dotsc,x\rr{R}$ are chosen such that $s\rr{r}\leq \tfrac{n}{2}$ for all $r\in[-1,R]$. Also, note that the definition of $p_1$ in Lemma \ref{lemma technical dist of M} implies that
    $$\frac{(n-f)f}{n^2}\leq p_1(n,f)\leq \frac{f}{n},\quad (n\geq f\geq 0).$$ 
    Using the above inequalities along with Lemma \ref{lemma technical dist of M} we now get
    \begin{equation*}
        \begin{aligned}
        &\mathbb{P}\big(M_1\rr{r}=x\rr{r},\;r\in[R]\big|M_1\rr{-1}=x\rr{-1},\,M_1\rr{0}=x\rr{0}\big)\\
        =&\prod_{r=1}^R\mathbb{P}\big(M_1\rr{r}=x\rr{r}\big|M_1\rr{s}=x\rr{s},\; s\in[-1,r-1]\big)\\
        =&\prod_{r=1}^R\binom{n-s\rr{r-1}}{x\rr{r}}\big[p_1(n-s\rr{r-2},\,x\rr{r-1})\big]^{x\rr{r}}\big[1-p_1(n-s\rr{r-2},\,x\rr{r-1})\big]^{n-{s\rr{r-1}}-x\rr{r}}\\
        \geq &\prod_{r=1}^R\frac{(n-s\rr{r-1}-x\rr{r})^{x\rr{r}}}{x\rr{r}!}\left[ \frac{(n-s\rr{r-1})x\rr{r-1}}{(n-s\rr{r-2})^2}\right]^{x\rr{r}}\left[1-\frac{x\rr{r-1}}{n-s\rr{r-2}}\right]^{n-s\rr{r-2}-(x\rr{r-1}+x\rr{r})}\\
        \geq & \prod_{r=1}^R\left[\frac{n-s\rr{r-2}-(x\rr{r-1}+x\rr{r})}{n-s\rr{r-2}}\right]^{2x\rr r}\exp\left(-\frac{(x\rr{r-1})^2}{n-s\rr{r-2}}\right)\frac{e^{-x\rr{r-1}}(x\rr{r-1})^{x\rr r}}{x\rr r!}\\
        \geq &\prod_{r=1}^R\left(1-\frac{2n^q\log{n}}{n/2}\right)^{2n^q\log{n}}\exp\left(-\frac{n^{2q}\log^2{n}}{n/2}\right)\mathbb{P}\big(X\rr{r}=x\rr{r}|X\rr{r-1}=x\rr{r-1}\big)\\
        \geq & \bigg(\bigg(1-\frac{4}{\frac{n^{1-q}}{\log{n}}}\bigg)^{\frac{n^{1-q}}{\log{n}}}\bigg)^{\frac{2}{\log{n}}}\exp\left(-\frac{2}{\log{n}}\right)\mathbb{P}\big(X\rr r=x\rr r,\; r\in[R]\big|X\rr{-1}=x\rr{-1},X\rr 0=x\rr 0\big).
        \end{aligned}
    \end{equation*}
    And as
    $$ \bigg(\bigg(1-\frac{4}{\frac{n^{1-q}}{\log{n}}}\bigg)^{\frac{n^{1-q}}{\log{n}}}\bigg)^{\frac{2}{\log{n}}}\exp\left(-\frac{2}{\log{n}}\right)\rightarrow 1\quad\text{as }n\rightarrow\infty,$$
    the result follows. 
\end{proof}
Let $(X\rr{r})_{r\in[-1,R]}$ be the sequence of random variables from the above lemma. 
    Note that
    \begin{align*}
        \mathbb{P}\big(M_1\rr{R}=0\big)&\geq \mathbb{P}\big(M_1\rr{R}=0,\,M_1\rr{r}\leq n^q\log{n},\; r\in[-1,R]\big)\\
        &=\sum_{x\rr{-1}=0}^{\lfloor n^q\log{n}\rfloor }\dotsb\sum_{x\rr{R-1}=0}^{\lfloor n^q\log{n}\rfloor }\mathbb{P}\big(M_1\rr{-1}=x\rr{-1},\dotsc,M_1\rr{R-1}=x\rr{R-1},\,M_1\rr{R}=0\big),
    \end{align*}

and using Lemma \ref{lemma Poisson approximation sequence} and letting $x\rr{R}=0$ each summand can be upper bounded by
    \begin{align*}
        &\mathbb{P}\big(M_1\rr{-1}=x\rr{-1},\dotsc,M_1\rr{R}=x\rr{R}\big)\\
        =&\mathbb{P}\big(M_1\rr{r}=x\rr{r},\;r\in[R]\big|M_1\rr{-1}=x\rr{-1},M_1\rr{0}=x\rr{0}\big)\mathbb{P}\big(M_1\rr{-1}=x\rr{-1},M_1\rr{0}=x\rr{0}\big)\\
        \geq&\mathbb{P}\big(X\rr{r}=x\rr{r},\;r\in[R]\big|X\rr{-1}=x\rr{-1},X\rr{0}=x\rr{0}\big)\mathbb{P}\big(X\rr{-1}=x\rr{-1},X\rr{0}=x\rr{0}\big) E(n)\\
        =&\mathbb{P}\big(X\rr{-1}=x\rr{-1},\dotsc,X\rr{R}=x\rr{R}\big)E(n).
    \end{align*}
    Inserting this lower bound in the sum gives that
    \begin{equation} \label{eq M_1^R=0 if}
        \mathbb{P}\big(M_1\rr{R}=0\big)\geq E(n) \mathbb{P}\big(X\rr R=0,\;X\rr r\leq n^q\log{n},\; r\in[-1,R-1]\big),
    \end{equation}
    where $E$ is the function from Lemma \ref{lemma Poisson approximation sequence}. To establish our result we thus only need the two lemmas below
\begin{lemma} \label{lemma M_1^R=0 (1)}
        We have that
        $$\mathbb{P}\big(X\rr{r}\leq  n^q\log{n},\;r\in[-1,R-1]\big)\rightarrow 1\quad\text{as }n\rightarrow\infty.$$
    \end{lemma}
    \begin{lemma} \label{lemma M_1^R=0 (2)}
        We have that 
        $$\mathbb{P}\big(X\rr{R}=0\big)\rightarrow 1\quad\text{as }n\rightarrow\infty.$$
    \end{lemma}
\begin{proof}[Proof of Lemma \ref{lemma M_1^R=0 (1)}.] Let $X\rr{r}:=X\rr{R}$ for $r>R$ and also define the $\sigma$-algebras $\mathcal{F}\rr{r}=\sigma\big(X\rr{-1},\dotsc,X\rr{r}\big)$ for $r\geq -1$. Then for each $r\geq -1$ we have
$$\mathbb{E}[X\rr{r}|\mathcal{F}\rr{r-1}]=X\rr{r-1},$$
why $(X\rr{r})_{r\geq -1}$ is a martingale w.r.t. the filtration $(\mathcal{F}\rr{r})_{r\geq -1}$. As it is non-negative, we can make use of optional sampling (Thm. 28, Chapter V of \cite{Dellacherie11}). Consider the stopping time
$$\tau=\min\{r\in\mathbb{N}_0\,:\,X\rr{r}=0\text{ or }X\rr{r}\geq n^q\log{n}\},$$
and let $C_0$ be the constant of Lemma \ref{lemma under-constrained alternative result}. Then
$$C_0n^q\geq \mathbb{E}[X\rr{0}]\geq \mathbb{E}\big[X\rr{\tau}\big]\geq n^q\log{n}\mathbb{P}\big(X\rr \tau\geq n^q\log{n}\big)\Rightarrow \mathbb{P}\big(X\rr \tau\geq n^q\log{n}\big)\leq \frac{C_0}{\log{n}}.$$
As $0$ is an absorbing state this implies that
$$\mathbb{P}\big(X\rr{r}\leq n^q\log{n},\;r\in[-1,R]\big)\rightarrow 1\quad\text{as }n\rightarrow\infty,$$
which was the claim.
\end{proof}

\begin{proof}[Proof of Lemma \ref{lemma M_1^R=0 (2)}.]
    Note that the distribution of $(X\rr{r})_{r\in[R]}$ has the same law as a critical Galton-Watson tree with offspring distribution $Poisson(1)$ cut off at depth $R$, see Chapter 1 in \cite{Athreya04}. Thus, using standard results for such processes (see e.g. Thm. 1 in section 1.9 of \cite{Athreya04}), there exists a constant $C>0$ such that 
    $$\mathbb{P}(X\rr{R}=0)\geq \mathbb{E}\bigg[\Big(1-\frac{C}{R}\Big)^{X\rr{0}}\bigg]\geq \Big(1-\frac{C}{n^{1-2q}\log^{-3}{n}}\Big)^{\mathbb{E}[X\rr{0}]}\geq \Big(1-\frac{C n^{3q-1}\log^3{n}}{n^{q}}\Big)^{C_0n^q}\rightarrow 1,\quad\text{as }n\rightarrow\infty,$$ 
    where $C_0$ is the constant of Lemma \ref{lemma under-constrained alternative result} and Jensen's inequality is also used. Thus, the result is established. 
\end{proof}

\subsection{Establishing Definition \ref{def degrees of freedom} (1) using Lemma \ref{lemma under-constrained alternative result}} \label{section proof 4}
We will now couple Lemma \ref{lemma under-constrained alternative result} to our main result in the regime $q<1/3$. We will do this by closely controlling the first couple of rounds in the unit-propagation algorithm. Thus, we will once again repeat the notation used when going through this procedure. However, as this section uses none of the defined elements from the other sections this will not be a problem. 

Let $\Phi\sim F_2(n,n)$ and let $\mathcal{L}\subseteq\pm[n]$ be a consistent set of literals with $|\mathcal{L}|=\lfloor n^q\rfloor$. We need to show that $\liminf_{n\rightarrow\infty}\mathbb{P}(\Phi_\mathcal{L}\inSAT)=\liminf_{n\rightarrow\infty}\mathbb{P}(\Phi\inSAT)$. Let $G$ be the function of Lemma \ref{lemma technical G} and define \begin{align*}
    \Psi_2\rr{0}:=G(&\Phi,\mathcal{L}),\quad N_1\rr{0}=\lfloor n^q\rfloor,\quad T\rr{-1}:=0,\\
    T\rr{0}:=\lfloor n^q\rfloor,&\quad \mathcal{L}\rr{0}:=[n]\backslash[n-T\rr{0}],\quad \mathcal{N}\rr{0}=\{\emptyset,\Omega\}.
\end{align*}
Note that $\Psi_2\rr{0}\sim F_2(n,n)$ and also 
\begin{equation} \label{eq last decomp 1}
    \big\{\Phi_\mathcal{L}\inSAT\big\}=\big\{(\Psi_2\rr{0})_{\mathcal{L}\rr{0}}\inSAT\big\}.
\end{equation}
Unlike previously we now only repeat the unit-propagation procedure twice. Thus, recursively for $r=1,2$ define the following:

Let $G_1$ and $G_2$ be the functions of Lemma \ref{lemma technical dist of M} and define $\Phi_k\rr{r}=G_k(\Psi_k\rr{r-1},\mathcal{L}\rr{r-1})$. Let $N_k\rr{r}$ be the number of clauses in $\Phi_k\rr{r}$ for $k\in\{1,2\}$ and let $N_0\rr{r}$ and $N_\star\rr{r}$ be the number of unsatisfied- and satisfied clauses of $(\Psi_2\rr{r-1})_{\mathcal{L}\rr{r-1}}$, respectively. Define the $\sigma$-algebra $\mathcal{N}\rr{r}=\sigma\big(\mathcal{N}\rr{r-1}\cup\sigma(N_k\rr{r}:k\in K)\big)$. The elements are constructed such that
\begin{equation} \label{eq last decomp 2}
        \begin{aligned}
            \big\{({\Psi}_2\rr{r-1})_{\mathcal{L}\rr{r-1}}\inSAT\big\}=\big\{({\Phi}_2\rr{r})_{{\Phi}_1\rr{r}}\inSAT,\;{\Phi}_1\rr{r}\inSAT,\;N_0\rr{r}=0\big\}.
        \end{aligned}
    \end{equation}
and also 
\begin{align}
    (N_k\rr{r})_{k\in K}|\mathcal{N}\rr{r-1}&\sim\text{Multinomial}\Big(n-T\rr{r-2},\;{p}\big(n-T\rr{r-2},\,N_1\rr{r-1}\big)\Big),\label{eq N_k^r dist}\\
    {\Phi}_k\rr{r}|\mathcal{N}\rr{r}&\sim F_k(n-T\rr{r-1},\;N_k\rr{r}),\quad (k\in\{1,2\}),\notag
\end{align}
and $\Phi_1\rr{r}$ and $\Phi_2\rr{r}$ are conditionally independent. Let $\bar{N}_1\rr{r}$ be the number of distinct variables appearing in $\Phi_1\rr{r}$ and define further
$$T\rr{r}:=T\rr{r-1}+N_1\rr{r},\quad \bar{\mathcal{L}}\rr{r}:=[n-T\rr{r-1}]\backslash[n-T\rr{r-1}-\bar{N}_1\rr{r}],\quad \mathcal{L}\rr{r}:=[n-T\rr{r-1}]\backslash[n-T\rr{r}].$$ 
Also, let $\Psi_2\rr{r}:=G(\Phi_2\rr{r},\;\mathcal{L}(\Phi_1\rr{r}))$, where again $G$ is defined in Lemma \ref{lemma technical G}. Then using Lemma \ref{lemma technical dist of M} we see
\begin{align}
    \Psi_2\rr{r}|\mathcal{N}\rr{r}&\sim F_2\big(n-T\rr{r-1},\;N_2\rr{r}\big), \notag\\
    \big\{(\Phi_2\rr{r})_{\Phi_1\rr{r}}\inSAT\big\}&=\big\{(\Psi_2\rr{r})_{\bar{\mathcal{L}}\rr{r}}\inSAT\big\}\supseteq \big\{(\Psi_2\rr{r})_{{\mathcal{L}}\rr{r}}\inSAT\big\}, \label{eq last decomp 3}
\end{align}
where we lastly used that $\bar{\mathcal{L}}\rr{r}\subseteq\mathcal{L}\rr{r}$. Now, we are in the same setup as initially and our recursive step has ended. Combining (\ref{eq last decomp 1}), (\ref{eq last decomp 2}), and (\ref{eq last decomp 3}), we now see
\begin{equation} \label{eq decomposition lower bound real case}
    \mathbb{P}\big(\Phi_\mathcal{L}\inSAT\big)\geq \mathbb{P}\big((\Psi_2\rr{2})_{\mathcal{L}\rr{2}}\inSAT,\;\Phi_1\rr{2}\inSAT,\;\Phi_1\rr{1}\inSAT,\;N_0\rr{2}=0,\;N_0\rr{1}=0\big).
\end{equation}
The above equation implies that it is sufficient to lower bound the right-hand side of the above expression to establish our main theorem in the case $q<1/3$. We do this by proving the below lemmas.
    \begin{lemma} \label{lemma final lower 2}
        We have
    $$\lim_{n\rightarrow\infty}\mathbb{P}\big(N_0\rr{1}=0\big)=\lim_{n\rightarrow\infty}\mathbb{P}\big(N_0\rr{2}=0\big)=1.$$
    and
    $$\lim_{n\rightarrow\infty}\mathbb{P}\big(\Phi_1\rr{1}\inSAT\big)=\lim_{n\rightarrow\infty}\mathbb{P}\big(\Phi_1\rr{2}\inSAT\big)=1.$$
    \end{lemma}
    \begin{lemma} \label{lemma final lower 3}
    We have
    $$\liminf_{n\rightarrow\infty}\mathbb{P}\big((\Psi_2\rr{2})_{\mathcal{L}\rr{2}}\inSAT\big)\geq \liminf_{n\rightarrow\infty}\mathbb{P}\big(\Phi\inSAT\big).$$
    \end{lemma}
These lemmas will imply our main theorem when $q<1/3$
\begin{proof}[Proof of Definition \ref{def degrees of freedom} (1)]
    As $\{\Phi_\mathcal{L}\inSAT\}\subseteq \{\Phi\inSAT\}$ this implies that
    $$\liminf_{n\rightarrow\infty}\mathbb{P}\big(\Phi_\mathcal{L}\inSAT\big)\leq \liminf_{n\rightarrow\infty}\mathbb{P}\big(\Phi\inSAT\big).$$
    On the other hand, equation (\ref{eq decomposition lower bound real case}) along with Lemma \ref{lemma final lower 2} and Lemma \ref{lemma final lower 3} gives
    \begin{align*}
        &\liminf_{n\rightarrow\infty}\mathbb{P}\big(\Phi_\mathcal{L}\inSAT\big)\\
        \geq&\liminf_{n\rightarrow\infty}\mathbb{P}\big((\Psi_2\rr{2})_{\mathcal{L}\rr{2}}\inSAT,\;\Phi_1\rr{2}\inSAT,\;\Phi_1\rr{1}\inSAT,\;N_0\rr{2}=0,\;N_0\rr{1}=0\big)\\
        \geq& \liminf_{n\rightarrow\infty}\mathbb{P}\big(\Phi\inSAT\big).
    \end{align*}
    Combining the above implies that the two limit infimum coincide. 
    
\end{proof}
To prove our main theorem it thus suffices to establish Lemma \ref{lemma final lower 2} and \ref{lemma final lower 3}. To do so we need the following technical result.
\begin{lemma} \label{lemma final lower 1}
         There exists a constant $C_0>0$ such that $\mathbb{E}[N_k\rr{r}]\leq C_0n^q$ for $r\in\{1,2\}$ and $k\in\{1,\star\}$. Furthermore, 
         $$\lim_{n\rightarrow\infty}\mathbb{P}\big(N_k\rr{r}\leq \tfrac{1}{2}n^q\log{n}\big)=1,\quad k\in\{1,\star\},\;r\in\{1,2\},$$
         and
         $$\lim_{n\rightarrow\infty}\mathbb{P}\big(N_\star\rr{1}+N_\star\rr{2}\geq n^q\big)=1.$$
    \end{lemma}
\begin{proof}
Note that $p_0(n,\,\lfloor n^q\rfloor)\leq p_1(n,\,\lfloor n^q\rfloor)\leq p_\star(n,\,\lfloor n^q\rfloor)\leq n^{q-1}$, see the definitions in Lemma \ref{lemma technical dist of M}, and thus 
    \begin{equation}
        \mathbb{E}[N_k\rr{1}]=n\cdot p_k(n,\lfloor n^q\rfloor)\leq n^q,\quad(k\in\{0,1,\star\}). \label{eq mean N_k^1}
    \end{equation}
    Using the previous observations, we get 
    $$\mathbb{E}\big[p_1(n-T\rr{0},\,N_1\rr{1})\big]\leq \mathbb{E}\big[p_\star(n-T\rr{0},N_1\rr{1})\big]\leq\frac{n^{q}}{n-n^q-2},$$ 
    and thus if we let $q_1\in(\tfrac{q}{2},\,q)$, then 
    \begin{equation} \label{eq mean N_k^2}
        \mathbb{E}[N_k\rr{2}]=\mathbb{E}\big[N_2\rr{1}\cdot p_k(n-T\rr{0},\,N_1\rr{1})\big]\leq \frac{n^{1+q}}{n-n^q-2}\leq n^q+n^{q_1},\quad (k\in\{1,\star\}),
    \end{equation}
    where we used the upper bound $N_2\rr{1}\leq n$. Note that (\ref{eq mean N_k^1}) and (\ref{eq mean N_k^2}) imply the first claim of the lemma. Furthermore, these two equations along with Markov's inequality give
    \begin{equation*}
        \mathbb{P}\big(N_k\rr{r}>\tfrac{1}{2}n^q\log{n}\big)< \frac{\mathbb{E}[N_k\rr{r}]}{\tfrac{1}{2}n^q\log{n}}\leq  \frac{n^q+n^{q_1}}{\tfrac{1}{2}n^q\log{n}}\rightarrow 0\quad\text{as }n\rightarrow\infty,\quad(k\in\{1,\star\},\;r\in\{1,2\}),
    \end{equation*}
    and this is the second claim of the lemma. Next, using (\ref{eq N_k^r dist}), (\ref{eq mean N_k^1}) and a Chernoff bound we get
    \begin{align*}
        \mathbb{P}\big(N_k\rr{1}\geq \mathbb{E}[N_k\rr{1}]+n^{q_1}\big)&\leq \exp\big(-\tfrac{1}{3}n^{2q_1-q}\big),\quad (k\in\{0,1,\star\}),\\
        \mathbb{P}\big(N_k\rr{1}\leq \mathbb{E}[N_k\rr{1}]-n^{q_1}\big)&\leq \exp\big(-\tfrac{1}{3}n^{2{q_1}-q}\big),\quad (k\in\{0,1,\star\}).
    \end{align*}
    Using this, (\ref{eq mean N_k^1}), and also that
    $\mathbb{E}[N_k\rr{1}]\geq n^q-n^{q_1}$ for $k\in\{1,\star\}$ (this is a direct consequence of the definition)
    we see that
    \begin{equation} \label{eq Chernoff on N_k^1}
        \begin{aligned}
            \mathbb{P}\big(N_k\rr{1}\geq n^q+n^{q_1}\big)&\leq \exp\big(-\tfrac{1}{3}n^{2{q_1}-q}\big),\quad (k\in\{0,1,\star\}),\\
        \mathbb{P}\big(N_1\rr{1}\leq n^q-2n^{q_1}\big)&\leq \exp\big(-\tfrac{1}{3}n^{2q_1-q}\big),\quad (k\in\{1,\star\}).
        \end{aligned}
    \end{equation}
    Next, we again use that for $X\sim \text{Binomial}(n,p)$ we have $\mathbb{E}[X^2]\leq \mathbb{E}[X]+\mathbb{E}[X^2]$, see (\ref{eq second moment of binomial}). Then we get the bound:
    \begin{equation} \label{eq mean of (N_1^r)^2}
        \begin{aligned}
            \mathbb{E}\big[(N_\star\rr{2})^2\big]&\leq \mathbb{E}\Big[\mathbb{E}[N_\star\rr{2}|\mathcal{N}\rr{1}]+\big(\mathbb{E}[N_\star\rr{2}|\mathcal{N}\rr{1}]\big)^2\Big)\\
        &=\mathbb{E}\Big[N_2\rr{1}\cdot p_\star(n-T\rr{0},\;N_1\rr{1})+\big(N_2\rr{1}\cdot p_\star(n-T\rr{0},\;N_1\rr{1})\big)^2\Big]\\
        &\leq  \frac{n}{n-n^q-2}\mathbb{E}[N_1\rr{1}]+\left(\frac{n}{n-n^q-2}\right)^2\mathbb{E}\big[(N_1\rr{1})^2\big]\\
        &\leq \frac{n}{n-n^q-2}n^q+\left(\frac{n}{n-n^q-2}\right)^2\big(\mathbb{E}[N_1\rr{1}]+(\mathbb{E}[N_1\rr{1}])^2\big)\\
        &\leq 2\left(\frac{n}{n-n^q-2}\right)^2n^q+\left(\frac{n}{n-n^q-2}\right)^4n^{2q}\\
        &\leq n^{2q}+Cn^q,
        \end{aligned}
    \end{equation}
    for a constant $C$ chosen large enough. We also want to lower bound the mean. Using that
    $$N_2\rr{2}=n-N_0\rr{2}-N_1\rr{2}-N_\star\rr{2},$$
    and that $$N_\star\rr{2}|\mathcal{N}\rr{1}\sim\text{Binomial}\big(N_2\rr{1},\,p_\star(n-T\rr{0},N_1\rr{1})\big),$$
    where $p_\star(n,l)\geq p_1(n,l)\geq \frac{(n-l)l}{n^2}$, we get that
    \begin{align*}
        &\mathbb{E}\big[ N_\star\rr{2}\;\big|\;N_1\rr{1}>n^q-2n^{q_1}\text{ and }N_k\rr{1}< n^q+n^{q_1}\text{ for }k\in\{0,1,\star\}\big]\\
        \geq & \big(n-3(n^q+n^{q_1})\big)\frac{\big(n-n^q-(n^q+n^{q_1})\big)(n^q-2n^{q_1})}{n^2}\geq n^q-Cn^{q_1},
    \end{align*}
    again for $C$ chosen large enough. This and (\ref{eq Chernoff on N_k^1}) now implies that
    \begin{align*}
        \mathbb{E}[N_\star\rr{2}]\geq& (n^q-Cn^{q_1})\mathbb{P}\big(N_1\rr{1}>n^q-2n^{q_1}\text{ and }N_k\rr{1}< n^q+n^{q_1}\text{ for }k\in\{0,1,\star\}\big)\\
        \geq &(n^q-Cn^{q_1})\Big(1-\mathbb{P}\big(N_1\rr{1}\leq n^q-2n^{q_1}\big)-\sum_{k\in\{0,1,\star\}}\mathbb{P}\big(N_k\rr{1}\geq n^q+n^{q_1}\Big)\\
        \geq &(n^q-Cn^{q_1})\big(1-4\exp\big({-\tfrac{1}{3}n^{2p_1-q}}\big)\big),
    \end{align*}
and by redefining $C$ we get that
\begin{equation} \label{eq mean N_star^2 lower bound}
    \mathbb{E}\big[N_\star\rr{2}\big]\geq n^q-Cn^{q_1}.
\end{equation}
Combining this with (\ref{eq mean of (N_1^r)^2}) we now get 
\begin{align*}
    \mathbb{V}\big(N_\star\rr{2}\big)=\mathbb{E}[(N_\star\rr{2})^2]-\big(\mathbb{E}[N_\star\rr{2}]\big)^2\leq Cn^{q+q_1},
\end{align*}
where $C$ is again redefined. Now, let $q_2\in(\tfrac{q+q_1}{2},q)$. Then
$$\mathbb{P}\big(\big|N_\star\rr{2}-\mathbb{E}[N_\star\rr{2}]\big|>n^{q_2}\big)\leq\frac{\mathbb{V}\big(N_\star\rr{2}\big)}{n^{2q_2}}\leq C\frac{n^{q+q_1}}{n^{2q_2}}\rightarrow 0\quad\text{as }n\rightarrow\infty.$$
    Therefore, using this and (\ref{eq mean N_star^2 lower bound}) we see
    $$\mathbb{P}\big(N_\star\rr{2}\geq n^q-Cn^{q_1}-n^{q_2}\big)\rightarrow 1\quad\text{as }n\rightarrow\infty$$
    and this along with (\ref{eq Chernoff on N_k^1}) gives
    \begin{align*}
    \mathbb{P}\big(N_\star\rr{1}+N_\star\rr{2}>n^q\big)\geq \mathbb{P}\big(N_\star\rr{1}\geq n^q-2n^{q_1},\;N_\star\rr{2}\geq n^q-Cn^{q_1}-n^{q_2}\big)\rightarrow 1\quad\text{as }n\rightarrow\infty,
    \end{align*}
    which finishes the proof.
\end{proof}

Now, we can prove our two remaining lemmas of this section.

\begin{proof}[Proof of Lemma \ref{lemma final lower 2}.]
Using Lemma \ref{lemma technical dist of M} we note that $p_0(n,l)\leq \tfrac{l^2}{4(n-1)^2}$, why
$$\mathbb{P}\big(N_0\rr{1}=0\big)\geq \bigg(1-\frac{n^{2q}}{4(n-1)^2}\bigg)^n\rightarrow 1,\quad\text{as }n\rightarrow\infty,$$
and
\begin{align*}
    \mathbb{P}\big(N_0\rr{2}=0\big)&\geq\mathbb{P}\big(N_0\rr{2}=0\big|\,N_1\rr{1}\leq \tfrac{1}{2}n^q\log{n} \big)\mathbb{P}\big(N_1\rr{1}\leq \tfrac{1}{2}n^q\log{n}\big)\\
    &\geq \bigg(1-\frac{\big(\tfrac{1}{2}n^q\log{n}\big)^2}{4(n-n^q-2)^2}\bigg)^{n}\mathbb{P}\big(N_1\rr{1}\leq \tfrac{1}{2}n^q\log{n}\big)\rightarrow 1\quad\text{as }n\rightarrow\infty,
\end{align*}
where we lastly used Lemma \ref{lemma final lower 1}. Now, using this Lemma again along with Lemma \ref{lemma 1-SAT} we get
\begin{equation*}
\begin{aligned}
    \mathbb{P}\big(\Phi_1\rr{1}\inSAT\big)&\geq \mathbb{P}\big(\Phi_1\rr{1}\inSAT\big|N_1\rr{1}\leq \tfrac{1}{2}n^q\log{n}\big)\mathbb{P}\big(N_1\rr{1}\leq \tfrac{1}{2}n^q\log{n}\big)\\
    &\geq \bigg(1-\frac{\tfrac{1}{2}n^q\log{n}}{n-n^q-2}\bigg)^{\tfrac{1}{2}n^q\log{n}}\mathbb{P}\big(N_1\rr{1}\leq \tfrac{1}{2}n^q\log{n}\big)\rightarrow 1\quad\text{as }n\rightarrow\infty.
\end{aligned}
\end{equation*}
A similar argument gives that
\begin{equation*}
    \begin{aligned}
        \mathbb{P}\big(\Phi_1\rr{2}\inSAT\big)&\geq \mathbb{P}\big(\Phi_1\rr{2}\inSAT\big|N_1\rr{r}\leq \tfrac{1}{2}n^q\log{n},\;r\in\{1,2\}\big)\mathbb{P}\big(N_1\rr{r}\leq \tfrac{1}{2}n^q\log{n},\;r\in\{1,2\}\big)\\
        &\geq \bigg(1-\frac{\tfrac{1}{2}n^q\log{n}}{n-n^q\log{n}}\bigg)^{\frac{1}{2}n^q\log{n}}\mathbb{P}\big(N_1\rr{r}\leq \tfrac{1}{2}n^q\log{n},\;r\in\{1,2\}\big)\rightarrow 1\quad\text{as }n\rightarrow\infty.
    \end{aligned}
\end{equation*}
\end{proof}
\begin{proof}[Proof of Lemma \ref{lemma final lower 3}]
    Remember that 
    $$\mathcal{L}\rr{2}=[n-T\rr{1}]\backslash[n-T\rr{2}]=[n-\lfloor n^q\rfloor -N_1\rr{1}]\backslash[n-\lfloor n^q\rfloor -N_1\rr{1}-N_1\rr{2}],$$ 
    and using that $N_2\rr{2}=n-\sum_{k\in\{0,1,\star\},r\in\{1,2\}}N_k\rr{r}$ we see
    $$\Psi_2\rr{2}|\mathcal{N}\rr{2}\sim F_2(n-T\rr{1},\;N_2\rr{2})=F_2(n-T\rr{1},\;n-T\rr{2}+\lfloor n^q\rfloor -N_0\rr{1}-N_\star\rr{1}-N_0\rr{2}-N_\star\rr{2}).$$
    Let the literals of $\Psi_2\rr{2}$ be given by $(L_{j,1},L_{j,2})$ for $j\in[N_2\rr{2}]$. If $n-T\rr{2}>N_2\rr{2}$ define additional random variables $(L_{j,1},L_{j,2})$ for $j\in[N_2\rr{2}+1,n-T\rr{2}]$, where conditional on $\mathcal{N}\rr{2}$ the pairs of random variables are independent and uniformly distributed in $\{(\ell_1,\ell_2)\in(\pm[n-T\rr{1}])^2:|\ell_1|<|\ell_2|\}$. Define
    $$\Phi'=\bigwedge_{j=1}^{n-T\rr{2}}(L_{j,1}\vee L_{j,2}),$$
    and let also $\mathcal{L}':=\mathcal{L}\rr{2}$. Then
    \begin{align*}
        \mathbb{P}\big((\Psi_2\rr{2})_{\mathcal{L}\rr{2}}\inSAT\big)\geq \mathbb{P}\big(\Phi'_{\mathcal{L}'}\inSAT,\;n-T\rr{2}\geq N_2\rr{2}\big)\geq \mathbb{P}\big(\Phi'_{\mathcal{L}'}\inSAT\big)+\mathbb{P}\big(n-T\rr{2}\geq N_2\rr{2}\big)-1.
    \end{align*}
    Note that
    $$\big\{n-T\rr{2}\geq N_2\rr{2}\big\}\supseteq \big\{N_\star\rr{1}+N_\star\rr{2}\geq n^q\big\},$$
    and thus this along with Lemma \ref{lemma final lower 1} implies that $\lim_{n\rightarrow\infty}\mathbb{P}\big(n-T\rr{2}\geq N_2\rr{2}\big)=1$. Thus 
    \begin{equation} \label{eq liminf first bound}
        \liminf_{n\rightarrow\infty}\mathbb{P}\big((\Psi_2\rr{2})_{\mathcal{L}\rr{2}}\inSAT\big)\geq\liminf_{n\rightarrow\infty}\mathbb{P}\big(\Phi'_{\mathcal{L}'}\inSAT\big).
    \end{equation}
    Define now further
    $$M_1\rr{-1}:=\lfloor n^q\rfloor+N_1\rr{1},\quad M_1\rr{0}:=N_1\rr{2},\quad\mathcal{M}\rr{0}:=\sigma\big(M_1\rr{-1},\,M_1\rr{0}\big).$$
    Note that as $\mathcal{M}\rr{0}\subseteq \mathcal{N}\rr{2}$ and
    $$\Phi'|\mathcal{N}\rr{2}\sim F_2\big(n-T\rr{1},\,n-T\rr{2}\big)=F_2\big(n-M_1\rr{-1},\,n-M_1\rr{-1}+M_1\rr{0}\big),$$
    we get that
    $$\Phi'|\mathcal{M}\rr{0}\sim F_2\big(n-M_1\rr{-1},\,M_1\rr{-1}+M_1\rr{0}\big)$$
    Moreover, the definitions imply that 
    $$\mathcal{L}'=[n-M_1\rr{-1}]\backslash[n-M_1\rr{-1}-M_1\rr{0}].$$
    Lemma \ref{lemma final lower 3} also gives that $\mathbb{E}[M_1\rr{0}]\leq C_0n^q$ for some $C_0>0$ and also that
    $$\lim_{n\rightarrow\infty}\mathbb{P}\big(M_1\rr{-1}\leq n^q\log{n}\big)=\lim_{n\rightarrow\infty}\mathbb{P}\big(M_1\rr{0}\leq n^q\log{n}\big)=1$$
    Thus, our defined elements satisfy all assumptions of Lemma \ref{lemma under-constrained alternative result}. Therefore
    $$\liminf_{n\rightarrow\infty}\mathbb{P}\big(\Phi'_{\mathcal{L}'}\inSAT\big)\geq \liminf_{n\rightarrow\infty}\mathbb{P}\big(\Phi\inSAT\big).$$
    Combining this with (\ref{eq liminf first bound}) establishes the lemma.
\end{proof}

\subsection{Proof of technical lemmas} \label{section proof 5}

We begin by proving Lemma \ref{lemma technical G} and this Lemma is established by a coupling argument where literals are swapped. The swapping will not change the distribution of the resulting formula as the clauses are uniformly distributed.
\begin{proof}[Proof of Lemma \ref{lemma technical G}]
Let $\varphi$ be a non-random SAT-formula with $n$ variables and $m$ clauses and let $(\ell_{j,i})_{j\in[m],i\in[2]}$ denote its literals. Write $\mathcal{L}=\{\ell_1,\dotsc,\ell_m\}$, where $|\ell_1|<\dotsb<|\ell_l|$ and let $\mathcal{L}_{abs}=\{|\ell_j|:j\in[f]\}$. Also let $0<\ell_{f+1}<\dotsc<\ell_{n}\in[n]$ be defined such that $\{|\ell_1|,\dotsc,|\ell_n|\}=[n]$. Define a function $\Gamma:[n]\rightarrow[n]$ by letting $\Gamma(|\ell_i|)=n-i$ for $i\in[n]$. Then $\Gamma$ is a permutation satisfying that $\Gamma(\mathcal{L}_{abs})=\mathcal{L}'$, where $\mathcal{L}'=[n]\backslash[n-f]$. Define another function $\theta:[n]\rightarrow\{\pm1\}$ where $\theta(|\ell_i|)=\sign(\ell_i)$. Define a new SAT-formula $\varphi'$ with literals $(\ell_{j,i}')_{j\in[m],i\in[2]}$, where
$$\{\ell_{j,1}',\ell_{j,2}'\big\}=\big\{\theta(|\ell_{j,i}|)\cdot \sign(\ell_{j,i})\cdot \Gamma(|\ell_{j,i}|):i\in[2]\big\}.$$
 Then we define $G(\varphi,\mathcal{L}):=\varphi'$. Let $x=(x_1,\dotsc,x_n)\in\mathbb{B}^{n}$ and define $x'=(x_1',\dotsc,x_n')\in\mathbb{B}^n$ by letting $x'_v=x_{\Gamma^{-1}(v)}$ for $v\in[n]$. Note that $x\mapsto x'$ is a bijection. Let for $j\in[m]$, $i\in[2]$ be chosen such that $\ell'_{j,1}=\theta(|\ell_{j,i}|)\cdot \text{sign}(\ell_{j,i})\cdot \Gamma(|\ell_{j,i}|)$. 
 
 Now, if $\ell_{j,i}\in \mathcal{L}$ then $(\ell_{j,i})_\mathcal{L}(x)=\texttt{true}$. Also, there exists a $v\in[f]$ such that $\ell_{j,i}=\ell_{v}$ and also $\Gamma(|\ell_{v}|)\in \mathcal{L}'$. Thus
 $$(l_{j,1}')_{\mathcal{L}'}(x')=\big[\theta(|\ell_{v}|)\cdot \sign(\ell_{v})\cdot \Gamma(|\ell_{v}|)\big]_{\mathcal{L}'}(x')=\Gamma(|\ell_{v}|)_{\mathcal{L}'}(x')=\texttt{true}.$$
 If $-\ell_{j,i}\in \mathcal{L}$ then $(\ell_{j,i})_\mathcal{L}(x)=\texttt{false}$. Again there exists $v\in[l]$ such that $\ell_{j,i}=-\ell_v$ and also $\Gamma(|\ell_v|)\in \mathcal{L}'$ and thus
$$(l_{j,1}')_{\mathcal{L}'}(x')=\big[\theta(|\ell_{v}|)\cdot \sign(-\ell_{v})\cdot \Gamma(|\ell_{v}|)\big]_{\mathcal{L}'}(x')=-\Gamma(|\ell_{v}|)_{\mathcal{L}'}(x')=\texttt{false}.$$
Lastly, if $\pm \ell_{j,i}\notin \mathcal{L}$ then $(\ell_{j,i})_\mathcal{L}(x)=\ell_{j,i}(x)$ and also $\pm\Gamma(|\ell_{j,i}|)\notin \mathcal{L}'$. Therefore
$$(l_{j,1}')_{\mathcal{L}'}(x')=\big[\sign(\ell_{j,i})\cdot \Gamma(|\ell_{j,i}|)\big](x')=\sign(\ell_{j,i})\cdot x'_{\Gamma(|\ell_{j,i}|)}=\sign(\ell_{j,i})\cdot x_{|\ell_{j,i}|}=\ell_{j,i}(x).$$
Repeating the argument on $\ell_{j,2}'$ implies that
$$(l_{j,1}\vee l_{j,2})_\mathcal{L}(x)=(\ell{j,1}'\vee l_{j,2}')_{\mathcal{L}'}(x'),$$
and thus $\varphi_\mathcal{L}\inSAT$ if and only if $\varphi'_{\mathcal{L}'}\inSAT$. 

Let $\Phi\sim F_2(n,m)$ and define $\Phi'=G(\Phi,L)$. Then the above argument implies that 
$$\big\{G(\Phi,L)\inSAT\big\}=\big\{\Phi'_{[n]\backslash[n-l]}\inSAT\big\}.$$
Let $(L_{j,i})_{j\in[m],i\in[2]}$ be the random literals of $\Phi$ and let $(L_{j,i}')_{j\in[m],i\in[2]}$ be the random literals of $\Phi'$. Note that as the clause $\big(L_{j,1}',L_{j,2}'\big)$ is constructed from $\big(L_{j,1},L_{j,2}\big)$ for each $j\in[m]$ the clauses of $\Phi'$ are independent. Let $j\in[m]$ and assume WLOG that $\Gamma(|L_{j,1}|)<\Gamma(|L_{j,2}|)$. Then, for $(\ell_1',\ell_2')\in(\pm[n])^2$ with $|\ell_1'|<|\ell_2'|$ we have
\begin{align*}
    \mathbb{P}\big(L_{j,i}'=\ell_i',\;i\in[2]\big)&=\mathbb{P}\Big(\theta(|L_{j,i}|)\cdot\sign(L_{j,i})\cdot\Gamma(|L_{j,i}|)=\ell_i',\;i\in[2]\Big)\\
    &=\mathbb{P}\Big(|L_{j,i}|=\Gamma^{-1}(|\ell_i'|),\;\sign(L_{j,i})=\theta\big(\Gamma^{-1}(|\ell_i'|)\big)\cdot\sign(\ell_i'),\;i\in[2]\Big)\\
    &=\mathbb{P}\big(L_{j,i}=\ell_i,\;i\in[2]\big),
\end{align*}
where $\ell_i=\theta\big(\Gamma^{-1}(\ell_i')\big)\cdot\sign(\ell_i')\cdot\Gamma^{-1}(|\ell_i'|)$ for $i\in[2]$ and as the clause $(L_{j,1},L_{j,2})$ is uniformly distributed, the result follows.  
\end{proof}
More or less direct calculations imply the next lemma. We recall that equation (\ref{eq technical A sets}) defines the sets
\begin{equation*} 
    \begin{aligned}
        \mathcal{A}_0(n,f):=-\mathcal{L}\times -\mathcal{L},\quad& \mathcal{A}_1(n,f):=\pm[n-f]\times -\mathcal{L},\\
    \mathcal{A}_2(n,f):=\pm[n-f]\times \pm[n-f],\quad& \mathcal{A}_\star(n,f):=\pm[n]\times \mathcal{L}.
    \end{aligned}
\end{equation*}
\begin{proof}[Proof of Lemma \ref{lemma technical dist of M}]
Let $\Phi\sim F_2(n,m)$ and let ${L}=(L_{j,i})_{j\in[m],i\in[2]}$ be its literals. As the clauses are i.i.d. and
$$M_k=\big|\big\{j\in [m]:(L_{j,1},L_{j,2})\in\mathcal{A}_k(n,f)\big\}\big|,\quad (k\in K),$$
where each clause belongs to exactly one of the sets $\mathcal{A}_k(n,f)$ for $k\in K$ this implies that 
$${M}:=(M_0,M_1,M_2,M_\star)\sim\text{Multinomial}\big(m,\;{p}(n,f)\big),$$
where ${p}=(p_0,p_1,p_2,p_\star)$ and $p_k(n,f)=\mathbb{P}\big((L_{1,1},L_{1,2})\in\mathcal{A}_k(n,l)\big)$ for $k\in K$. As the clauses are uniformly distributed on $\mathcal{D}=\{(\ell_1,\ell_2)\in(\pm[n])^2:|\ell_1|<|\ell_2|\}$ we further get that $p_k(n,l)=\frac{|\mathcal{A}_k(n,f)\cap\mathcal{D}|}{|\mathcal{D}|}$ for $k\in K$, so
\begin{align*}
    p_0(n,f)&=\frac{\binom{f}{2}}{2^2\binom{n}{2}}=\frac{f(f-1)}{4n(n-1)},\\
    p_1(n,f)&=\frac{2(n-f)f}{2^2\binom{n}{2}}=\frac{(n-f)f}{n(n-1)},\\
    p_2(n,f)&=\frac{2^2\binom{n-f}{2}}{2^2\binom{n}{2}}=\frac{(n-f)(n-f-1)}{n(n-1)},\\
    p_\star(n,f)&=1-p_0(n,f)-p_1(n,f)-p_2(n,f)=\frac{f(n-\tfrac{f}{4}-\tfrac{3}{4}))}{n(n-1)},
\end{align*}
We will need the following result to establish the last part of the lemma. For $X,Y$ independent random functions, sets $A,B$, and elements $x,y$ with $x\in A$ and $y\in B$ we have
$$\mathbb{P}\big(X=x,Y=y|X\in A,Y\in B\big)=\mathbb{P}\big(X=x|X\in A\big)\mathbb{P}\big(Y=y|Y\in B\big).$$
Define $\mathcal{C}_k=\big\{j\in[m]:(L_{j,1},L_{j,2})\in \mathcal{A}_k(n,f)\big\}$ for $k\in K$ and let ${\mathcal{C}}=(\mathcal{C}_k)_{k\in K}$. For elements $\ell_{j,1}\in\pm[n-f]$ for $j\in\mathcal{C}_1$ and $(\ell_{j,1},\ell_{j,2})\in\mathcal{A}_2(n,f)$ for $j\in\mathcal{C}_2$ we use the independence of the clauses and the above equation and get
\begin{equation} \label{eq proof Phi1 Phi2 (1)}
\begin{aligned}
    \mathbb{P}\Big(\Phi_1=\bigwedge_{j\in\mathcal{C}_1}(\ell_{j,1})\Big|\mathcal{C}\Big)&=\prod_{j\in\mathcal{C}_1}\mathbb{P}\big(L_{j,1}=\ell_{j,1}\big|j\in\mathcal{C}_1\big),\\
    \mathbb{P}\Big(\Phi_2=\bigwedge_{j\in\mathcal{C}_2}(\ell_{j,1}\vee \ell_{j,2})\Big|\mathcal{C}\Big)&=\prod_{j\in\mathcal{C}_2}\mathbb{P}\big((L_{j,1},L_{j,2})=(\ell_{j,1},\ell_{j,2})\big|j\in\mathcal{C}_2\big),
\end{aligned}
\end{equation}
and
\begin{equation} \label{eq proof Phi1 Phi2 (2)}
\begin{aligned}
    &\mathbb{P}\Big(\Phi_1=\bigwedge_{j\in\mathcal{C}_1}(\ell_{j,1}),\;\Phi_2=\bigwedge_{j\in\mathcal{C}_2}(\ell_{j,1}\vee \ell_{j,2})\Big|\mathcal{C}\Big)\\
    =&\prod_{j\in\mathcal{C}_1}\mathbb{P}\big(L_{j,1}=\ell_{j,1}|j\in\mathcal{C}_1\big)\prod_{j\in\mathcal{C}_2}\mathbb{P}\big((L_{j,1},L_{j,2})=(\ell_{j,1},\ell_{j,2})|j\in\mathcal{C}_2\big).
\end{aligned}
\end{equation}
Now, using that the clauses are uniformly distributed on $\mathcal{D}$ along with the definitions of the sets $\mathcal{A}_k(n,f)$, $k\in\{1,2\}$ we first get for $j\in\mathcal{C}_1$
\begin{align*}
    \mathbb{P}\big(L_{j,1}=\ell_{j,1}\big|j\in\mathcal{C}_1\big)&=\sum_{\ell_{j,2}\in-\mathcal{L}}\mathbb{P}\big((L_{j,1},L_{j,2})=(\ell_{j,1},\ell_{j,2})\big|j\in\mathcal{C}_1\big)\\
    &=\sum_{\ell_{j,2}\in-\mathcal{L}}\frac{\mathbb{P}\big((L_{j,1},L_{j,2})=(\ell_{j,1},\ell_{j,2})\big)}{\mathbb{P}\big((L_{j,1},L_{j,2})\in\mathcal{A}_1(n,f)\big)}=l\cdot \frac{\frac{1}{2^2\binom{n}{2}}}{p_1(n,f)}=\frac{1}{2\binom{n-f}{1}},
\end{align*}
and next for $j\in\mathcal{C}_2$
\begin{align*}
    \mathbb{P}\big((L_{j,1},L_{j,2})=(\ell_{j,1},\ell_{j,2})\big|j\in\mathcal{C}_2\big)=\frac{\mathbb{P}\big((L_{j,1},L_{j,2})=(\ell_{j,1},\ell_{j,2})\big)}{\mathbb{P}\big((L_{j,1},L_{j,2})\in\mathcal{A}_2(n,l)\big)}=\frac{\frac{1}{2^2\binom{n}{2}}}{p_2(n,f)}=\frac{1}{2^2\binom{n-f}{2}}.
\end{align*}
Inserting this in (\ref{eq proof Phi1 Phi2 (1)}) gives
\begin{align*}
    \mathbb{P}\Big(\Phi_1=\bigwedge_{j\in\mathcal{C}_1}(\ell_{j,1})\Big|\mathcal{C}\Big)=\bigg(\frac{1}{2\binom{n-f}{1}}\bigg)^{M_1}\quad\text{and}\quad \mathbb{P}\Big(\Phi_2=\bigwedge_{j\in\mathcal{C}_2}(\ell_{j,1}\vee \ell_{j,2})\Big| \mathcal{C}\Big)=\bigg(\frac{1}{2^2\binom{n-f}{2}}\bigg)^{M_2}.
\end{align*}
Thus for $\varphi_1$ a $1$-SAT formula with $n-l$ variables and $M_1$ clauses and $\varphi_2$ a $2$-SAT formula with $n-l$ variables and $M_2$ clauses we get
\begin{align*}
    \mathbb{P}\big(\Phi_1=\varphi_1\big|M\big)&=\mathbb{E}\big[\mathbb{P}\big(\Phi_1=\varphi_1\big|\mathcal{C}\big)\big|M\big]=\bigg(\frac{1}{2\binom{n-l}{1}}\bigg)^{M_1},\\
    \mathbb{P}\big(\Phi_2=\varphi_2\big|M\big)&=\mathbb{E}\big[\mathbb{P}\big(\Phi_2=\varphi_2\big|\mathcal{C}\big)\big|M\big]=\bigg(\frac{1}{2^2\binom{n-l}{2}}\bigg)^{M_2},
\end{align*}
and this corresponds to having $\Phi_k|M\sim F_k(n-l,M_k)$ for $k\in \{1,2\}$. Repeating this argument with equation  (\ref{eq proof Phi1 Phi2 (2)}) gives that
$$\mathbb{P}\big(\Phi_1=\varphi_1,\;\Phi_2=\varphi_2\big|M\big)=\mathbb{P}\big(\Phi_1=\varphi_1\big|M\big)\mathbb{P}\big(\Phi_2=\varphi_2\big|M\big)$$
which implies the conditional independence. 
\end{proof}

\bibliographystyle{alpha}
\bibliography{bib.bib}

\newcommand{\etalchar}[1]{$^{#1}$}
\begin{thebibliography}{ACOHK{\etalchar{+}}21}

\bibitem[Ach00]{Achlioptas00}
Dimitris Achlioptas.
\newblock Setting 2 variables at a time yields a new lower bound for random 3-{SAT}.
\newblock In {\em Proceedings of the thirty-second annual ACM symposium on Theory of computing}, pages 28--37, 2000.

\bibitem[ACIM01]{Achlioptas01}
Dimitris Achlioptas, Arthur Chtcherba, Gabriel Istrate, and Cristopher Moore.
\newblock The phase transition in 1-in-k {SAT} and {NAE} 3-{SAT}.
\newblock In {\em Proceedings of the twelfth annual ACM-SIAM symposium on Discrete algorithms}, pages 721--722, 2001.

\bibitem[ACOHK{\etalchar{+}}21]{Noela2021}
Dimitris Achlioptas, Amin Coja-Oghlan, Max Hahn-Klimroth, Joon Lee, No{\"e}la M{\"u}ller, Manuel Penschuck, and Guangyan Zhou.
\newblock The number of satisfying assignments of random 2-{SAT} formulas.
\newblock {\em Random structures \& algorithms}, 58(4):609--647, 2021.

\bibitem[AKKK01]{Achlioptas01_1}
Dimitris Achlioptas, Lefteris~M Kirousis, Evangelos Kranakis, and Danny Krizanc.
\newblock Rigorous results for random (2+ p)-{SAT}.
\newblock {\em Theoretical Computer Science}, 265(1-2):109--129, 2001.

\bibitem[ANN04]{Athreya04}
Krishna~B Athreya, Peter~E Ney, and PE~Ney.
\newblock {\em Branching processes}.
\newblock Courier Corporation, 2004.

\bibitem[BBC{\etalchar{+}}01]{Bollobas_2001}
B{\'e}la Bollob{\'a}s, Christian Borgs, Jennifer~T Chayes, Jeong~Han Kim, and David~B Wilson.
\newblock The scaling window of the 2-{SAT} transition.
\newblock {\em Random Structures \& Algorithms}, 18(3):201--256, 2001.

\bibitem[BCM02]{Biroli02}
Giulio Biroli, Simona Cocco, and R{\'e}mi Monasson.
\newblock Phase transitions and complexity in computer science: an overview of the statistical physics approach to the random satisfiability problem.
\newblock {\em Physica A: Statistical Mechanics and its Applications}, 306:381--394, 2002.

\bibitem[BOOS25]{Basse_2024}
Andreas Basse-O'Connor, Tobias~Lindhardt Overgaard, and Mette Skjøtt.
\newblock On the regularity of random 2-{SAT} and 3-{SAT}.
\newblock {\em arXiv:2504.11979 [math.PR]}, 2025.

\bibitem[CCOM{\etalchar{+}}24]{Doja-Ochlan2024}
Arnab Chatterjee, Amin Coja-Oghlan, Noela M{\"u}ller, Connor Riddlesden, Maurice Rolvien, Pavel Zakharov, and Haodong Zhu.
\newblock The number of random 2-{SAT} solutions is asymptotically log-normal.
\newblock {\em arXiv preprint arXiv:2405.03302}, 2024.

\bibitem[CF86]{Chao86}
Ming-Te Chao and John Franco.
\newblock Probabilistic analysis of two heuristics for the 3-satisfiability problem.
\newblock {\em SIAM Journal on Computing}, 15(4):1106--1118, 1986.

\bibitem[CKT{\etalchar{+}}91]{Cheeseman91}
Peter~C Cheeseman, Bob Kanefsky, William~M Taylor, et~al.
\newblock Where the really hard problems are.
\newblock In {\em Ijcai}, volume~91, pages 331--337, 1991.

\bibitem[Coo71]{Cook71}
Stephen~A Cook.
\newblock The complexity of theorem-proving procedures.
\newblock In {\em Proceedings of the Third Annual ACM Symposium on Theory of Computing}, STOC '71, page 151–158. Association for Computing Machinery, New York, NY, USA, 1971.

\bibitem[CR92]{Chvatal_1992}
Va{\v{s}}ek Chv{\'a}tal and Bruce Reed.
\newblock Mick gets some (the odds are on his side)(satisfiability).
\newblock In {\em Proceedings of the forty-seventh annual ACM symposium on Theory of computing}, pages 620--627. IEEE Computer Society, 1992.

\bibitem[DLL62]{Davis62}
Martin Davis, George Logemann, and Donald Loveland.
\newblock A machine program for theorem-proving.
\newblock {\em Communications of the ACM}, 5(7):394--397, 1962.

\bibitem[dLV01]{Fernandez_2001}
W~Fernandez de~La~Vega.
\newblock Random 2-{SAT}: results and problems.
\newblock {\em Theoretical computer science}, 265(1-2):131--146, 2001.

\bibitem[DM11]{Dellacherie11}
Claude Dellacherie and P-A Meyer.
\newblock {\em Probabilities and potential, c: potential theory for discrete and continuous semigroups}.
\newblock Elsevier, 2011.

\bibitem[DSS15]{SSZ21}
Jian Ding, Allan Sly, and Nike Sun.
\newblock Proof of the satisfiability conjecture for large k.
\newblock In {\em Proceedings of the forty-seventh annual ACM symposium on Theory of computing}, pages 59--68, 2015.

\bibitem[GGW06]{Marco_2006}
Aarti Gupta, Malay~K Ganai, and Chao Wang.
\newblock {SAT}-based verification methods and applications in hardware verification.
\newblock In {\em International School on Formal Methods for the Design of Computer, Communication and Software Systems}, pages 108--143. Springer, 2006.

\bibitem[Goe96]{Goerdt_1996}
Andreas Goerdt.
\newblock A threshold for unsatisfiability.
\newblock {\em Journal of Computer and System Sciences}, 53(3):469--486, 1996.

\bibitem[Gol79]{Goldberg79}
Allen~T Goldberg.
\newblock {\em On the complexity of the satisfiability problem}.
\newblock New York University, 1979.

\bibitem[GW94]{Gent94}
Ian~P Gent and Toby Walsh.
\newblock The {SAT} phase transition.
\newblock In {\em ECAI}, volume~94, pages 105--109. PITMAN, 1994.

\bibitem[Knu15]{Knuth_2015}
Donald~E Knuth.
\newblock {\em The art of computer programming, Volume 4, Fascicle 6: Satisfiability}.
\newblock Addison-Wesley Professional, 2015.

\bibitem[KS94]{Kirkpatrick94}
Scott Kirkpatrick and Bart Selman.
\newblock Critical behavior in the satisfiability of random boolean expressions.
\newblock {\em Science}, 264(5163):1297--1301, 1994.

\bibitem[MS08]{Marques_2008}
Joao Marques-Silva.
\newblock Practical applications of boolean satisfiability.
\newblock In {\em 2008 9th International Workshop on Discrete Event Systems}, pages 74--80. IEEE, 2008.

\bibitem[Pap03]{papadimitriou94}
Christos~H Papadimitriou.
\newblock Computational complexity.
\newblock In {\em Encyclopedia of computer science}, pages 260--265. John Wiley and Sons Ltd., 2003.

\bibitem[SML96]{Selmanetal1996}
Bart Selman, David~G Mitchell, and Hector~J Levesque.
\newblock Generating hard satisfiability problems.
\newblock {\em Artificial intelligence}, 81(1-2):17--29, 1996.

\end{thebibliography}

\end{document}